\documentclass[11pt]{article}
\usepackage{latexsym, amsfonts, amscd, amssymb, verbatim, amsmath,
enumerate}

\newtheorem{theorem}{Theorem}[section]
\newtheorem{proposition}{Proposition}[section]
\newtheorem{definition}{Definition}[section]

\newtheorem{lemma}{Lemma}[section]
\newtheorem{remark}{Remark}[section]

\newtheorem{example}{Example}[section]

\newcommand{\proofbox}{\hspace{\fill}{$\Box$}}
\newenvironment{proof}{\textbf{Proof}.}{\proofbox}

\def\emptyset{\mbox{{\rm \O}}}
\def\bar{\overline}
\date{}

\numberwithin{equation}{section}

\begin{document}

\title{Second-order Optimality Conditions for Time-Optimal Control Problems Governed by Semilinear Parabolic Equations} 
	
\author{H. Khanh\footnote{Department of Mathematics, FPT University, Hoa Lac Hi-Tech Park, Hanoi, Vietnam   and Department of Optimization and Control Theory,	Institute of Mathematics, Vietnam Academy of Science and Technology,	18 Hoang Quoc Viet road, Hanoi, Vietnam; email: khanhh@fe.edu.vn},    B.T. Kien\footnote{Department of Optimization and Control Theory, Institute of Mathematics, Vietnam Academy of Science and Technology,18 Hoang Quoc Viet road, Hanoi, Vietnam; email: btkien@math.ac.vn}  \   and   A. R\"{o}sch\footnote{Faculty of Mathematics, University of Duisbug-Essen, Thea-Leymann-Strasse 9, 45127 Essen,
Germany; email: arnd.roesch@uni-due.de}}
	
\maketitle
	
\medskip
	
\noindent {\bf Abstract.} {\small A class of time-optimal  control problems governed by semilinear parabolic equations with mixed pointwise constraints and final point constraints is considered. By introducing the so-called locally optimal solution to time-optimal  control problems, we establish first and second-order necessary optimality conditions of KKT-type and second-order sufficient conditions for locally optimal solutions to the problem.
	
\medskip
	
	
	
\vspace{0.4cm}

\section{Introduction}
	
Let $\Omega$ be a bounded domain in $\mathbb{R}^N$ with $N=2$ or $3$ and its  boundary $\Gamma=\partial \Omega$ is of class $C^2$. Let $D=H^2(\Omega)\cap H_0^1(\Omega)$, $V=H_0^1(\Omega)$,  $H=L^2(\Omega)$ and $Q_T:=\Omega\times (0, T)$. We consider  the problem of finding $T>0$,  control function $u\in L^\infty(Q_T)$  and the corresponding state $y\in C(\bar Q_T)\cap W^{1, 1}_2(0, T; D, H)$ which solve
\begin{align}
    &J(T,y, u):= \psi_0(T, y(T))+ \int_0^T\int_\Omega L(x, t,  y(x, t), u(x, t)) dxdt\to\inf \label{P1}\\
    &{\rm s.t.}\notag\\
 (P)\quad\quad\quad   &\frac{\partial y}{\partial t} +Ay + \psi(x, t, y) =u\quad \text{in}\quad Q_T,\ y(x, t)=0\ \text{on}\ \Sigma_T = \Gamma\times [0, T],\label{P2} \\
    & y(0)=y_0 \quad \text{in}\quad \Omega, \label{P3}\\
    &\psi_i(T, y(T))\le 0,  \quad   i=1,2,..., m,  \label{P4}\\
    & a\leq g(x,t,  y(x, t), u(x, t))\leq b \quad \text{a.a.}\quad (x, t)\in Q_T,  \label{P5}
\end{align} where $y_0\in H^1_0(\Omega)\cap C(\bar\Omega)$, $a<b$,  $A$ denotes a second-order elliptic operator of the form
$$
Ay =  - \sum_{i,j = 1}^N D_j\left(a_{ij}(x) D_i y \right),
$$  $L:Q \times {\mathbb R} \times {\mathbb R} \to {\mathbb R}$,  $\psi: \Omega\times[0, +\infty)\times{\mathbb R} \to {\mathbb R}$, $\psi_i: [0, +\infty)\times C(\bar \Omega)\to \mathbb{R}$ with $i=0,1,..., m$ and $g: \Omega\times[0, +\infty)\times \mathbb{R}\times\mathbb{R} \to \mathbb{R}$ are of class $C^2$.  Throughout of the paper, we denote by    $\Phi$ the  feasible set of problem $(P)$, that is, $\Phi$ consists of triples $(T, y, u) \in (0, +\infty)  \times \big(C(\bar Q_T)\cap W^{1, 1}_2(0, T; D, H)\big) \times L^\infty(Q_T)$ which satisfies conditions \eqref{P2}-\eqref{P5}. The goal of problem $(P)$ is to find a control function $u$ which plays a role as heat source to control the temperature of the system in the shortest  time $T>0$.

The class of time-optimal control problems is important for modern technology. They  have   several applications in different areas. In particular, there are many problems in the field of  Aerospace and Rocketry can be solved by finding solutions  of time-optimal control problems such as: Soft landing on the Moon (see \cite{Cesari}), Optimal ascent trajectories and Orbital transfer problems (see \cite{Longuski}), Launch vehicle guidance and trajectory optimization (see \cite{Pesce} and \cite{Y}). Recently, scientists of SpaceX have used time-optimal control for guiding reusable launch vehicles, such as the SpaceX Falcon 9, back to Earth for a soft landing. Here, the goal is to minimize the time for descent and landing while conserving fuel and ensuring a safe touchdown.

In order to find solution to time-optimal control problems, we often derive necessary optimality conditions and solve the system of optimality conditions by computer to obtain approximate solutions. Then we prove that the approximate solutions converge to the exact solution and  give error estimate betwen the exact solution and approximate solutions. This can be done by using  second-order sufficient optimality conditions. Besides, we can use the second-order sufficient conditions  to check whether or not an extremal  point is an optimal solution to the problem. Therefore,  optimality conditions for optimal control problems are significant. In this paper we aim at deriving first and second-order necessary optimality conditions and second-order sufficient optimality conditions  for problem $(P)$. 

Although, there have been many papers dealing with first and second-order optimality conditions for optimal control problems with fixed final time (see for instance \cite{Arada-2000}-\cite{Arada-2002-2}, \cite{Bayen}, \cite{Casas-2000}-\cite{Casas-2023}, \cite{Hu}, \cite{Khanh.Kien-2024} and \cite{Troltzsch}),  few of papers study second-order optimality conditions for time-optimal control problems governed by PDEs or ODEs. This is because the structure of a time-optimal control problem is quite complicated. It is never convex even the constraints are of linear forms; the function spaces vary in  time variable $T$. These causes difficulties for the existence of optimal solution and how to define the so-called locally optimal solution to the problem.   In \cite{Maurer-2002}, H. Maurer and H.J. Oberle gave sufficient optimality conditions for time optimal control problems governed by ODE with mixed constraints while  H. Maurer and N.P. Osmolovskii \cite{Maurer-2004} studied second order optimal conditions for time-optimal  control which is linear in control variable. In \cite{Raymond} J.-P. Raymond and  H. Zidani establised a Pontryagin's principle (first-order optimality conditions) for time-optimal control problems governed by semilinear
parabolic equations with pointwise state constraints and unbounded controls.  Also,   N. Arada and J.-P. Raymond \cite {Arada-2003} derived first-order optimality conditions for time- optimal problems with Dirichlet boundary controls. Recently, in \cite{Bon},  L. Bonifacius et al. have considered a time-optimal control problem governed by linear parabolic equation with quadratic cost function and pure control constraint. They established necessary and sufficient second-order optimality conditions and gave error estimates between the exact solution of the original problem and approximate solutions of discretization problems which are discretized by the space-time finite element method.

It is noted that, to derive optimality conditions for optimal solution $(T_*, y_*, u_*)$ to time-optimal control problems, in the previous papers,  the authors assumed that  it was a strongly local solution, that is, $(T_*, y_*, u_*)$ is an optimal solution to the problem in a neighborhood of $(T_*,y_*)$  (see Definition 3.3 in \cite{Bayen}). To our best knowledge, so far there have been no paper dealing with first and second-order optimality conditions for a locally optimal solutions to time-optimal control problems governed by seminlinear parabolic equations with mixed constraints. Therefore, our obtained results in this paper are novel.

In the present setting of problem $(P)$,  we will face with some challenges: the regularity of solution to the state equation,   how to define a locally optimal solution; how to derive optimality conditions in both state variable and control variable and how to transform the problem to a new problem with fixed final time.  To  overcome these challenges, we first establish  the existence and regularity of solution to the state equation. Then we  give a suitable definition for locally optimal solution by extending both sate function and control function. Since implicit function theorem is fail to apply for $(P)$, we need to derive optimality conditions in both variable $y$ and $u$. For this we will derive optimality conditions for a specific mathematical programming problem in Banach space  and apply the obtained result for $(P)$. In order to transform $(P)$ to an optimal control  problem with fixed final time,  we use the simplest way by  changing variable $t=Ts$ with $s\in [0,1]$ to reduce $(P)$  to a fixed final time problem $(P_1)$, where $(T, v)$ plays a role as a new control variable. It is worth pointing out that a new technique is given in this paper to establish second-order sufficient conditions for $(P_1)$, which deals with  only control variable $v$. 

The remainder of the paper is organized as follows. In Section 2, we establish some related results on the regularity of solution to the state equation and derive optimality conditions for a specific mathematical programming in Banach spaces. The statement of main results are given in Section 3. In section 4, we derive first and second-order optimality conditions for a reduced problem with fixed final time. Section 5 is destined for  proofs of main results.

\section{Related results}

\subsection{Mathematical programming problem}

Let  $E$,  $W$ and $Z$ be Banach spaces with the dual spaces  $E^*$,  $W^*$ and $Z^*$, respectively. We denote by $\tau(\mathbb{R}^m)$ the strong topology in $\mathbb{R}^m$, by $\sigma(E^*, E)$ the weak$^*$ topology in $E^*$ and by $\sigma(W^*, W)$ the weak$^*$ topology in $W^*$.  Given a Banach space $X$,  $B_X(x_0, r)$ stands for the open ball with center $x_0$ and radius $r$. Given a subset  $M$ in $X$. We shall denote by ${\rm Int}M$ and $\overline{M}$ the interior and the closure of $M$, respectively. 
	
Let $f: Z\to\mathbb{R}$,   $F\colon Z\to W$,  $G: Z\to E$ and $H_i:Z\to \mathbb{R}^n$ with $i=1,2,..., m$  be mappings  and $K$ be a nonempty  convex set in $E$. 	 We consider the following  optimization problem. 
\begin{align*}
\text{(MP)}\quad\quad
\begin{cases}
		&f(z)\to \min\\
		&\text{s.t.}\\
           &F(z)=0,\ H_i(z)\leq 0,\  G(z)\in K,\  i=1,2,...,m.
  \end{cases}
\end{align*}
We shall denote by $\Sigma$ the feasible set of  (MP), that is,
$$
\Sigma=\{z\in Z\;|\; F(z)=0, H_i(z)\leq 0, G(z)\in K, i=1,2,..., m\}.
$$ Let us recall some concepts  of variational analysis.	Given a  nonempty subset $S$   of $X$ and $\bar x\in \bar S$,  the set
	\begin{align*}
	T(S, \bar x)&:=\left\{h\in X\;|\; \exists t_n\to 0^+, \exists h_n\to h, \bar x+t_nh_n\in S \ \ \forall n\in\mathbb{N}\right\},
	\end{align*}
	is called {\em the contingent cone} to $M$ at $\bar x$. It is well-known that when $M$ is convex, then
 $$
 T(S, \bar x)=\overline{{\rm cone}(S-\bar x)},
 $$	where 
$$
\mathrm{cone}\,(S-\bar x)=\{\lambda (v-\bar x)\;|\; v\in S, \lambda>0\}.
$$ When $S$ is a convex set, the normal cone to $S$ at $\bar x$ is defined by
\begin{equation*}
N(S, \bar x):=\{x^*\in X^*\;|\; \langle x^*, x-\bar x\rangle\leq 0 \ \ \forall x\in S\}=
		\{x^*\in X^*\;|\; \langle x^*, h\rangle\leq 0\ \ \forall h\in T(S, \bar x)\}.
\end{equation*}
Given a feasible point $z_0 \in \Sigma$,  we denote by $f'(z_0)$ or $Df(z_0)$  the first order derivative of $f$ at $z_0$ and by $f''(z_0)$ or $D^2 f(z_0)$ the second-order derivative of $f$ at $z_0$. Let us impose the following assumptions.

\begin{enumerate}
	\item[$(A1)$] ${\rm int} K\neq \emptyset$;
	
	\item[$(A2)$] The mappings $f, F$ and $G$ are of class $C^2$ around $z_0$;
	
	\item[$(A3)$] The range of $DF(z_0)$ is closed in $W$;

        \item[$(A4)$]  $DF(z_0)$ is onto and there exists $\widetilde z\in Z$ such that
        \begin{align}\label{Man-From}
        \begin{cases}
		&DF(z_0)\widetilde z=0,\\
            & H_i(z_0)+  DH_i(z_0)\widetilde z <0\ \forall i=1,2,..., m\\
		&G(z_0) +  DG(z_0) \widetilde z  \in {\rm int}(K).
        \end{cases}
	\end{align}
\end{enumerate}
It is  noted that $(A4)$ is a type  of the Mangasarian–Fromowitz condition. Obviously, $(A4)$ implies $(A3)$. 

To deal with  second-order optimality conditions, we need the so-called critical cone which is defined as follows. Let ${\mathcal C}_0[z_0]$ be a set of vectors $d \in Z$ satisfying conditions:
\begin{itemize}
	\item [$(a_1)$] $\nabla f(z_0)d \leq 0, $
	
	\item [$(a_2)$] $DF(z_0)d = 0,$
	
	\item [$(a_3)$] $DH_i(z_0)d\leq 0$ for $i\in I_0$, where $I_0:=\{i\in\{1,2,..., m\}: H_i(z_0)=0 \}$, 
	
	\item [$(a_4)$] $DG(z_0)d  \in {\rm cone}(K-G(z_0))$.
\end{itemize}  Then  the closure of ${\mathcal C}_0[z_0]$ in $Z$ is called the critical cone at $z_0$ and denoted by ${\mathcal C}[z_0]$. Each vector $d\in \mathcal{C}[z_0]$ is called a critical direction. 

 Let us assume that 
\begin{align*}	{\mathcal L}\left( z, \lambda, w^*, l, e^* \right) = \lambda f(z) + \langle w^*, F(z)\rangle + \sum_{i=1}^ml_iH_i(z)+ \langle e^*, G(z)\rangle 
\end{align*} is  the Lagrange function  associated with the problem $(OP)$, where $\lambda\in\mathbb{R}$, $e^* \in E^*$, $l=(l_1,l_2,..., l_m)\in\mathbb{R}^m$ and $w^* \in W^*$. We say that vector $(\lambda, e^*, l, w^*)\in \mathbb{R}\times E^*\times\mathbb{R}^n\times W^*$ are Lagrange multipliers of $(MP)$ at $z_0$ if the following conditions are fulfilled:
\begin{align}
	&D_z  {\mathcal L}\left( z_0,\lambda,  e^*, l,  w^* \right)= \lambda Df(z_0) + DF(z_0)^*e^* +\sum_{i=1}^m l_i DH_i(z_0) + DG(z_0)^* w^*=0,\label{L1}\\
	&\lambda\geq 0,\ l_i\geq 0,\ l_i H_i(z_0)=0, i=1,2,.., m,\label{L3}\\
	& e^*\in N(K, G(z_0)),\label{L4}
\end{align} where $l:=(l_1, l_2,..., l_n)$. We denote by $\Lambda[z_0]$ the set of Lagrange multipliers at $z_0$. In addition,  if $\lambda=1$, then we say $(1, e^*, l, w^*)$ are normal. We shall denote by $\Lambda_*[z_0]$ the set of  Lagrange multipliers $(e^*, l, w^*)$ such that $(1, e^*, l, w^*)\in\Lambda[z_0]$.

We now have the following important lemma.

\begin{lemma}\label{Lemma-key1} Suppose that $z_0\in \Sigma$  is a locally optimal solution of $(MP)$ under which  $(A1)$, $(A2)$ and $(A3)$ are valid,  and $DF(z_0)$ is surjective.  Then for each  $d\in \mathcal{C}_0[z_0]$, the system
	\begin{align}
 \begin{cases}
		&f'(z_0)z+\frac{1}2f''(z_0)d^2< 0,   \\
		&H'_i(z_0)z+\frac{1}22H''_i(z_0)d^2< 0, \ \ \forall i\in I_0,  \\
		&DF(z_0)z +\frac{1}2D^2 F(z_0)(d,d) =0, \\
		&D G(z_0)z+\frac{1}2D^2G(z_0))(d,d)\in {\rm cone}({\rm int} K-G(z_0))
  \end{cases}
	\end{align}
	has no solution $z\in Z$.
\end{lemma}
\begin{proof} The proof is proceeded analogously to the proof of \cite[Lemma 3.2]{Kien-Binh-2023}. For this  we need to use a Taylor expansion and the Ljusternik theorem.
\end{proof}

The following result gives  first and second-order necessary optimality conditions for $(MP)$.

\begin{proposition}\label{KeyProp} Suppose that $z_0$ is a locally optimal solution to $(MP)$. Then  the following assertions are fulfilled:
	
\noindent $(a)$ If $(A1), (A2)$ and $(A3)$ are satisfied, then for each  $d\in \mathcal{C}_0[z_0]$, there exist  $(\lambda, w^*,l,  e^*)\in \Lambda_*[z_0]$  such that 
\begin{align}\label{VOP-SOC1}
 D^2_z\mathcal{L}(z_0, \lambda, e^*, l,  w^*)[d,d]\geq 0.
\end{align} As a consequence, 
\begin{align*}
 \sup_{(\lambda, e^*,l,  w^*)\in\Lambda[z_0]}D^2_z\mathcal{L}(z_0, \lambda, e^*, l,  w^*)[d,d]\geq 0\quad\forall d\in\mathcal{C}_0[z_0]. 
\end{align*}

\noindent $(b)$ If $(A1), (A2)$ and $(A4)$ are satisfied, then $\Lambda_*[z_0]$  is nonempty and compact  in the topology $\sigma(\mathbb{R}^k, \mathbb{R}^k)\times\sigma(E^*, E)\times\sigma(W^*, W)$,  and  
\begin{align}\label{VOP-SOC2}
\max_{(e^*,l,  w^*)\in\Lambda_*[z_0]} D^2_z\mathcal{L}(z_0, 1, e^*, l,  w^*)[d,d]\geq 0\quad \forall d\in \mathcal{C}[z_0].   
\end{align}
\end{proposition}
\begin{proof} $(a)$. Fixing  any $d\in \mathcal{C}_0[z_0]$, we consider two cases.
	 
\noindent {\it Case 1}. $D F(z_0)Z\neq E$. 

Then there is a point $e\in E$ such that $e\notin DF(z_0)Z$. Since $DF(z_0)Z$ is a closed subspace,  the separation theorem (see \cite[Theorem 3.4]{Rudin}) implies that there exists a nonzero functional $e^*\in E^*$ which separates $e$ and $\nabla F(z_0)Z$, that is, $\langle e^*, e\rangle \geq \langle e^*, DF(z_0)z\rangle$ for all $z\in Z$. This implies that $DF(z_0)^* e^*=0$. Putting $\lambda=0, l=0, w^*=0$, we get $(0, e^*, 0, 0)\in\Lambda(z_0)$. If $e^*DF(z_0)d^2\leq 0$, then we replace $e^*$ by $-e^*$. Then $(0, -e^*, 0, 0)$ satisfies the conclusion of $(a)$. 
	
\noindent {\it Case 2}. $D F(z_0)Z= E$.

Let $m_0=|I_0|$. We define a set $S$ which consists of vectors $(\mu, \gamma,  e, w)\in\mathbb{R}\times\mathbb{R}^{m_0}\times E\times W$ such that there exists $z\in Z$ satisfying
\begin{align}
\begin{cases}
	&f'(z_0)z+\frac{1}2f''(z_0)d^2<\mu,\\
	&H'_j(z_0)z+\frac{1}2H''_j(z_0)d^2<\gamma_j, j\in I_0,\\
	&D F(z_0)z+\frac{1}2 D^2F(z_0)d^2=e \\
	&D G(z_0)z+ \frac{1}2 D^2G(z_0)d^2 -w\in{\rm cone}({\rm int}K-G(z_0)).
 \end{cases}
\end{align} It is clear that $S$ is convex. By using the open mapping theorem, we can  show that $S$ is open.   By Lemma \ref{Lemma-key1}, we have  $(0, 0, 0, 0)\notin S$. By  the separation theorem (see \cite[Theorem 1, p. 163]{Ioffe-1979},  there exists a nonzero vector 
$(\lambda, l, e^*,  w^*)\in \mathbb{R}^k\times\mathbb{R}^{m_0}\times E^*\times \times W^*$ such that 
\begin{align}\label{Seperation2}
	\lambda\mu + l\gamma \langle e^*, e\rangle  +\langle w^*, w\rangle\geq 0\quad \forall (\mu, \gamma, e,  w)\in S. 
\end{align} 
Fix any $z\in  Z$,  $w'\in{\rm cone }({\rm int }K-G(z_0))$,  $r>0$ and $r_j>0$ . Put
\begin{align*}
	& \mu=r+ f'(z_0)z +\frac{1}2 f''(z_0)d^2,\ \gamma_j=r_j+  H'_j(z_0)z +\frac{1}2H''_j(z_0)d^2,\\
	&e=D F(z_0)z +\frac{1}2D^2 F(z_0)d^2,\  w=D G(z_0)z+ \frac{1}2D^2 G(z_0)d^2-w'.
\end{align*} Then $(\mu,\gamma,  e, w)\in S$. From this and  \eqref{Seperation2}, we get
\begin{align*}
&\lambda(r+ f'(z_0)z +\frac{1}2 f''(z_0)d^2)+ \sum_{j\in I_0} l_j(r_j+ H'_j(z_0)z +\frac{1}2 H''_j(z_0)d^2)\\
&+ \langle e^*, DF(z_0)z +\frac{1}2D^2 F(z_0)d^2\rangle +\langle w^*, DG(z_0)z +\frac{1}2 D^2G(z_0)d^2\rangle -\langle w^*, w'\rangle\geq 0.
\end{align*} If  $\lambda<0$ then by letting $r\to+\infty$, the term on the left hand side approach to $-\infty$ which is impossible. Hence we must have $\lambda \geq 0$. Similarly, we have $l_j\geq 0$ for all $j\in I_0$. 

By letting $r\to 0$ and $r_j\to 0$,  we get 
\begin{align}\label{Seperation3}
	\lambda f'(z_0)+\sum_{j\in I_0}l_jH'_j(z_0) + DF(z_0)^* e^* + DG(z_0)^* w^*=0
\end{align} and 
$$
\frac{1}2 \big(f''(z_0)d^2+ \sum_{j\in I_0} l_i H''_i(z_0) d^2+  \langle e^*, D^2F(z_0)d^2\rangle +
 \langle w^*, D^2G(z_0)d^2\rangle \big)\geq \langle w^*, w'\rangle 
$$ for all $w'\in{\rm cone}({\rm int}K-G(z_0))$. It follows that 
$$
\langle w^*, w'\rangle\leq 0,  \forall w'\in{\rm cone}({\rm int}K-G(z_0)).
$$ Since $ K\subseteq\bar K=\overline{{\rm int}K}$, we obtain
$$
\langle w^*, w'\rangle\leq 0,  \forall w' \in {\rm cone}(K-G(z_0)).
$$ This implies that $w^*\in N(K, G(z_0))$.     By letting $w'\to 0$, we get 
\begin{align}\label{Saperation4}
f''(z_0)d^2+ \sum_{j\in I_0} l_j H''_i(z_0) d^2+  \langle e^*,D^2F(z_0)d^2\rangle +
 \langle w^*, D^2G(z_0)d^2\rangle \geq 0.
\end{align} We now take $l_j=0$ for $j\in\{1,2,..., m\}\setminus I_0$. Then we have $l_j\geq 0$ and $l_j H_j(z_0)=0$ for all $j=1,2,.., m$ and \eqref{Seperation3} and \eqref{Saperation4} become
\begin{align*} 
	 f'(z_0)+\sum_{j=1}^n l_jH'_j(z_0) + DF(z_0)^*e^* + DG(z_0)^*w^*=0
\end{align*} and 
\begin{align*}
	f''(z_0)d^2+ \sum_{j=1}^n l_jH''_i(z_0) d^2+  \langle e^*. D^2F(z_0)d^2\rangle +
	\langle w^*, D^2G(z_0)d^2\rangle \geq 0.
\end{align*} Assertion $(a)$ is proved. 

\noindent $(b)$. We claim that $\lambda\neq 0$. Indeed, if $\lambda=0$, then 
$$
DH(z_0)^* l^T+ DF(z_0)^* e^* + DG(z_0)^* w^*= 0.
$$ Here $H:=(H_1, H_2,..., H_m)$.  Let $\tilde z$ be a vector  satisfying  $(A4)$. Then we have 
$$
l^TDH(z_0)\tilde z+  \langle e^*,  DF(z_0)\tilde z\rangle   + \langle w^*,  DG(z_0)^* \tilde z\rangle =0.
$$  This implies that 
$$
l^T DH(z_0)\tilde z + \langle w^*, DG(z_0) \tilde z\rangle=0,  
$$ where
$$
DH(z_0)\tilde z\in (-\infty, 0)^m -H(z_0)={\rm int}( (-\infty, 0]^m-H(z_0))
$$ and 
$$
DG(z_0) \tilde z \in {\rm int}K- G(z_0)={\rm int}(K-G(z_0)). 
$$ Then for  any $\xi\in\mathbb{R}^m$ and $w\in W$, there exist $\gamma>0$ and  $s>0$ small enough such that 
$$
\gamma\xi + DH(z_0)\tilde z \in (-\infty, 0]^m-H(z_0)
$$
and 
$$
sw+ DG(z_0) \tilde z \in K-G(z_0).
$$  Since $l\geq 0$, $l^T H(z_0)=0$ and $w^*\in N(K, G(z_0))$, we have 
$$
\gamma l^T\xi + s\langle w^*, w\rangle=l^T(\tau \xi+ DH(z_0)\tilde z)+\langle w^*, sw+ DG(z_0) \tilde z \rangle\leq 0.  
$$ Hence $l^T\xi\leq 0$ for all $\xi\in\mathbb{R}^n$ and  $\langle w^*, w\rangle \leq 0$ for all $w\in W$. This implies that $l=0$ and  $w^*=0$. Consequently, $DF(z_0)^* e^*=0$. Since $DF(z_0)$ is surjective, we get $e^*=0$. Hence $(\lambda, e^*, l,  w^*)=(0,0,0, 0)$ which is absurd. The claim is justified. 

Let us define  $\widehat{K}=\{0\}\times (-\infty, 0]^m\times K$ and $\widehat{G}=(F, H,  G)$. Then the constraint of $(MP)$ is  equivalent to  constraint $\widehat{G}(z) \in \widehat{K}$.
According to \cite[Corollary 2.101,  page 70]{Bonnans}, $(A4)$ is equivalent to the  Robinson constraint qualification:
\begin{align*}
	0\in {\rm int} \{\tilde{G}(z_0) + D\tilde{G}(z_0)Z -\tilde{K}\}. 
\end{align*} By \cite[Theorem 3.9, p. 151]{Bonnans}, $\Lambda_*[z_0]$ is  nonempty convex bounded and compact in the topology $\sigma(E^*, E)\times\tau(\mathbb{R}^m)\times\sigma(W^*, W)$. Since the function
$$
(e^*, l,  w^*, d)\mapsto D_z^2\mathcal{L}(z_0, 1, e^*, l, w^*)(d,d).
$$ is continuous on  the topology $\sigma(E^*, E)\times\tau(\mathbb{R}^m)\times\sigma(W^*, W)\times \tau(Z)$,  the function 
$$
d\mapsto \max_{(e^*,l, w^*)\in\Lambda (z_0)}D_z^2\mathcal{L}(z_0, 1, e^*, l, w^*)(d, d)
$$ is continuous (see \cite[Theorem 1 and 2, p. 115]{Berge}). Taking any  $d\in \mathcal{C}[z_0]$, we see that, there exists a sequence $(d_j)\subset \mathcal{C}_0[z_0]$ such that $d_j\to d$. By assertion $(a)$, 
\begin{align*}
\max_{( e^*,l,  w^*)\in\Lambda_*[z_0]}D_z^2\mathcal{L}(z_0, 1, e^*, l,  w^*)(d_j, d_j)\geq 0.
\end{align*} Passing to the limit, we obtain 
\begin{align*}
	\max_{( e^*, l, w^*)\in\Lambda_*[z_0]}D_z^2\mathcal{L}(z_0, 1, e^*,l,  w^*)(d, d)\geq 0.
\end{align*}The proof of the position is complete.
\end{proof}
\begin{remark} {\rm When $K$ is a cone and $H_j=0$ for all $j=1,2,..., m$, we get a similar result with \cite[Theorem 8.2]{Bental}. However, in this case, the obtained result is failed to apply for  problem $(P)$.  }   
\end{remark}

\subsection{State equation}

In the sequel, we assume that $\partial\Omega$ is of class $C^2$. Let  $H:=L^2(\Omega)$, $V:=H_0^1(\Omega)$ and $D:= H^2(\Omega)\cap H_0^1(\Omega)$. The norm and the scalar product in $H$ are denoted by $|\cdot|$ and $(\cdot, \cdot)_H$,  respectively.   It is known that the embeddings 
$$
D\hookrightarrow\ \hookrightarrow V\hookrightarrow\ \hookrightarrow H
$$ and each space is dense in the following one. Hereafter,  $\hookrightarrow\ \hookrightarrow$ denotes a compact embedding. 

\noindent $W^{m,p}(\Omega)$ for $m$ integer and $p\geq 1$ is a Banach space of elements in $ v\in L^p(\Omega)$ such that their generalized derivatives $D^\alpha v\in L^p(\Omega)$ for $|\alpha|\leq m$. The norm of a element  $v\in W^{m,p}(\Omega)$ is given by 
$$
\|u\|_{m,p}=\sum_{|\alpha|\leq m}\|D^\alpha u\|_p,
$$ where $\|\cdot\|_p$ is the norm in $L^p(\Omega)$. The closure of $C_0^\infty(\Omega)$ in  $W^{m,p}(\Omega)$ is denoted by $W_0^{m, p}(\Omega)$.   When $p=2$, we write $H^m(\Omega)$ and $H^m_0(\Omega)$ for $W^{m, 2}(\Omega)$ and $W^{m, 2}_0(\Omega)$, respectively. 

\noindent $W^{s, p}(\Omega)$ for real number $s\geq 0$ consists of elements  $v$ so that 
$$
\|v\|_{s,p}=\Big[\|v\|^p_{m,p} +\sum_{|\alpha|=m}\int_\Omega\int_\Omega \frac{|D^\alpha v(x)-D^\alpha v(x')|^p}{|x-x'|^{N+\sigma p}} dxdx' \Big]^{1/p}< +\infty,
$$ where $s=m+\sigma$ with $m=[s]$ and $\sigma\in (0, 1)$. It is known that 
\begin{align}
  W^{s,p}(\Omega)=\big(W^{k_0, p}(\Omega), W^{k_1, p}(\Omega)\big)_{\theta, 1},  
\end{align} where $s=(1-\theta)k_0 +\theta k_1$, $0<\theta <1$, $k_0<k_1$ and $1\leq p\leq \infty$. Here $\big(W^{k_0, p}(\Omega), W^{k_1, p}(\Omega)\big)_{\theta, 1}$ is a interpolation space in $K-$method. 

\noindent $W^{2l, l}_p(Q_T)$ for $l$ integer and $p\geq 1$ is a Banach space of elements in $ v\in L^p(Q_T)$ such that their generalized derivatives $D^r_t D^s_x v\in L^p(\Omega)$ with $2r+s\leq 2l$. The norm of a element  $v\in W^{2l,l}_p(\Omega)$ is given by 
\begin{align*}
    \big\| v\big\|_{p, Q_T}^{(2l)}=\sum_{j=0}^{2l} \langle\langle v\rangle\rangle^{(j)}_{p, Q_T},
\end{align*} where 
\begin{align*}
  \langle\langle v\rangle\rangle^{(j)}_{p, Q_T}:=\sum_{2l+s=j}\|D^r_t D^s_x v\|_{L^p(Q_T)}.   
\end{align*}

Given $\alpha\in (0, 1)$ and $T>0$, we denote by $C^{0, \alpha}(\Omega)$ and $C(\bar Q_T)$ the space of H\"{o}lder continuous functions on $\Omega$ and the space of continuous functions on $\bar Q_T$, respectively. 

Let $H^{-1}(\Omega)$ be the dual of $H_0^1(\Omega)$. We define the following function spaces
\begin{align*}
& H^1(Q)=W^{1,1}_2(Q)=\{y\in L^2(Q):  \frac{\partial y}{\partial x_i}, \frac{\partial y}{\partial t}\in L^2(Q)\},\\
&V_2(Q)=L^\infty(0, T; H)\cap L^2(0, T; V),\\
&W(0, T)=\{y\in L^2(0, T; H^1_0(\Omega)): y_t\in L^2(0, T; H^{-1}(\Omega))\},\\
& W^{1,1}_2(0, T; V, H)=\{y\in L^2([0, T], V): \frac{\partial y}{\partial t}\in L^2([0, T], H) \},\\
& W^{1,1}_2(0, T; D, H)=\{y\in L^2([0, T], D): \frac{\partial y}{\partial t}\in L^2([0, T], H) \},\\
 &U_T=L^\infty(Q_T),\\ 
 &Y_T =\big\{y\in  W^{1,1}_2(0, T; D, H) \cap  W^{2,1}_p(Q_T)| Ay\in L^p(Q_T)\big\}.
\end{align*} Hereafter, we assume that 
\begin{align} \label{Dimension1}
    1+\frac{N}{2}<p< N+2
\end{align}  and $q$ is the conjugate number of $p$. Then $Y_T$ is a Banach space under the graph norm
\begin{align*}
    \|y\|_{Y_T}:= \|y\|_{W^{1, 1}_2(0, T; D, H)} + \|y\|_{W^{2,1}_p(Q_T)} +\|Ay\|_{L^p(Q_T)}. 
\end{align*} Note that if $y\in W^{2,1}_p(Q_T)$, then $y_t\in L^p(0, T; L^p(\Omega))$. By  the proof of Lemma \ref{Lemma-stateEq}, we have
\begin{align}\label{keyEmbed1}
  W^{2,1}_p(Q_T)\hookrightarrow\  \hookrightarrow C(\bar Q_T).  
\end{align} Hence 
\begin{align}
Y_T \hookrightarrow\ \hookrightarrow  C(\bar Q_T).
\end{align} Besides, we have the following continuous embeddings: 
\begin{align}\label{keyEmbed3}
W^{1, 1}_2(0, T; D, H)\hookrightarrow C([0, T], V),\quad W(0, T)\hookrightarrow C([0, T], H),
\end{align} where $C([0, T], X)$ stands for the space of continuous mappings $v: [0, 1]\to X$ with $X$ is a Banach space. 

Recall that given  $y_0\in H$ and $u\in L^2(0, T; H)$,  a function $y\in W(0, T)$ is said to be a weak solution of the semilinear parabolic equation \eqref{P2}-\eqref{P3} if 
\begin{align}\label{WeakSol}
\begin{cases}
\langle y_t, v\rangle + \sum_{i,j=1}^n \displaystyle \int_\Omega a_{ij}D_iy D_j v dx +(\psi(\cdot, \cdot, y), v)_H =(u, v)_H\quad \forall v\in H_0^1(\Omega)\quad \text{a.a.}\quad t\in [0, T]\\
y(0)=y_0.
\end{cases}
\end{align} 
If a  weak solution $y$ such that $y\in W^{1,1}_2(0, T; D, H)$ and $\psi(\cdot, \cdot, y)\in L^2(0, T; H)$ then we have $y_t +Ay +\psi(\cdot, \cdot, y)\in L^2(0, T; H)$ and
$$
\langle y_t, v\rangle + (Ay, v) +(\psi(\cdot, \cdot, y), v) =(u, v)\quad \forall v\in H_0^1(\Omega).
$$ Since $H_0^1(\Omega)$ is dense in $H=L^2(\Omega)$, we have  
$$
(y_t, v) + (Ay, v) +(\psi(\cdot,  \cdot, y), v) =(u, v)\quad \forall v\in H.
$$  Hence
$$
y_t + Ay + \psi(\cdot, \cdot, y) =u\quad \text{a.a.}\ t\in [0, T],\ y(0)=y_0.
$$ In this case we say $y$ is a {\it strong solution} of \eqref{P2}-\eqref{P3}. From now on a  solution to \eqref{P2}-\eqref{P3} is understood a strong solution. 

Let us make the following assumptions which are related to the state equation. 

\begin{enumerate}
\item [$(H1)$] Coefficients ${a_{ij}}=a_{ji} \in C(\bar \Omega)$ for every $1 \le i,j \le N$, satisfy the uniform ellipticity conditions, i.e.,  there exists a number $\gamma > 0$ such that 
\begin{align}
\gamma {\left| \xi  \right|^2} \le \sum\limits_{i,j = 1}^N {{a_{ij}}\left( x \right){\xi _i}{\xi _j}} \,\,\,\,\,\,\,\,\forall \xi  \in {{\mathbb R}^N}\,\,\,{\rm {for\,\,\,a.a.}}\,\,\,x \in \Omega .
\end{align}

\item[$(H2)$]  The mapping $\psi:\bar\Omega\times[0, +\infty)\times{\mathbb R} \to {\mathbb R}$ is a Carath\'{e}odory and for each $(x, t)\in \bar\Omega\times[0, +\infty)$, $\psi(x, t, \cdot)$ is of class $C^2$ in $y$ and satisfies the following properties:
\begin{align*}
& \psi_y(x, t, y)\geq 0\quad\forall (x,t, y)\in \bar\Omega\times[0, +\infty)\times{\mathbb R},\ \psi(\cdot, \cdot, 0)\in L^p(Q_T)\quad\forall T>0,\\
&\forall M>0, \exists k_{\psi,M}>0: |\psi_y(x, t, y)| \le k_{\psi,M}, \\
&|\psi(x, t_1, y_1)-\psi(x, t_2, y_2)|+ |\psi'(x, t_1, y_1)-\psi'(x, t_2, y_2)|+|\psi''(x,t_1, y_1)-\psi''(x, t_2, y_2)|\\
&\leq k_{\psi,M}(|t_1-t_2|+|y_1-y_2|)
\end{align*} for all $y, y_i, t_i$ satisfying $ |y|, |t_i|, |y_i|\leq M$ with $i=1,2$. 

\item[$(H3)$] $y_0\in H_0^1(\Omega)\cap W^{2-\frac{2}p, p}(\Omega)$, where $p>1$ satisfying \eqref{Dimension1}. 

\end{enumerate}

Note that $(H3)$ makes sure that $y_0\in C(\bar \Omega)$.  Since $2> 2-\frac{2}p>1$, Theorem 1.4.3.2 in \cite{Grisvard} implies that   
$$
W^{2,p}(\Omega)\hookrightarrow\ \hookrightarrow W^{2-\frac{2}p, p}(\Omega)\hookrightarrow\ \hookrightarrow W^{1,p}(\Omega).
$$   Hence the condition $y_0\in W^{2-\frac{2}p, p}(\Omega)$  relaxes the condition $y_0\in W^{2,p}(\Omega)$.

\begin{lemma}\label{Lemma-stateEq} Suppose that $(H1), (H2)$ and $(H3)$ are valid. Then for each $T>0$ and  $u\in L^p(Q_T)$ with $p$ satisfying \eqref{Dimension1},  the state equation \eqref{P2}-\eqref{P3} has a unique  solution   $y\in Y_T$ and there exist positive constants $C_1>0$ and $C_2>0$ such that 
\begin{align}\label{KeyInq0}
    \|y\|_{C(\bar Q_T)} +\|\psi(\cdot, \cdot, y)\|_{L^2(0, T; H)}\leq C_1 (\|u\|_{L^p(Q_T)} + \|y_0\|_{L^\infty(\Omega)}+\|\psi(\cdot, \cdot, 0)\|_{L^p(Q_T)})
\end{align} and 
\begin{align}\label{KeyInq1}
\|y_t\|_{L^p(Q_T)}+\|y\|_{L^p(Q_T)}+\|Dy\|_{L^p(Q_T)}+\|D^2 y\|_{L^p(Q_T)}\leq C_2(\|y_0\|_{C(\bar\Omega)} +  \|u\|_{L^p(Q_T)}+ \|\psi(\cdot, \cdot, 0)\|_{L^p(Q_T)}), 
\end{align} where $C_2$ depends on $T, \Omega, \|y_0\|_{C(\bar\Omega)}$,  $\|u\|_{L^p(Q_T)}$ and $\|\psi(\cdot, \cdot, 0)\|_{L^p(Q_T)}$. 
\end{lemma}
\begin{proof} We first claim that $y_0\in C(\bar\Omega)$. In fact, by \eqref{Dimension1}, $(2-\frac{2}p)-\frac{N}p\in(0, 1)$. By  \cite[Theorem 1.4.3.1 and Theorem 1.4.4.1]{Grisvard}, we have $W^{2-\frac{2}p, p}(\Omega)\hookrightarrow C^{0, \alpha}(\Omega)$ with $\alpha=(2-\frac{2}p)-\frac{N}p.$ This implies that $y_0\in C(\bar\Omega)$. Hence $y_0(x)=0$ for $x\in\Gamma$ and the claim is justified.  

By the same argument as in  proofs of \cite[Theorem 2.1]{Casas-2023} and \cite[Lemma 2.2.III]{Casas-2023}, we see that for each $u\in L^p(0, T; L^p(\Omega))=L^p(Q_T)$ with $p>\frac{N+2}{2}$, the state equation has a unique solution $y\in L^\infty(Q)\cap W^{1, 1}_2(0, T; V, H)$ such that $\psi(\cdot, \cdot, y)\in L^2(0, T; H)$ and inequalities \eqref{KeyInq0} is fulfilled.


By $(H2)$, we have $\psi(\cdot, \cdot, y)\in L^\infty(Q_T)$. Let us show that $y\in W^{2,1}_p(Q_T)$. For this,  we consider equation
$$
z_t + Az = u-\psi(x, t, y),\  z(0)=y_0.
$$ Since $u-\psi(\cdot, \cdot, y)\in L^p(Q_T)$, \cite[Theorem 9.1, p. 341]{Ladyzhenskaya} implies that the equation has a unique solution $z\in W^{2,1}_p(Q_T)$ such that $y(0)=y_0$ and $y(x, t)=0$ for  $(x, t)\in\Gamma\times[0, T]$. Besides,   there exists a constant $C>0$ such that 
\begin{align}\label{KeyInq2}
 \|z\|^{(2)}_{p, Q_T}\leq C(\|y_0\|_{2-\frac{2}{p},p} +\|u-  \psi(\cdot, \cdot, y)\|_{L^p(Q_T)}). 
\end{align} 
By definition of norm $\|y\|^{(2)}_{p, Q_T}$, we have 
\begin{align} 
\|z_t\|_{L^p(Q_T)}+\|z\|_{L^p(Q_T)}+\|Dz\|_{L^p(Q_T)}+\|D^2 z\|_{L^p(Q_T)}\leq C(\|y_0\|_{2-\frac{2}{p},p} +\|u-  \psi(\cdot, \cdot, y)\|_{L^p(Q_T)}). 
\end{align} This implies that 
$$
z\in W^{2,1}_p(Q_T)=\{z\in L^p(0, T; W^2_p(\Omega))| z_t\in L^p(0, T; L^p(\Omega))\}. 
$$ We now prove  that \eqref{keyEmbed1} is valid and so $z\in C(\bar Q_T)$. Indeed, consider the space $W^{s,p}(\Omega)$, where
 $s\in (\frac{N}p, 2-\frac{2}p)$. By \cite[Theorem 1.4.3.2]{Grisvard},  the embeddings
\begin{align}
W^{2,p}(\Omega)\hookrightarrow\ \hookrightarrow W^{s,p}(\Omega)\hookrightarrow\ \hookrightarrow L^p(\Omega).
\end{align}
 By \eqref{Dimension1},  $0<s-N/p< 1$. Using  \cite[Theorem 1.4.3.1 and Theorem 1.4.4.1]{Grisvard} again, we have   
\begin{align}\label{ContEmb}
W^{s,p}(\Omega)\hookrightarrow C^{0,\alpha}(\Omega)=C^{0,\alpha}(\bar\Omega)\hookrightarrow C(\bar\Omega). 
\end{align} Here we used the fact that if $\Omega$ is bounded domain, then $C^{0,\alpha}(\Omega)=C^{0,\alpha}(\bar\Omega)$. 
By the generalized Gagliardo-Nirenberg inequality in  \cite[Theorem 1, part A]{Brezis2}, we have 
\begin{align*}
\|v\|_{s, p}\leq C\|v\|^{\theta}_{2,p}\|v\|^{1-\theta}_{0, p} \quad \quad  \forall v \in W^{2,p}(\Omega), 
\end{align*}
where $\theta=s/2\in (\frac{N}{2p}, 1  -  \frac{1}{p})\subset (0, 1)$. By \cite[section 1.10.1, Lemma (a), p. 61]{Triebel}, we have 
\begin{align}
\big( L^p(\Omega), W^{2,p}(\Omega)\big)_{\theta,1}\hookrightarrow W^{s,p}(\Omega)\hookrightarrow L^p(\Omega). 
\end{align}
Using   \cite[Theorem 7.4.3, p. 268]{Amann} for the case  $p_\theta=p_0=p_1=p$ and $0< 1-\theta -\frac{1}p$,  we obtain 
\begin{align*}
W^{2,1}_p(Q_T)\hookrightarrow\ \hookrightarrow C([0, T], W^s_p(\Omega)).
\end{align*}
From this and \eqref{ContEmb}, we get 
\begin{align*}
W^{2,1}_p(Q_T)\hookrightarrow\ \hookrightarrow C([0, T], C(\bar\Omega))\hookrightarrow C(\bar Q_T)  
\end{align*} and \eqref{keyEmbed1} is justified.

Since $z(x, t)=0$ on $\Gamma\times [0, T]$, we get  $z\in C([0, T], W^{1, p}_0(\Omega))\subset L^2(0, T; H^1_0(\Omega))$. By the uniqueness, we have $z=y$. 

Let  $M> C_1 (\|u\|_{L^p([0, T], H)} + \|y_0\|_{L^\infty(\Omega)})+ \|\psi(\cdot, \cdot, 0)\|_{L^p(Q_T)}$. By $(H2)$ and \eqref{KeyInq0},   there exists $k_M>0$ such that 
$$
|\psi(x, t, y)|=|\psi(x, t, y)-\psi(x, t, 0)|+ |\psi(x, t, 0)|\leq k_M|y| +|\psi(x, t,0)|. 
$$ Combining this with \eqref{KeyInq2} and \eqref{KeyInq0}, we have 
\begin{align*}
\|y\|^{(2)}_{p, Q_T}&\leq C\big(\|y_0\|_{2-\frac{2}{p},p} +\|u\|_{L^p(Q_T)} +k_M\|y\|_{L^p(Q_T)} +\|\psi(\cdot, \cdot, 0)\|_{L^p(Q_T)}\big)\\
&\leq C(\|y_0\|_{2-\frac{2}{p},p} +\|u\|_{L^p(Q_T)} +k'_M\|y\|_{C(Q_T)}+\|\psi(\cdot, \cdot, 0)\|_{L^p(Q_T)})\\
&\leq C_2(\|y_0\|_{C(\bar\Omega)} +  \|u\|_{L^p(Q_T)}+\|\psi(\cdot, \cdot, 0)\|_{L^p(Q_T)}) 
\end{align*} for some constant $C_2>0$ which depends on $T, \Omega, \|y_0\|_{C(\bar\Omega)}$,  $\|u\|_{L^p(Q_T)}$ and $\|\psi(\cdot, \cdot, 0)\|_{L^p(Q_T)}$. By definition of norm $\|y\|^{(2)}_{p, Q_T}$, we have 
\begin{align} 
\|y_t\|_{L^p(Q_T)}+\|y\|_{L^p(Q_T)}+\|Dy\|_{L^p(Q_T)}+\|D^2 y\|_{L^p(Q_T)}\leq C_2(\|y_0\|_{C(\bar\Omega)} +  \|u\|_{L^p(Q_T)})
\end{align} which is inequality \eqref{KeyInq1}.  Since $Ay= u-\psi(\cdot, \cdot, y) - y_t\in L^p(Q_T)$,  $y\in Y_T$. The  lemma is proved.
\end{proof}

\medskip

Let us consider the linearized equation 
\begin{align}\label{LinearizedEq1}
   y_t + A y + c(x,t)y =u,\quad y(0)=\phi_0. 
\end{align} 

\begin{lemma}\label{Lemma-LinearizedEq}
Suppose $u\in L^p(Q_T)$, $c\in L^\infty (Q_T)$ and  $\phi_0\in W^{2-\frac{2}p, p}(\Omega)\cap H^1_0(\Omega)$ with $p$ satisfying \eqref{Dimension1}. Then equation \eqref{LinearizedEq1} has a unique solution $y\in Y_T$ and there exist positive  constant $C_1$ and $C_2$ which depend on $T$ such that 
\begin{align}
  & \|y_t\|_{L^p(Q_T)}+\|y\|_{L^p(Q_T)}+\|Dy\|_{L^p(Q_T)}+\|D^2 y\|_{L^p(Q_T)}\leq C_1\big(\|u\|_p +\|\phi_0\|_{W^{2-\frac{2}p, p}(\Omega)}\big),  \label{KeyInq2.1} \\
  &\|y\|_{C(\bar Q_T)} \leq C_2 \big(\|u\|_p +\|\phi_0\|_{W^{2-\frac{2}p, p}(\Omega)}\big).\label{KeyInq3}
\end{align}
\end{lemma}
\begin{proof} The conclusion of the lemma follows directly from \cite[Theorem 9.1, p. 341]{Ladyzhenskaya}. By embedding \eqref{keyEmbed1},   we have estimate \eqref{KeyInq3}.
\end{proof}

\section{Main results}

Let $L=L(x, t, y, u):\Omega\times[0, +\infty)\times \mathbb{R}\times\mathbb{R}\to\mathbb{R}$ be a function. Given a triple $ (T_*, y_*, u_*)\in\Phi$, the symbols  $L[x, t]$,  $L_t[x, t], L_y[x, t]$, $L_u[x, t]$, $L[\cdot, \cdot]$, etc., stand for   $L(x, t, y_*(x, t), u_*(x, t))$, $L_\xi(x, t, y_*(x,t), u_*(x, t))$,  $L_y(x, t, y_*(x, t), u_*(x, t))$, $L_u(x, t, y_*(x, t), u_*(x, t))$, $L(\cdot, \cdot, y_*(\cdot, \cdot), u_*(\cdot, \cdot))$, etc., respectively. Let $\psi_i=\psi_i(T, \zeta)$. The symbols  $\psi_i[T_1, T_2], \psi_{iT}[T_1, T_2], \psi_{i\zeta}[T_1, T_2]$, etc., stand for $\psi_i(T_1, y_*(T_2))$, $\psi_{iT}(T_1, y_*(T_2))$, $\psi_{i\zeta}(T_1, y_*(T_2))$, etc.,  respectively.    Define 
\begin{align*}
    K_{T_*} :=  \{v \in L^\infty(Q_{T_*}) : a  \le  v(x, t) \le b \ {\rm a.a.} \ (x, t) \in Q_{T_*}\}. 
\end{align*}
Let $\ell: \Omega\times[0, +\infty)\times\mathbb{R}\times\mathbb{R}\to\mathbb{R}$ be a mapping which stands for $L$ and $g$. We impose the following hypotheses. 

\begin{itemize}

    \item [$(H4)$] (i) $\ell$ is a Carath\'{e}odory  function and  for each $x\in \Omega$, $\ell(x, \cdot, \cdot, \cdot)$ is of class $C^2$ and satisfies the following property:  for each $M>0$, there exists $k_{\ell,M}>0$ such that 
\begin{align*}
    &|\ell(x, t_1, y_1, u_1)-\ell(x, t_2, y_2, u_2)|+ |\ell_t(x, t_1, y_1, u_1)-\ell_t(x, t_2, y_2, u_2)|+\\
    & |\ell_y(x, t_1, y_1, u_1)-\ell_y(x, t_2, y_2, u_2)|+|\ell_u(x, t_1, y_1, u_1)-\ell_u(x, t_2, y_2, u_2)|+\\
    &|\ell_{tt}(x, t_1, y_1, u_1)-\ell_{tt}(x, t_2, y_2, u_2)|+|\ell_{yy}(x, t_1, y_1, u_1)-\ell_{yy}(x, t_2, y_2, u_2)|+\\
    &|\ell_{yu}(x, t_1, y_1, u_1)-\ell_{yu}(x, t_2, y_2, u_2)|+ |\ell_{uu}(x, t_1, y_1, u_1)-\ell_{uu}(x, t_2, y_2, u_2)|  \\
    &\leq k_{\ell,M}(|t_1-t_2|+ |y_1-y_2|+ |u_1-u_2|)
\end{align*}  
for all $(x, t_i, y_i, u_i)\in \Omega\times[0, +\infty)\times \mathbb{R}\times\mathbb{R}$ satisfying $|t_i|, |y_i|, |u_i|\leq M$ with $i=1,2$. Furthermore, we require that the functions $\ell_y(\cdot, \cdot, 0, 0), \ell_u(\cdot, \cdot, 0, 0), \ell_{yy}(\cdot, \cdot, 0, 0), \ell_{yu}(\cdot, \cdot, 0, 0)$ and $\ell_{uu}(\cdot, \cdot, 0, 0)$ belong to $L^\infty(Q_{T_*})$. 

$(ii)$ $\psi_j:[0, \infty)\times C(\bar\Omega)\to\mathbb{R}$ is of class $C^2$ in $(T, \zeta)$ with $j=0,1,..., m$.

\item [$(H5)$] The function $\frac{1}{g_u(\cdot, \cdot, y, u)}$ belongs to $L^p (Q_T)$ for all $(T, y, u)\in \Phi$.

\item [$(H6)$]  There exists $(\widehat T, \widehat y, \widehat u) \in \mathbb{R}\times Y_{T_*} \times U_{T_*}$ such that the following conditions are fulfilled: 

\noindent \quad \quad $(i)$ \quad    $\dfrac{\partial \widehat y}{\partial t} +  A\widehat y  +  \psi_t[x,t]\frac{\widehat T t}{T_*}+  \psi_y[x, t]\widehat y - \widehat u+ \frac{\widehat T}{T_*}(Ay_*+\psi[x, t] - u_*)=0, \quad    \widehat y(0)=0$;

\noindent \quad \quad $(ii)$ \quad $\psi_i[T_*, T_*] +\psi_{iT}[T_*, T_*]\widehat T  + \psi_{i\zeta}[T_*, T_*]\widehat y(T_*) < 0$, $i = 1, 2, ..., m;$

\noindent \quad \quad $(iii)$  \quad $g[x, t] + g_t[x,t]\frac{\widehat T t}{T_*}+  g_y[x, t]\widehat y + g_u[x, t]\widehat u  \in {\rm int}(K_{T_*})$. 

\item [$(H7)$] For each $(T_0, w_0)\in [0, +\infty)\times C(\bar\Omega)$, the following mappings are continuous:
\begin{align*}
    &H\ni h\mapsto \psi_{i\zeta}(T_0, w_0)h,\\
    &H\ni h\mapsto \psi_{i\zeta\zeta}(T_0, w_0)[h, h],\\
    &\mathbb{R}\times H\ni (T, h)\mapsto \psi_{iT\zeta}(T_0, w_0)[T,h].
\end{align*} 
\end{itemize}

Note that hypothesis $(H4)$ makes sure that $J$ and $g$ are of class $C^2$ on $Y_{T}\times U_{T}$ for each fixed $T>0$. Meanwhile, $(H5)$ guarantees that the Lagrange multipliers belong to $L^1(Q_{T_*})$ and $(H6)$ ensures that  the Lagrange multipliers are normal. Assumption $(H7)$ requires that the mapping $\psi_{i\zeta}(T_0, w_0), \psi_{i\zeta\zeta}(T_0, w_0) $ and $\psi_{iT\zeta}(T_0, w_0)$ have continuous extension on $H$ and $\mathbb{R}\times H$. For an example the function $\psi_0(T, y(T))= T+ \int_\Omega f(x, y(x, T))dx$ satisfies $(H7)$ whenever $f$ is a Carath\'{e}odory function and  $f(x, \cdot)$ is of class $C^2$. 
Let us give other example under which $(H5)$ and $(H6)$ are fulfilled.

\begin{example} {\rm   
    Let $b= 0$ and $g(x, t, y, u) = - u^3 -  u(y^2 + 1)$. It is easy to see that $\frac{1}{g_u[\cdot, \cdot]} \in L^\infty (Q_{T_*})$ and so $(H5)$ is valid. Sine $g[x, t] \le b = 0$, we have $0 \le u_* \in L^\infty(Q_{T_*})$. The maximum principle implies that the solution $y_*$ of the state equation corresponding to $u_*$ satisfies the property that $y_* \ge 0$ on $\bar Q_{T_*}$ and so $y_* (x, T_*) \ge 0$ a.a. $x \in \Omega$.  Let $\delta > 0$ and take $\widehat u \in L^\infty(Q_{T_*})$ such that $\widehat u(x, t) \ge \delta > 0$ a.a. $(x, t) \in Q_{T_*}$. By Lemma  \ref{Lemma-LinearizedEq}, equation 
    \begin{align*}
        \dfrac{\partial \widehat y}{\partial t} +  A\widehat y +  \psi_y[x, t] \widehat y - \widehat u = 0, \quad    \widehat y(0)=0
    \end{align*}
    has unique solution $\widehat y \in Y_{T_*}$.   And by the  maximum principle  we also have $\widehat y(x, t) \ge 0$ on  $Q_{T_*}$. Let $\psi_i$ defined by 
    \begin{align*}
        \psi_i(T, y(T)) = -  \int_\Omega T ^2[y^{2i + 1}(T) + 1] dx, \quad i = 1, 2, ..., m.
    \end{align*}
    Take
    \begin{align*}
        a :=  - &\Big[\|u_*\|^3_{L^\infty(Q_{T_*})} + \|u_*\|_{L^\infty(Q_{T_*})}(\| y_*\|^2_{L^\infty(Q_{T_*})} + 1)   + 2\|y_*\|_{L^\infty(Q_{T_*})}\|u_*\|_{L^\infty(Q_{T_*})} \|\widehat y\|_{L^\infty(Q_{T_*})} \\
         &+  3\| u_*\|^2_{L^\infty(Q_{T_*})}\|\widehat u\|_{L^\infty(Q_{T_*})} + \|\widehat u\|_{L^\infty(Q_{T_*})}(\|y_*\|^2_{L^\infty(Q_{T_*})} + 1)  
          \Big]  - \delta 
    \end{align*}
    (more general, $a < 0$ and $|a|$ is large enough). With  $\widehat y$ and  $\widehat u$ as above and $\widehat T = 0$, it's easy to see that condition $(i)$ of $(H6)$ is satisfied. Moreover, we have 
    \begin{align*}
        \psi_i[T_*, T_*] + \psi_{i\zeta}[T_*, T_*]\widehat y(T_*) &=  -  \int_\Omega T_*^2[y_*^{2i + 1}(T_*) + 1] dx - (2i + 1)\int_\Omega T_*^2y_*^{2i}(T) \widehat y(T_*) dx\le  -  T_*^2|\Omega| < 0, 
    \end{align*}
    for $i = 1, 2, ..., m$. Hence $(ii)$ is satisfied with $(0, \widehat y, \widehat u)$.  Since $ y_*,  u_*, \widehat y \ge 0$, $\widehat u \ge \delta$ on $\bar Q_{T_*}$ and definition of number $a$, we have
    \begin{align*}
        a < a + \delta \le g[x, t] + g_y[x, t]\widehat y + g_u[x, t]\widehat u \le - \delta < 0 = b \quad {\rm a.a. \   on} \ Q_{T_*}.  
    \end{align*}
    Hence $g[x, t] + g_y[x, t]\widehat y + g_u[x, t]\widehat u \in {\rm int}(K_{T_*})$ and so $(iii)$ is satisfied with $(0, \widehat y, \widehat u)$. Thus $(H6)$ is valid.}    
    \end{example}

Given a triple $(T, y, u)\in \Phi$, we define the extension of $y$ on the right by setting 
\begin{align*}
    y_e(x, t)=
    \begin{cases}
        y(x, t)\quad &\text{if}\quad (x, t)\in\bar\Omega\times [0, T]\\
        y(x, T)\quad  &\text{if}\quad (x, t)\in\bar\Omega\times (T,  +\infty). 
    \end{cases}
\end{align*} Then $y_e$ is continuous on $\bar\Omega\times [0, +\infty)$. Moreover, for each $T_0\in (T, +\infty)$, $y_e$ is uniform continuous on the compact set $\bar\Omega\times[0, T_0]$. 

\begin{definition}\label{Def-LocalOptim}   A triple $(T_*,  y_*,  u_*) \in \Phi$ is said to be a locally optimal solution of $(P)$ if $T_*>0$ and there exists some $\varepsilon > 0$ such that 
    \begin{align}\label{StrongOptimDef}
     &J(T_*,  y_*,  u_*) \le J(T, y, u) \quad  \forall (T, y, u)\in \Phi\ \text{satisfying}\notag \\ 
     &|T -  T_*| +  \max_{(x, t) \in \bar \Omega \times [0, T_*\vee T]} |y_e(x, t) -  y_{*e}(x, t)| +{\rm esssup}_{(x,t)\in\Omega\times [0, T_*]}|u(x, \frac{Tt}{T_*})-u_*(x, t)| \leq \varepsilon. 
    \end{align} Hereafter $T_*\vee T=\max (T_*, T)$. 
\end{definition} It is noted that when $T=T_*$, Definition \ref{Def-LocalOptim}  becomes a definition for locally optimal solution to the problem with fixed final time $T_*$. 

 Let us denote by $\mathcal C_0[(T_*, y_*, u_*)]$ the set of vectors $(T, y, u)\in \mathbb{R}\times Y_{T_*} \times U_{T_*}$ which satisfy the following conditions:
\begin{itemize}
    \item [$(c_1)$] $\psi_{0T}[T_*, T_*]T +\psi_{0\zeta}[T_*, T_*]y(T_*)  +    \displaystyle \int_{Q_{T_*}}\Big( \frac{T}{T_*} L[x, t] +L_t[x, t]\frac{Tt}{T_*}+ L_y[x, t]y + L_u[x, t]u  \Big) dxdt \leq 0$;

    \item [$(c_2)$]   $\dfrac{\partial y}{\partial t} +  Ay +  \psi_t[x, t] \dfrac{Tt}{T_*}+  \psi_y[x, t]y - u  + \dfrac{T}{T_*}(Ay_* +\psi[x, t]  -u_*) =  0, \quad    y(0)=0$;

    \item [$(c_3)$] $\psi_{iT}[T_*, T_*]T +\psi_{i\zeta}[T_*, T_*]y(T_*) \le 0$  \ \  for $i\in \{1,2,.., m| \psi_i[T_*, T_*]=0\}$;

    \item [$(c_4)$] $g_t[\cdot, \cdot] \dfrac{Tt}{T_*}+ g_y[\cdot, \cdot]y + g_u[\cdot, \cdot]u  \in {\rm cone}\big(K_{T_*} -g[\cdot, \cdot]\big)$.
\end{itemize}
Define $\mathcal{C}[(T_*, y_*, u_*)] = \bar {\mathcal{C}_0[(T_*, y_*, u_*)]}$,  which is the closure of $\mathcal{C}_0[(T_*, y_*, u_*)]$ in $\mathbb{R}\times Y_{T_*} \times U_{T_*}$. The set $\mathcal C[(T_*, y_*, u_*)]$ is called a critical cone to the problem $(P)$. Each vector $d \in \mathcal C[(T_*, y_*, u_*)]$  is called a critical direction to the $(P)$ at $(T_*, y_*, u_*)$. It is clear that $\mathcal C[(T_*, y_*, u_*)]$ is a closed convex cone containing $0$.

We have the following main result on necessary optimality conditions.

\begin{theorem}\label{main-theorem}
Suppose that assumptions $(H1)$-$(H5)$ are satisfied and $(T_*,  y_*,  u_*)\in\Phi$  is a locally strongly optimal solution to $(P)$. Then for each $d=(T, y, u)\in \mathcal{C}_0[(T_*, y_*, u_*) ]$, there exist Lagrange multipliers $\lambda \in \mathbb{R}_+$, $\mu = (\mu_1, \mu_2, ..., \mu_m) \in \mathbb{R}^m$, $\widetilde\varphi \in L^q(Q_{T_*})  \cap L^1(0, T_*; W^{1, 1}_0(\Omega))$, $\widetilde e \in L^1(Q_{T_*})$ and an absolutely continuous function $\widetilde \phi : [0, T_*] \to \mathbb{R}$ satisfying the following conditions:

\noindent $(i)$  (the adjoint equations)  
\begin{align}\label{kq1}
    \begin{cases}
    - \dfrac{\partial \widetilde\varphi}{\partial t} +  A^*\widetilde\varphi + \psi_y[\cdot, \cdot]\widetilde\varphi = -  \lambda  L_y[\cdot, \cdot] - \widetilde  e g_y[\cdot, \cdot]  \quad {\rm in}\  Q_{T_*}, \\
      \widetilde \varphi = 0 \quad {\rm on}\ \Sigma_{T_*}, \\
      \widetilde\varphi(T_*) = -\lambda\psi_{0\zeta}[T_*, T_*]- \sum_{i=1}^m \mu_i \psi_{i\zeta}[T_*, T_*] \quad {\rm in}\quad  \Omega
    \end{cases}
\end{align}
and
\begin{align}
\label{kq2}
    \begin{cases}
        \widetilde \phi'(t) &=  - \displaystyle \int_\Omega\big(\lambda L_t[x, t]  +  \widetilde \varphi \psi_t[x, t] + \widetilde e g_t[x, t]\big) dx \quad {\rm in} \ (0, T_*), \\
        \widetilde\phi(T_*) &= 0
    \end{cases}
\end{align}
where $A^*$ is the adjoint operator of $A$, which is  defined by $A^* \widetilde\varphi = - \sum_{i, j = 1}^N D_i (a_{ij}(x)D_j\widetilde\varphi)$;

\noindent $(ii)$  (optimality condition for $u_*$) 
\begin{align}\label{kq3}
    \lambda  L_u[x, t]  -  \widetilde \varphi(x, t) +  \widetilde e(x, t) g_u[x, t]  =  0 \quad {\rm a.a.} \ (x, t) \in Q_{T_*}; 
\end{align}

\noindent $(iii)$   (optimality condition for $T_*$) 
\begin{align}\label{kq4}  
       \int_0^{T_*} \widetilde \phi(t)dt  =  -  \int_{Q_{T_*}}\big(\lambda L[x, t]+  \widetilde\varphi \big(A y_* +\psi[\cdot, \cdot]-u_*\big)\big) dxdt 
      -   \lambda T_* \psi_{0\zeta}[T_*, T_*]  - T_* \sum_{i=1}^m \mu_i \psi_{i\zeta}[T_*, T_*];
\end{align}

\noindent $(iv)$ (the complementary conditions)
\begin{align}  
    &\ \mu_i \ge 0,\ \ {\rm and} \ \  \mu_i \psi_i [T_*, T_*]  =  0, \quad  i = 1, 2, ..., m,   \label{kq5}  \\
    &\widetilde e(x, t)\in N([a, b], g[x, t])\ \text{a.a.}\ (x, t)\in Q_{T_*};
\end{align}

\noindent $(v)$ (the nonnegative second-order condition)  
\begin{align} \label{NSOC-P}
&D^2_{(T, y, u)}\mathcal{L}^P(T_*, y_*, u_*, \lambda, \widetilde\varphi, \mu, \widetilde\phi, \widetilde e)[d,d]:=  \lambda \Big(\psi_{0TT}[T_*, T_*]T^2 + 2\psi_{0 T \zeta}[T_*, T_*]T y(T_*) +\psi_{0 \zeta\zeta }[T_*, T_*]y(T_*)^2\Big) \notag\\
&+ \lambda \int_{Q_{T_*}}\Big( L_{tt}[x, t](\frac{Tt}{T_*})^2   + 2L_{ty}[x, s]\frac{Tt}{T_*} y(x,t)+  2L_{tu}[x, t]\frac{Tt}{T_*}u(x,t) \Big) dxdt \notag\\
&+ \lambda\int_{Q_{T_*}}\Big( L_{yy}[x, t]y(x,t)^2 + L_{uu}[x, t]u(x,t)^2  +  2L_{yu}[x, t]y(x,t)u(x,t)\Big) dxdt  \nonumber \\
 &+ 2\lambda  \int_{Q_{T_*}}\frac{T}{T_*} \Big(L_t[x, t]\frac{Tt}{T_*}  + L_y[x, t]y(x,t)  +  L_u[x, t]u(x,t)\Big)  dxdt\notag\\
 &+  \int_{Q_{T_*}}\widetilde\varphi(x,t)\Big(\psi_{tt}[x, t](\frac{Tt}{T_*})^2   + 2\psi_{ty}[x, s]\frac{Tt}{T_*} y(x,t)+ \psi_{yy}[x, t]y(x,t)^2 \Big)  dxdt   \nonumber \\
 &+  \int_{Q_{T_*}}  2 \frac{T}{T_*}\widetilde\varphi(x,t) \Big(A y + \psi_t[x, t]\frac{Tt}{T_*}  + \psi_y[x, t]y - u\Big)dxdt \nonumber \\
 &+  \int_{Q_{T_*}} \widetilde e(x, s) \big(g_{tt}[x, t](\frac{Tt}{T_*})^2  +  g_{yy}[x, t]y^2 + g_{uu}[x, t]u^2  + 2g_{ty}[x, t]\frac{Tt}{T_*} y + 2g_{tu}[x, t]\frac{Tt}{T_*} u  +  2g_{yu}[x, t]yu\big)  dxdt  \nonumber \\
 &+  \sum_{i = 1}^m \mu_i \Big[\psi_{iTT}[T_*, T_*]T^2 + 2\psi_{i T \zeta}[T_*, T_*]T y(T_*) +\psi_{i\zeta \zeta}[T_*, T_*] y(T_*)^2\Big]   \geq 0.
     \end{align} In addition, if  $(H6)$ is satisfied then  $\lambda=1$ and 
\begin{align} \label{NSOC-P+}
   \max_{(1, \widetilde\varphi, \mu, \widetilde\phi)\in\Lambda[(T_*, y_*, u_*)]} D^2_{(T, y, u)}\mathcal{L}^P(T_*, y_*, u_*, 1, \widetilde\varphi, \mu, \widetilde\phi, \widetilde e)[d,d]\geq 0\quad \forall d\in \mathcal{C}[(T_*, y_*, u_*)].
\end{align}
\end{theorem} 

To deal with second-order sufficient conditions, we need to enlarge the critical cone $\mathcal C[(T_*, y_*, u_*)]$. For this we define $\mathcal {C}'[(T_*, y_*, u_*)]$ being a set of vectors $(T,  y, u)\in \mathbb{R}\times W_2^{1,1}(0, T_*,H, D) \times L^2(Q_{T_*})$  satisfying the following conditions:
\begin{itemize}
    \item [$(c_1')$] $\psi_{0T}[T_*, T_*]T +\psi_{0\zeta}[T_*, T_*]y(T_*)  +    \displaystyle \int_{Q_{T_*}}\Big( \frac{T}{T_*} L[x, t] +L_t[x, t]\frac{Tt}{T_*}+ L_y[x, t]y + L_u[x, t]u  \Big) dxdt \leq 0$;

    \item [$(c_2')$] $\dfrac{\partial y}{\partial t} +  Ay +  \psi_t[x, t] \dfrac{Tt}{T_*}+  \psi_y[x, t]y - u  + \dfrac{T}{T_*}(Ay_* +\psi[x, t]  -u_*) =  0, \quad    y(0)=0$;

    \item [$(c_3')$] $\psi_{iT}[T_*, T_*]T +\psi_{i\zeta}[T_*, T_*]y(T_*) \le 0$  \ \  for $i\in \{1,2,.., m| \psi_i[T_*, T_*]=0\}$;

    \item [$(c_4')$] $g_t[x, t] \dfrac{Tt}{T_*}+ g_y[x, t]y(x,t) + g_u[x, t]u(x,t)  \in T\big([a, b], g[x, t]\big)$ for a.a. $(x, t)\in Q_{T_*}$.
\end{itemize}

The following theorem gives second-order sufficient optimality conditions for locally optimal solution to $(P)$.

\begin{theorem}\label{main-theorem+} Suppose assumptions $(H1)$-$(H5)$ and $(H7)$, $T_*>0$ and $(T_*,  y_*,  u_*)\in\Phi$.  Assume that there exist Lagrange multipliers $\lambda=1$, $\mu = (\mu_1, \mu_2, ..., \mu_m) \in \mathbb{R}^m$, $\widetilde\varphi \in L^q(Q_{T_*})  \cap L^1(0, T_*; W^{1, 1}_0(\Omega))$, $\widetilde e \in L^1(Q_{T_*})$ and an absolutely continuous function $\widetilde \phi : [0, T_*] \to \mathbb{R}$ such that conditions $(i)-(iv)$ of Theorem \ref{main-theorem} are fulfilled. Furthermore,  assume that the following conditions are fulfilled:

\noindent $(v)$  (the strictly second-order condition) 
    \begin{align} \label{SSOC-P}
    D^2_{(T, y, u)}\mathcal{L}^P(T_*, y_*, u_*, 1, \widetilde\varphi, \mu, \widetilde\phi, \widetilde e)[(\widehat T, \widehat y, \widehat u), (\widehat T, \widehat y, \widehat u)]>0\quad \forall (\widehat T, \widehat y, \widehat u)\in \mathcal{C}'[(T_*, y_*, u_*)]\setminus\{(0,0,0)\};
\end{align}

\noindent $(vi)$ (the Legendre-Clebsch condition) there is a number  $\Lambda_0 > 0$ such that
    \begin{align}\label{L-C-condition-P}
       \ L_{uu}[x, t] + \widetilde{e}(x, t)g_{uu}[x, t] \geq \Lambda_0   \quad {\rm a.a.} \ (x, t) \in Q_{T_*}. 
    \end{align}
    Then  there exist numbers $\varepsilon_0 > 0$ and $\kappa_0 > 0$ such that 
    \begin{align}
        J (T, y, u) \ge  J (T_*, y_*, u_*)   + \kappa_0\Big((T - T_*)^2  +    \frac{1}{T_*}\int_{Q_{T_*}}|u(x, \frac{Tt}{T_*}) - u_*(x, t)|^2 dxdt   \Big) 
    \end{align}
    for all $(T, y, u) \in \Phi \cap N\big((T_*, y_*, u_*), \varepsilon_0\big)$.     In particular, $(T_*, y_*, u_*)$ is a locally optimal solution to  $(P)$. Here
$$
\Phi \cap N\big((T_*, y_*, u_*), \varepsilon_0\big):=\Big\{(T, y, u)\in\Phi: {\rm dist}\big[(T, y, u), (T_*, y_*, u_*)\big]\leq \epsilon_0\Big\},
$$ with 
\begin{align*}
{\rm dist}[(T, y, u), (T_*, y_*, u_*)]:=
|T -  T_*| +  \max_{(x, t) \in \bar \Omega \times [0, T_*\vee T]} |y_e(x, t) -  y_{*e}(x, t)| +  \mathop{\rm esssup}\limits_{(x,t)\in\Omega\times [0, T_*]}|u(x, \frac{Tt}{T_*})-u_*(x, t)|.
\end{align*}
\end{theorem}

\begin{remark} {\rm By using similar arguments as in the proof of \cite[Proposition 2.6]{Kien-2022}, we can show that  if  strictly second-order condition \eqref{SSOC-P} and the Legendre-Clebsch condition \eqref{L-C-condition-P} are satisfied, then  the following strongly second-order condition holds: 
\begin{align*} 
    \exists \gamma_0>0:\  D^2_{(T, y, u)}\mathcal{L}^P(T_*, y_*, u_*, 1, \widetilde\varphi, \mu, \widetilde\phi, \widetilde e)[(\widehat T, \widehat y, \widehat u), (\widehat T, \widehat y, \widehat u)]\geq \gamma_0\big(\widehat{T}^2 + \|\widehat u\|_2^2\big)\  \forall (\widehat T, \widehat y, \widehat u)\in \mathcal{C}'[(T_*, y_*, u_*)].
\end{align*}  Conversely, by a modification of the proof of \cite[Lemma 5.1]{Troltzsch-2000}, we can prove that  if the above strongly second-order condition is satisfied,  $e\in L^\infty (Q_{T_*})$ and $\psi_{i\zeta\zeta}[\cdot, \cdot]=0$, then  the strictly second-order condition and the Legendre-Clebsch condition are fulfilled.  }
\end{remark}

\section{Reduced problem}

Recall that $Q_1=\Omega\times(0, 1)$. We put 
\begin{align*}
    U_1=  L^\infty(Q_1), \ \
    Y_1= \Big\{\zeta \in W^{1, 1}_2(0, 1; D, H) \cap W^{2,1}_p(Q_1)| \  A\zeta \in L^p(Q_1)\Big\}.
\end{align*} Then $Y_1$ is a Banach space under the graph norm
\begin{align*}
    \|\zeta\|_{Y_1}:= \|\zeta\|_{W^{1, 1}_2(0, 1; D, H)} + \|\zeta\|_{W^{2,1}_p(Q_1)} +   \big\|A\zeta\big\|_{L^p(Q_1)}.
\end{align*}

To establish optimality conditions for time-optimal problem $(P)$, we need to change  variables to transform $(P)$ into a  problem with a fixed terminal time. Taking any  $T>0$, we define a function $\xi :[0, 1]\to [0, T]$ by setting 
\begin{align*}
    \xi(s)=\int_0^sTd\tau=Ts. 
\end{align*}
Let
\begin{align*}
    \zeta(x, s):= y(x, \xi(s)), \ \  v(x, s):= u(x, \xi(s)). 
\end{align*} By changing variable $t=\xi(s)=Ts$ in $(P)$,  
 we obtain the following  optimal control problem with fixed final time: 
\begin{align*}
(P_1)\quad \quad
\begin{cases}
    &\widehat J(\xi, T,  \zeta, v) := \psi_0(T, \zeta(1))+ \displaystyle\int_0^1 \int_\Omega TL(x, \xi(s),  \zeta(x, s), v(x, s)) dxds   \to  \inf \\
    &{\rm s.t.}\notag\\
    &\dfrac{\partial \zeta}{\partial s} + TA\zeta + T\psi(x, \xi, \zeta) = Tv \quad \text{in} \ Q_1, \quad  \zeta(x, s)=0 \quad  \text{on} \ \Sigma_1 = \Gamma\times [0, 1], \\
    &\zeta(0)=y_0 \quad \text{in} \ \Omega, \\
    &\xi = \displaystyle \int_0^{(\cdot)}Td\tau,  \\
    &\psi_i(T, \zeta(1)) \le 0,   \quad  i=1,2,..., m, \\
    &a\leq g(x, \xi(s), \zeta(x, s), v(x, s)) \le b \quad {\rm a.a.} \ x \in \Omega, \ \forall s  \in [0, 1]. 
    \end{cases}
\end{align*} We shall denote by  $\Phi_1$ the feasible set of $(P_1)$ and put
\begin{align*}
    Z = C([0,1], \mathbb{R}) \times  \mathbb{R}\times Y_1\times U_1.    
\end{align*} Given a vector $z=(\xi, T, \zeta, v)\in Z$, we define
\begin{align*}
\|z\|_*:= \|\xi\|_{C([0,1], \mathbb{R})}+|T| +\|\zeta\|_{C(\bar Q_1)}+\|v\|_{U_1},\ \
\|z\|_Z:= \|\xi\|_{C([0,1], \mathbb{R})}+|T| +\|\zeta\|_{Y_1}+\|v\|_{U_1}. 
\end{align*} Given $z_0\in Z$ and $r>0$, we denote by $B_{*Z}(z_0,r)$ and $B_Z(z_0, r)$ the balls center at $z_0$ and radius $r$ in norm $\|\cdot\|_*$ and $\|\cdot\|_Z$, respectively. 
\begin{definition} Vector $z_*=(\xi_*, T_*, \zeta_*, v_*)\in \Phi_1$ is a locally optimal solution to $(P_1)$ if there exists $\delta>0$ such that 
\begin{align}\label{DfLocalOptimSol}
\widehat{J}(z_*)\leq \widehat{J}(z) \quad \forall z\in \Phi_1\ \text{satisfying}\ \|z-z_*\|_*\leq \delta.  
\end{align}     
\end{definition} Since $Y_1\hookrightarrow C(\bar Q_1)$, there exists $\gamma>0$ such that $\|z\|_*\leq \gamma \|z\|_Z$ for all $z\in Z$.  Therefore,  if $\|z-z_*\|_Z\leq\frac{\delta}{\gamma}$, then $\|z-z_*\|_*\leq\delta$. Consequently,  if $z_*$ is a locally optimal solution to $(P_1)$ in norm $\|\cdot\|_*$, then it is also locally optimal solution to $(P_1)$  in norm $\|\cdot\|_Z$. 
 
\medskip

The following propositions give relations between optimal solutions of $(P)$ and $(P_1)$.

\begin{proposition} \label{relationPandP1}
 Suppose that $ (T_*, y_*,  u_*)\in  \Phi$  is a locally  optimal solution to $(P)$. Let   $\xi_*(s)=T_*s $ for $s\in [0,1]$ and    
    \begin{align*} 
     \zeta_* (x, s) := y_*(x, \xi_*(s)),\   v_*(x, s) := u_*(x,  \xi_*(s)). 
\end{align*} 
Then the vector  $( \xi_*, T_*,  \zeta_*,  v_*)$ is a locally optimal solution to $(P_1)$ and $\widehat J(\xi_*, T_*,  \zeta_*,  v_*) = J(T_*, y_*, u_*)$.  
\end{proposition}
\begin{proof}  By definition, $T_*>0$ and there exists $\epsilon > 0$  such that \eqref{StrongOptimDef} is valid. Fix a number $T_0> T_*$.  By the uniform continuity of $y_{*e}$ on $\bar\Omega\times [0,  T_0]$, there exists $\delta \in (0, \frac{\varepsilon}{4})$ such that 
\begin{align}
\label{0.1}
| y_{*e}(x, t_1)-y_{*e}(x, t_2)|\le \frac{\epsilon}{4}  \quad  \forall t_1, t_2\in [0,  T_0] \ \  \text{satisfying}\ \  |t_1-t_2| < \delta. 
\end{align} We can choose $\delta>0$ small enough such that $T>0$ whenever $|T-T_*|<\delta$.
Clearly, $(\xi_*, T_*,  \zeta_*, v_*) \in \Phi_1$.  Let $(\xi, T, \zeta, v)\in\Phi_1$ such that 
\begin{align}
\label{0.2}
 \|\xi-\xi_*\|_{C([0,1], \mathbb{R})} + |T-T_*|   + \|\zeta-\zeta_*\|_{C(\bar Q_1)} + \|v - v_*\|_{L^\infty(\bar Q_1)}   <\min(\delta, T_0-T_*). 
\end{align}  
We define $ y(x, t)=  \zeta(x, \xi^{-1}(t)),\   u(x, t)=  v(x, \xi^{-1}(t))$, where  $\xi^{-1}(t)=\frac{t}{T}$. It is easy to check that $(T, y, u) \in \Phi$. By  (\ref{0.2}) we have 
\begin{align*}
    |T - T_*|+ \|\xi-\xi_*\|_{C([0,1], \mathbb{R})} \leq \delta \leq \frac{\varepsilon}{4}.
\end{align*} 
Without loss of generality, we may assume that $T> T_*$. Then from condition $T- T_*< T_0- T_*$, we have $T<T_0$ and  $[0, T\vee T_*]=[0, T]$. Therefore, for $(x, t)\in\Omega\times [0, T]$, we have  
\begin{align}
    \label{0.4}
    |y_e(x, t) - y_{*e}(x, t)|  &=  \Big|y(x, t) - y_{*e}\big(x, t)\big)\Big| \nonumber \\
    &=   \Big|\zeta\big(x,  \xi^{-1}(t)\big) -y_{*e}(x,\xi(\xi^{-1}(t))\Big|   \nonumber \\
    &\le   \Big|\zeta\big(x,  \xi^{-1}(t)\big) -  y_*\big(x, \xi_* \big(\xi^{-1}(t)\big)\big)\Big|  + 
    \Big| y_*\big(x, \xi_* \big(\xi^{-1}(t)\big)\big)  -  y_{*e}\big(x, \xi \big(\xi^{-1}(t)\big)\big)  \Big|   \nonumber \\
    &=   \Big|\zeta\Big(x,  \xi^{-1}(t)\Big) - \zeta_*\Big(x,  \xi^{-1}(t) \Big)\Big|  + 
    \Big|y_{*e}\big(x, \xi_* \big(\xi^{-1}(t)\big)\big)   -  y_{*e}\big(x, \xi \big(\xi^{-1}(t)\big)\big)  \Big|.   
\end{align}
 Note that $\xi^{-1}(t) \in [0, 1]$ for all $t \in [0, T]$. By (\ref{0.2}) we deduce that 
\begin{align}
\label{0.41}
    \Big|\zeta\Big(x,  \xi^{-1}(t)\Big) -  \zeta_*\Big(x,  \xi^{-1}(t) \Big)\Big| \le  \|\zeta- \zeta_*\|_{C(\bar Q_1)} \le \delta \le \frac{\varepsilon}{4}
\end{align}
for all $(x, t) \in \bar \Omega \times [0,  T]$. Moreover, for all $t \in [0,  T]$  we have $t_1 = \xi_* \big(\xi^{-1}(t)\big) \in [0, T_*] \subset [0, T]$, $t_2 = \xi \big(\xi^{-1}(t)\big) = t \in [0,  T]$ and
\begin{align*} 
    |t_1 - t_2| \le \|\xi-\xi_*\|_{C([0,1], \mathbb{R})} \le \delta. 
\end{align*}
Therefore, we have from (\ref{0.1}) that  
\begin{align*}
    \Big|y_{*e}\Big(x, \xi_* \big(\xi^{-1}(t)\big)\Big)   -  y_{*e}\Big(x, \xi \big(\xi^{-1}(t)\big)\Big)  \Big|  \le \frac{\varepsilon}{4}\quad \forall (x, t) \in \bar \Omega \times [0, T].
\end{align*} Combining this with  (\ref{0.4}) and (\ref{0.41}), we get 
\begin{align*}    
    |y_e(x, t) -  y_{*e}(x, t)|  \le \frac{\varepsilon}{4} + \frac{\varepsilon}{4} = \frac{\varepsilon}{2}\quad \forall (x, t) \in \bar \Omega \times [0, T].
\end{align*}
 This implies that $ 
    \mathop {\rm max}\limits_{(x, t) \in \bar \Omega \times [0, T\vee T_*]} |y_e(x, t) - y_{*e}(x, t)|  \le \frac{\varepsilon}{2}$.  Hence 
\begin{align} \label{0.9}
    |T - T_*|  +  \mathop {\rm max}\limits_{(x, t) \in \bar \Omega \times [0, T\vee T_*]} |y_e(x, t) -  y_{*e}(x, t)|  \le \frac{\varepsilon}{4} +  \frac{\varepsilon}{2}  = \frac{3\varepsilon}4. 
\end{align}  On the other hand 
\begin{align*}
    {\rm esssup}_{t\in [0, T_*]}| u(\frac{Tt}{T_*})-u_*(t)|={\rm essup}_{t\in [0, T_*]}| v(x, \frac{t}{T_*})-v_*(\frac{t}{T_*})|\leq\|v-v_*\|_{L^\infty(Q_1)}\leq \delta<\frac{\epsilon}4.
\end{align*} Combining this with \eqref{0.9}, yields 
$$
|T - T_*|  +  \mathop {\rm max}\limits_{(x, t) \in \bar \Omega \times [0, T\vee T_*]} |y_e(x, t) -  y_{*e}(x, t)| +  {\rm esssup}_{t\in [0, T_*]}| u(\frac{Tt}{T_*})-u_*(t)|<\epsilon. 
$$
From this and \eqref{StrongOptimDef}, we have $ J (T_*,  y_*,  u_*) \le J (T, y, u)$. Hence
\begin{align*}
    \widehat J(\xi_*, T_*, \zeta_*, v_*)  =  J (T_*, y_*,  u_*) \le J (T, y, u) = \widehat J(\xi, T, \zeta, v).
\end{align*} The proof of the proposition is complete.
\end{proof}

\begin{proposition}\label{relationPandP1+}
 Suppose that $T_*>0$ and  $(\xi_*, T_*, \zeta_*, v_*)\in \Phi_1$ is a locally optimal solution to $(P_1)$. Let 
    \begin{align*} 
   y_*(x, t):=\zeta_* (x, \xi_*^{-1}(t)),\   u_*(x, t):= v_*(x, \xi_*^{-1}(t)). 
\end{align*} 
Then the vector $(T_*,  y_*, u_*)$  is  a locally  optimal solution to $(P)$ and  $J(T_*, y_*, u_*) = \widehat J(\xi_*, T_*, \zeta_*,  v_*)$.  
\end{proposition}
\begin{proof} Let $\delta>0$ such that \eqref{DfLocalOptimSol} is valid. Fix a number  $T_0> T_*$. Since $y_{*e}$ is uniform continuous on $\bar\Omega\times [0,  T_0]$, there exists $\epsilon\in (0, \delta)$ such that 
$|y_{*e}(x, t)-y_{*e}(x, t')|< \frac{\delta}3$ for all $t, t'\in [0, T_0]$ satisfying $|t-t'|<\epsilon$. We can choose $\epsilon>0$ small enough so that $T>0$ whenever $|T-T_*|<\epsilon$.  We now take $(T, y, u)\in\Phi$  satisfying
$$
|T-T_*| + \max_{(x, t)\in\bar\Omega\times [0, T \vee T_*]}|y_e(x,t)- y_{*e}(x, t)| + {\rm esssup}_{(x, t)\in \Omega\times[0, T_*]}|u(x, \frac{Tt}{T_*})- u_*(x, t)| <  \min(\frac{\epsilon}6, T_0- T_*). 
$$
Without loss of generality, we can assume that $T>T_*$. Then we have $T<T_0$.  Let $\xi(s)=Ts$ and  $\zeta(x, s)=y(x, \xi(s)), v(x, s)=u(x, \xi(s))$. Then $(\xi, T, \zeta, v)\in\Phi_1$. Note that 
\begin{align}\label{Estim1}
  \|\xi -\xi_*\|_{C([0,1], \mathbb{R})}  &=\max_{s\in [0, 1]}| Ts-T_*s|=|T- T_*|<\frac{\epsilon}6.   
\end{align} 
From this and uniform continuity of $y_{*e}$,  we have for all $(x, s)\in \bar\Omega\times [0, 1]$ that
 \begin{align*}
|\zeta(x, s)-\zeta_*(x, s)| &=|y(x, \xi(s))-  y_*(x, \xi_*(s))|\\
&\leq |y(x, \xi(s))- y_{*e}(x, \xi(s))| +|y_{*e}(x, \xi(s))- y_{*e}(x, \xi_*(s))|\\
&\leq \frac{\epsilon}6 + \frac{\delta}3\leq \frac{\delta}2.
\end{align*} This implies that 
\begin{align}\label{Estim2}
    \|\zeta-\zeta_*\|_{C(\bar Q_1)}\leq \frac{\delta}2. 
\end{align}
Also, we have 
\begin{align*}
  {\rm esssup}_{(x, s)\in\Omega\times[0, 1]}|v(x, s)- v_*(x, s)| &={\rm esssup}_{(x, s)\in\Omega\times[0, 1]}|u(x, Ts)-u_*(x, T_*s)|\\
  &={\rm esssup}_{(x, t)\in\Omega\times[0, T_*]}|u(x, \frac{Tt}{T_*})- u_*(x,t))|\leq \frac{\epsilon}6< \frac{\delta}6. 
\end{align*} Hence 
$ \|v-v_*\|_{L^\infty(Q_1)}\leq \frac{\delta}6$.   Combining this with \eqref{Estim1} and   \eqref{Estim2}, we obtain 
$$
\|\xi -\xi_*\|_{C([0,1], \mathbb{R})} +  \|\zeta-\zeta_*\|_{C(\bar Q_1)} + \|v-v_*\|_{L^\infty(Q_1)} +|T-T_*|<\frac{\delta}6 +\frac{\delta}2 +\frac{\delta}6 +\frac{\delta}6=\delta.
$$ Since $(\xi_*, T_*, \zeta_*, v_*)$ is a locally optimal solution to $(P_1)$, we have  
$$
J(T, y, u)=\widehat J( \xi, T, \zeta, v)\geq \widehat J(\xi_*, T_*, \zeta_*, v_*)= J(T_*, y_*, u_*). 
$$ The proof of the proposition  is complete. 
\end{proof}

\medskip

To reduce  $(P_1)$ to a mathematical programming problem, we put
\begin{align*}
E = L^p(Q_1) \times \Big(H_0^1(\Omega)\cap W^{2-\frac{2}p, p}(\Omega)\Big) \times C([0, 1], \mathbb R),  \quad
W= L^\infty(Q_1).
\end{align*} 
and define mappings $F: Z \to E$, $H : Z \to \mathbb{R}^m$ and $G : Z \to W$ by setting
\begin{align*}
    & F(\xi, T, \zeta, v)= \Big(F_1(\xi, T, \zeta, v), F_2(\xi, T, \zeta, v), F_3(\xi, T, \zeta, v)\Big) \nonumber \\
    &\quad\quad\quad\quad\quad = \Big(\frac{\partial \zeta}{\partial s} + TA\zeta + T\psi(\cdot, \xi, \zeta) - Tv,\  \zeta(0)- y_0,\ \xi -\int_0^{(\cdot)}Td\tau \Big),\\
    &H(\xi, T, \zeta, v)=(H_1(\xi, T, \zeta, v),..., H_m(\xi, T, \zeta, v))=\big(\psi_1(\xi(1), \zeta(1)),..., \psi_m(\xi(1), \zeta(1))\big), \\
    &G(\xi,T,  \zeta, v)= g(\cdot, \xi, \zeta, v).
\end{align*}  
By definition of space $Y_1$, if $(\xi, T, \zeta, v) \in Z$ then $\frac{\partial \zeta}{\partial s}, A\zeta \in L^p(Q_1)$,  $\psi(\cdot, \xi, \zeta) \in L^\infty(Q_1) \subset L^p(Q_1)$ (since $\xi \in C([0, 1], \mathbb R)$,  $\zeta \in C(\bar Q_1)$) and $v \in L^\infty(Q_1) \subset L^p(Q_1)$. Hence
\begin{align*}
    \frac{\partial \zeta}{\partial s} + TA\zeta + T\psi(\cdot, \xi, \zeta) - Tv \in L^p(Q_1) 
\end{align*}
and so  $F_1$ is well defined.  $F_2$ is also well defined. Indeed,  from $(H3)$ we have  $y_0 \in H_0^1(\Omega)\cap W^{2-\frac{2}p, p}(\Omega)$. Since $\zeta \in Y_1$, $\zeta \in W^{1, 1}_2(0, 1; D, H) \cap W^{2,1}_p(Q_1)$. Since $W^{1, 1}_2(0, 1; D, H) \hookrightarrow C([0, 1], V)$, $\zeta(0) \in H_0^1(\Omega)$. Since $\zeta \in  W^{2,1}_p(Q_1)$,  \cite[Lemma 3.4, p. 82]{Ladyzhenskaya} implies that $\zeta(0) \in W^{2-\frac{2}p, p}(\Omega)$. So  $\zeta(0) \in H_0^1(\Omega)\cap W^{2-\frac{2}p, p}(\Omega)$. Consequently,  $\zeta(0)-y_0\in H_0^1(\Omega)\cap W^{2-\frac{2}p, p}(\Omega)$.  Therefore,  $F$ is well defined.  Similarly, $H$ and $G$ are well defined.  We now rewrite problem $(P_1)$ in the  form of $(MP)$:
\begin{equation}
\text{(MP1)}\quad\quad 
\begin{cases}
    &\widehat J(\xi, T, \zeta, v)    \to  \inf \\
    &{\rm s.t.}\notag\\
    &F(\xi, T, \zeta, v) = 0, \\
    &H(\xi, T, \zeta, v) \le 0, \\
    &G(\xi, T, \zeta, v) \in K_1, 
\end{cases}
\end{equation} where 
\begin{align*}
    K_1:=\{ u\in L^\infty(Q_1)| a\leq u(x, s)\leq b\ {\rm a.a.}\ (x, s)\in Q_1\}.
\end{align*}
Next, we will  apply Proposition \ref{KeyProp} for the problem (MP1) to derive necessary optimality conditions.

As before, given a vector  $z_*=(\xi_*, T_*, \zeta_*, v_*) \in \Phi_1$, the symbols  $L[x, s]$,  $L_t[x, s]$, $L_y[x, s]$, etc., stand for   $L(x, \xi_* (s), \zeta_*(x, s), v_*(x, s))$,  $L_t(x, \xi_* (s), \zeta_*(x, s), v_*(x, s))$, $L_y(x, \xi_* (s), \zeta_*(x, s), v_*(x, s))$, etc., respectively. 

\medskip

The following lemmas shows that $(P_1)$ satisfies all conditions of Proposition \ref{KeyProp}.
\begin{lemma}\label{lemma1}
Suppose that assumptions $(H2)$ and  $(H4)$ are satisfied. Then the mappings $\widehat J, F, H$ and $G$ are of class $C^2$ around $z_*$.
\end{lemma}
\begin{proof}  Let  $B_Z(z_*, \epsilon)$ be a neighborhood of $z_*$ in $Z$. By assumptions $(H2)$,  $(H4)$ and some standard arguments, we can show that $\widehat J, F, H$ and $G$ are of class $C^2$ in  $B_Z(z_*, \epsilon)$. Namely, given a vector $\widehat z =  (\widehat\xi, \widehat T, \widehat \zeta, \widehat v)\in B_Z(z_*, \epsilon)$, we have 
    \begin{align}
         D\widehat J(\widehat z)[(\xi, T,  \zeta, v)] &= \psi_{0T}(\widehat T, \widehat\zeta(1))T+\psi_{0\zeta}(\widehat T, \widehat\zeta(1))\zeta(1)  \nonumber \\
    &+\int_{Q_1} \Big(T L(x, \hat\xi, \hat\zeta, \hat v) +  \widehat T[L_t(x, \hat\xi, \hat\zeta, \hat v)\xi  +  L_y(x, \hat\xi, \hat\zeta, \hat v)\zeta + L_u(x, \hat\xi, \hat\zeta, \hat v)v]\Big)  dxds, 
    \end{align}
    \begin{align}
      &DF(\widehat z)  =   \Big(DF_1(\widehat z),\  DF_2(\widehat z), \ DF_3(\widehat z) \Big), \label{df1} \\
      &DF_1(\widehat z)[(\xi, T,  \zeta, v)] = \frac{\partial \zeta}{\partial s} +\widehat T A\zeta +\widehat T \psi_y(\cdot, \widehat \xi, \widehat\zeta)\zeta  +\widehat T \psi_t(\cdot, \widehat \xi, \widehat\zeta)\xi -\widehat T v+ T(A\widehat\zeta +\psi(\cdot, \widehat \xi, \widehat\zeta)-\widehat v), \label{df11} \\
      &D F_2 (\widehat z)[(\xi, T,  \zeta, v)]= \zeta(0), \label{df21}\\
      &D F_3(\widehat z)[(\xi, T,  \zeta, v)]= \xi - \int_0^{(\cdot)} Td\tau, \label{df31}\\
      &DH(\widehat z)  =  \Big(DH_1(\widehat z), \ DH_2(\widehat z),...,   DH_m(\widehat z) \Big), \\
      &DH_i(\widehat z)[(\xi, T,  \zeta, v)]=\psi_{i T}(\hat T, \hat\zeta(1)) T+\psi_{i\zeta}(\hat T, \hat\zeta(1))\zeta(1),\  i = 1, 2,..., m,  \\
      &DG(\widehat z)[(\xi, T,  \zeta, v)]= g_y(\cdot, \hat\xi, \hat\zeta, \hat v)\zeta + g_u(\cdot, \hat\xi, \hat\zeta, \hat v)v+ g_t(\cdot, \hat\xi, \hat\zeta, \hat v)\xi
    \end{align} 
     and     
    \begin{align}
    \label{DJ2}
       & D^2\widehat J(\widehat z)(\xi, T,  \zeta, v)^2 =\psi_{0TT}(\widehat T, \widehat\zeta(1))T^2+\psi_{0\zeta\zeta}(\widehat T, \widehat\zeta(1))\zeta^2(1)+ 2\psi_{0T\zeta}(\widehat T, \widehat\zeta(1))T\zeta(1) \notag \\
        &+\int_{Q_1} \widehat T \Big[L_{tt}(x, \hat\xi, \hat\zeta, \hat v(T,  \zeta, v, \xi))\xi^2  +  L_{yy}(x, \hat\xi, \hat\zeta, \hat v)\zeta^2 + L_{uu}(x, \hat\xi, \hat\zeta, \hat v)v^2 \Big]  dxds \nonumber \\
        &+ \int_{Q_1} 2 \widehat T \Big[L_{ty}(x, \hat\xi, \hat\zeta, \hat v)\xi \zeta + L_{tu}(x, \hat\xi, \hat\zeta, \hat v)\xi v  +  L_{yu}(x, \hat\xi, \hat\zeta, \hat v)\zeta v\Big]  dxds,   \nonumber \\
        &+ \int_{Q_1} 2 T \big[L_t(x, \hat\xi, \hat\zeta, \hat v)\xi  + L_y(x, \hat\xi, \hat\zeta, \hat v)\zeta  +  L_u(x, \hat\xi, \hat\zeta, \hat v)v\big]  dxds,\\ 
        &D^2F(\widehat z)(\xi, T,  \zeta, v)^2  \nonumber \\ 
        &=   \Big(\widehat T \psi_{tt}(\cdot,\widehat \xi, \widehat \zeta)\xi^2  +  \widehat T \psi_{yy}(\cdot,\widehat \xi, \widehat \zeta)\zeta^2  
        +  2 \widehat T \psi_{ty}(\cdot,\widehat \xi, \widehat \zeta)\xi \zeta
        + 2 T \big(A \zeta   + \psi_t(\cdot,\widehat \xi,\widehat \zeta)\xi  +  \psi_y(\cdot,\widehat \xi,\widehat \zeta)\zeta - v\big), \  0, \  0\Big),\\
        &D^2H(\widehat z)(\xi, T,  \zeta, v)^2  =  \big(D^2H_1(\widehat z), \ D^2H_2(\widehat z), ...,  D^2H_m(\widehat z) \big)(\xi, T,  \zeta, v)^2   \\
    &D^2H_i(\widehat z)(\xi, T,  \zeta, v)^2=\psi_{iTT}(\hat T, \hat\zeta(1))T^2 + 2\psi_{i T \zeta}(\hat T, \hat\zeta(1))T\zeta(1)  +\psi_{i\zeta \zeta}(\hat T, \hat\zeta(1))\zeta^2(1),\\  
    & D^2G(\widehat z)(\xi, T,  \zeta, v)^2= g_{tt}(\cdot, \hat\xi, \hat\zeta, \hat v)\xi^2  +  g_{yy}(\cdot, \hat\xi, \hat\zeta, \hat v)\zeta^2 + g_{uu}(\cdot, \hat\xi, \hat\zeta, \hat v)v^2\notag \\     &\quad\quad\quad\quad\quad\quad\quad\quad\quad+2g_{ty}(\cdot, \hat\xi, \hat\zeta, \hat v)\xi \zeta + 2g_{tu}(\cdot, \hat\xi, \hat\zeta, \hat v)\xi v  +  2g_{yu}(\cdot, \hat\xi, \hat\zeta, \hat v)\zeta v. \label{DG2}
    \end{align} for all $(\xi, T,  \zeta, v) \in Z$. The lemma is proved. 
\end{proof}

\begin{lemma} \label{Lemma-ClosedRange}
    The operator $DF(\xi_*, T_*, \zeta_*, v_*): Z \to E$ is onto. In particular,  $DF(\xi_*, T_* \zeta_*, v_*)Z$ is closed.       
\end{lemma}
\begin{proof}  
    Taking any $(\phi_0, \zeta_0, \xi_0)\in E$, we consider the following equation on $Z$
    \begin{align*}
    DF(\xi_*, T_*, \zeta_*, v_*)[(\xi, T, \zeta, v)]=(\phi_0, \zeta_0, \xi_0).
    \end{align*} 
    By (\ref{df1})-(\ref{df31}), it is equivalent to the system of equations:
\begin{align*}
 &\frac{\partial \zeta}{\partial s} +T_* A\zeta +T_* \psi_y(\cdot, \xi_*, \zeta_*)\zeta  +T_* \psi_t(\cdot, \xi_*, \zeta_*)\xi -T_* v+ T(A\zeta_* +\psi(\cdot, \xi_*, \zeta_*)-v_*)=\phi_0,\\
& \zeta(0)=\zeta_0,\\
&\xi(s) -Ts=\xi_0(s).
\end{align*} Taking $T=T_*$ and $v=0$, we get $\xi=\xi_0+T_*s$. Consider parabolic equation
\begin{align}
 &\frac{\partial \zeta}{\partial s} +T_* A\zeta +  T_* \psi_y(\cdot, \xi_*, \zeta_*)\zeta   =\phi_0  -  T_*\psi_t(\cdot, \xi_*, \zeta_*)\xi  - T_*(A\zeta_* +\psi(\cdot, \xi_*, \zeta_*)   -v_*), \label{onto1}\\
& \zeta(0)=\zeta_0.  \label{onto2}
\end{align} 
By $(H2)$, $T_*\psi_t(\cdot, \xi_*, \zeta_*)\xi \in L^\infty(Q_1)$. Also we have   $T_*(A\zeta_* +\psi(\cdot, \xi_*, \zeta_*)-v_*)=-\frac{d}{ds}\zeta_*\in L^p(Q_1)$. Hence $\phi_0  -  T_*\psi_t(\cdot, \xi_*, \zeta_*)\xi  - T_*(A\zeta_* +\psi(\cdot, \xi_*, \zeta_*)   -v_*)\in L^p(Q_1)$. By Lemma \ref{Lemma-LinearizedEq},  equations \eqref{onto1}-\eqref{onto2} has a solution $\zeta\in Y_1$. Therefore, $DF(\xi_*, T_*, \zeta_*, v_*)$ is onto.  The lemma is proved.  
\end{proof}

The following lemma shows that the  Mangasarian-Fromowitz condition is valid. 

\begin{lemma}\label{lemma-Robinson}
  If  $(H6)$ is satisfied, then $(A4)$  is also fulfilled at $z_*  = (\xi_*, T_*, \zeta_*, v_*)$. 
\end{lemma}
\begin{proof} Recall that 
 $ \xi_* (s)= T_*s,\  \zeta_* (x,s)= y_*(x, T_*s),\   v_*(x,s)= u_*(x, T_*s).$
   By Lemma \ref{Lemma-ClosedRange}, condition $(i)$ of $(A4)$ is satisfied. Let $(\widetilde T, \widetilde y, \widetilde u)$ satisfy $(H6)$.  By defining 
    $$
    \widetilde\xi (s)=\widetilde{T}s,\  \widetilde\zeta (x,s)=\widetilde y(x, T_*s),\  \widetilde v(x,s)=\widetilde u(x, T_*s),
    $$ we see that $(H6)$ is equivalent to:
\begin{align*}
(H6)' 
    \begin{cases}
       & \frac{d\widetilde\zeta}{dt} +T_*(A\widetilde\zeta +\psi_t(\cdot, \xi_*, \zeta_*)\widetilde \xi  +\psi_y(\cdot, \xi_*, \zeta_*)\widetilde \zeta - \widetilde u) +\widetilde T( A\zeta_* + \psi(\cdot, \xi_*, \zeta_*) - v_*) = 0, \ \    \widetilde \zeta (0) = 0, \\
       &\psi_i(\xi_*(1), \zeta_*(1)) +\psi_{it}(\xi_*(1), \zeta_*(1))\widetilde \xi(1) + \psi_{iy}(\xi_*(1), \zeta_*(1))\widetilde\zeta(1)<0, \\
       & g[x, s] +g_t[x,s]\widetilde \xi+ g_{y}[x,s]\widetilde\zeta + g_{u}[x,s]\widetilde v \in {\rm int}K_1.
    \end{cases}
\end{align*} This means that   
    \begin{align*}
    \begin{cases}
        &DF(z_*)(\widetilde\xi, \widetilde T, \widetilde \zeta, \widetilde v)=0,\\
        &H_i(z_*) + DH_i(z_*)(\widetilde\xi, \widetilde T, \widetilde \zeta, \widetilde v) < 0, \quad i =  1, 2, ..., m, \\
        &G(z_*) + DG(z_*)(\widetilde\xi, \widetilde T, \widetilde \zeta, \widetilde v, ) \in {\rm int}K_1. 
    \end{cases}
    \end{align*} Hence $(A4)$ is satisfied.  The lemma is proved. 
\end{proof}

Let us  denote by $\mathcal{C}_{1, 0}[z_*]$ the critical cone of vectors  $(\xi, T, \zeta, v)\in Z$ satisfying conditions:

\begin{itemize}
    \item [$(b_1)$] $\psi_{0T}[T_*,1]T+ \psi_{0\zeta}[T_*, 1]\zeta(1)+ \displaystyle \int_{Q_1}\Big(T L[x, s] +  T_*\big(L_t[x, s]\xi  +  L_y[x, s]\zeta + L_u[x, s]v\big)\Big) dxds \leq 0$;

    \item [$(b_2)$] $\dfrac{\partial \zeta}{\partial s} +T_*\big( A\zeta + \psi_t[\cdot, \cdot]\xi  +  \psi_y[\cdot, \cdot]\zeta -v\big)+ T(A\zeta_* +\psi[\cdot, \cdot]-v_*)=0, \quad    \zeta(0)=0$;

    \item [$(b_3)$] $\xi(s)=Ts$ for all $s\in [0,1]$;

    \item [$(b_4)$] $\psi_{iT}[T_*, 1])\xi(1)+\psi_{i\zeta}[T_*, 1]\zeta(1)\leq 0$  \ \  for $i\in \{1,2,.., m| \psi_i[T_*, 1]=0\}$;

    \item [$(b_5)$] $ g_t[\cdot, \cdot]\xi  +  g_y[\cdot, \cdot]\zeta + g_u[\cdot, \cdot]v \in{\rm cone}\big(K_1-g[\cdot, \cdot]\big)$.   
\end{itemize}
 The closure of $\mathcal{C}_{1, 0}[z_*]$ in $Z$  is denoted by $\mathcal C_1[z_*]$. It is also called a critical cone to $(P_1)$ at $z_*$. 
 
Let us define a Lagrange function associated to $(P_1)$:
     \begin{align*}
         \mathcal{L} : Z \times \mathbb{R}_+ \times L^q(Q_1) \times (H_0^1(\Omega)\cap W^{2-\frac{2}p, p}(\Omega))^* \times C([0, 1], \mathbb R)^* \times \mathbb{R}^m \times L^\infty(Q_1)^*     \to \mathbb{R} 
     \end{align*} by setting 
    \begin{align*}
&\mathcal{L}(z, \lambda, \phi_1, \phi_2^*, \nu_1^*, \mu, w^*)= \notag\\
&\lambda \psi_0(T, \zeta(1))+\lambda\int_{Q_1}TL(x, \xi, \zeta, v)dxds  +   \int_{Q_1}\phi_1 \Big(\dfrac{\partial \zeta}{\partial s} + T A\zeta + T\psi(x, \xi, \zeta) - Tv\Big) dxds  \nonumber \\   
&+\langle \phi_2^*, \zeta (0)-y_0\rangle   +\langle \nu_1^*, \xi-\int_0^{(\cdot)} T d\tau\rangle +   \sum_{i=1}^m \mu_i \psi_i(T, \zeta(1)) + \langle w^*, g(\cdot, \xi, \zeta, v)\rangle,\ z=(\xi, T,  \zeta, v)\in Z.   
\end{align*} 

The following result gives necessary optimality conditions for $(P_1)$.

\begin{proposition}\label{Pro-KKT-P1}
Suppose assumptions $(H1)-(H5)$, $T_*>0$ and $z_*=(\xi_*, T_*, \zeta_*, v_*)\in\Phi_1$ is a locally optimal solution to $(P_1)$. Then for each $z=(\xi, T, \zeta, v)\in\mathcal{C}_{1,0}[z_*]$, there exist Lagrange multipliers $\lambda \in \mathbb R_+$, $\mu = (\mu_1, \mu_2, ..., \mu_m) \in \mathbb R^m$, $\varphi \in L^q(Q_1)  \cap L^1(0, 1; W^{1,1}_0(\Omega))$, $e \in L^1(Q_1)$ and an absolutely continuous function $\phi : [0, 1] \to \mathbb{R}$ not all are zero and  satisfy the following conditions:

\noindent $(i)$  (the adjoint equations)  
\begin{align}\label{cm0.1} 
    \begin{cases}
      - \dfrac{\partial \varphi}{\partial s} + T_* A^*\varphi + T_*\psi_y[\cdot, \cdot]\varphi = -  \lambda T_* L_y[\cdot, \cdot] -  e g_y[\cdot, \cdot]  \quad {\rm in} \ Q_1, \\
      \varphi = 0 \quad {\rm on} \ \Sigma_1, \\
      \varphi (1) = - \lambda\psi_{0\zeta}[T_*, 1]- \sum_{i=1}^m \mu_i  \psi_{i\zeta}[T_*,1]
    \end{cases}
\end{align}
and
\begin{align}\label{cm0.2}
    \begin{cases}
        \phi'(s) &=  - \displaystyle \int_\Omega\big(\lambda T_*L_t[x, s] +    \varphi T_*\psi_t[x, s] + e g_t[x, s]\big) dx \quad {\rm in} \ (0, 1), \\
        \phi(1) &= 0
    \end{cases}
\end{align}
where $A^*$ is the adjoint operator of $A$, which is defined by $A^* \varphi = - \sum_{i, j = 1}^N D_i (a_{ij}(x)D_j\varphi)$; 

\noindent $(ii)$   ( optimality condition for $v_*$) 
\begin{align}\label{cm0.3}
    \lambda T_*  L_u[x, s]  -  T_*\varphi(x, s) +  e(x, s) g_u[x, s]  =  0 \quad {\rm a.a.} \ (x, s) \in Q_1; 
\end{align}

\noindent $(iii)$   (optimality condition for $T_*$) 
\begin{align}\label{cm0.4}  
    \int_0^1 \phi(s)ds+ \int_{Q_1}\Big(\lambda L[x, s]+  \varphi \big(A\zeta_* +\psi[x, s] -v_*\big)\Big) dxds + \lambda \psi_{0\xi}[T_*,1]+\sum_{i=1}^m \mu_i \psi_{i\xi}[T_*,1]= 0;
\end{align}

\noindent $(iv)$ (the complementary conditions)
\begin{align} \label{cm0.5}
\begin{cases}
    &\ \mu_i \ge 0, \ \ {\rm and} \ \  \mu_i \psi_i (\xi_*(1), \zeta_*(1))= 0,\  i = 1, 2, ..., m,   \\
    &e(x, s)\in T([a, b],  g[x, s])\ \text{a.a.}\ (x, s)\in Q_1; 
\end{cases}
\end{align}

\noindent $(v)$ (the nonnegative second-order condition)  
\begin{align} \label{cm0.6} 
    & D^2_{(\xi, T, \zeta, v)}\mathcal{L}\big(\xi_*, T_*, \zeta_*, v_*, \lambda, \varphi, \mu, \phi, e\big)[z, z]=\lambda \Big(\psi_{0TT}[T_*, 1]T^2 + 2\psi_{0 T \zeta}[T_*,1]T\zeta(1)  +\psi_{0\zeta \zeta}[T_*,1]\zeta^2(1)\Big)\notag\\
    &+\lambda \int_{Q_1}T_*\Big(L_{tt}[x, s]\xi^2   + 2L_{ty}[x, s]\xi \zeta+  2L_{tu}[x, s]\xi v  +  L_{yy}[x, s]\zeta^2 + L_{uu}[x, s]v^2  +  2L_{yu}[x, s]\zeta v\Big) dxds  \nonumber \\
     &+   \lambda  \int_{Q_1} 2 T \Big(L_t[x, s]\xi  + L_y[x, s]\zeta  +  L_u[x, s]v\Big)  dxds   \nonumber \\  
     &+  \int_{Q_1}\varphi \Big[T_*\Big(\psi_{tt}[x, s]\xi^2  +  \psi_{yy}[x, s]\zeta^2  + 2\psi_{ty}[x, s]\xi \zeta\Big)    + 2 T \Big(A \zeta +  \psi_t[x, s]\xi  + \psi_y[x, s]\zeta - v\Big)\Big]dxds \nonumber \\
     &+  \int_{Q_1} e(x, s) \Big(g_{tt}[x, s]\xi^2  +  g_{yy}[x, s]\zeta^2 + g_{uu}[x, s]v^2  + 2g_{ty}[x, s]\xi \zeta + 2g_{tu}[x, s]\xi v  +  2g_{yu}[x, s]\zeta v\Big)  dxds  \nonumber \\
     &+  \sum_{i = 1}^m \mu_i \Big(\psi_{iTT}[T_*, 1]T^2 + 2\psi_{i T \zeta}[T_*,1]T\zeta(1)  +\psi_{i\zeta \zeta}[T_*,1]\zeta^2(1)\Big)   \geq 0. 
\end{align} In addition, if the assumption $(H6)'$ is satisfied then  $\lambda=1$ and 
\begin{align}
   \max_{(1, \mu, \varphi, \phi, e)\in\Lambda[z_*]} D^2_{(\xi, T, \zeta, v)}\mathcal{L}\big(\xi_*, T_*, \zeta_*, v_*, \lambda, \varphi, \mu, \phi, e\big)[z, z]\geq 0\quad \forall z\in \mathcal{C}_1[z_*].
\end{align}
\end{proposition}
\begin{proof}  It is easy to see that 
\begin{align*}
    {\rm int}K_1 = \{v \in L^\infty(Q_1): a < {\rm ess inf}v \le  {\rm ess sup}v < b\}.
    \end{align*}
Hence ${\rm int}K_1$ is nonempty and so $(A1)$ is valid. By Lemma \ref{lemma1} and Lemma \ref{Lemma-ClosedRange}, $(A2)$ and $(A3)$ are valid. Thus all conditions of Proposition \ref{KeyProp} are fulfilled.  Accordingly, for each $z=(\xi, T, \zeta, v)\in\mathcal{C}_{1,0}[z_*]$,   there exist Lagrange multipliers  $\lambda \in \mathbb R_+$, $\phi_1 \in L^q(Q_1)$, $\phi_2^* \in (H_0^1(\Omega)\cap W^{2-\frac{2}p, p}(\Omega))^*$, $\mu = (\mu_1, \mu_2, ..., \mu_m) \in \mathbb R^m$, $\nu_1^* \in C([0, 1], \mathbb R)^*$  and  $w^* \in L^\infty(Q_1)^*$ such that the following conditions hold:
\begin{align}
    &D_z\mathcal{L}(z_*, \lambda, \phi_1, \phi_2^*, \nu_1^*, \mu, w^*)=\notag\\
    &\lambda D \widehat J(z_*) +  \big \langle (\phi_1, \phi_2^*, \nu_1^*), D F(z_*) \big \rangle + \big \langle \mu, DH(z_*) \big\rangle + \Big\langle w^*,  DG(z_*) \Big\rangle  = 0, \label{OptimRedCond1}\\
    &\lambda \geq 0,\  \mu_i\geq 0,\ \mu_i H_i(z_*)=0,\ i=1,2,..., m,\label{OptimRedCond3}\\
    &w^* \in N(K_1, g[\cdot, \cdot]), \label{OptimRedCond4}\\
    & D^2_z\mathcal{L}(z_*, \lambda, \phi_1, \phi_2^*, \nu_1^*, \mu, w^*)[z, z]\geq 0. \label{OptimRedCond5}
\end{align} 
 Note that $\nu_1^*$ is a signed Radon measure which is absolutely continuous w.r.t the Lebesgue measure $|\cdot|$  on $[0, 1]$. By Riesz’s Representation (see \cite[Chapter 01, p. 19]{Ioffe-1979}  and  \cite[Theorem 3.8, p. 73]{Hirsch-1999}), there exists a  function of bounded variation $\nu_1$, which is continuous from the right and vanishes at zero such that 
\begin{align*}  
\langle \nu_1^*, \vartheta \rangle  =  \int_0^1 \vartheta(s) d\nu_1(s) \quad \quad  \forall  \vartheta \in  C([0,1], \mathbb{R}), 
\end{align*}
where the integral stands for the  Riemann-Stieltjes integral. We define function $\phi : [0, 1] \to \mathbb{R}$ by setting
\begin{align}
\label{defofp}
    \phi (s) :=- \nu_1 ((s, 1]) = \nu_1(s) - \nu_1(1). 
\end{align}
Then the function $\phi$ is of bounded variation and $\phi(1)=0$. Then for all $(\xi, T) \in C([0,1], \mathbb{R})\times\mathbb{R}$,  we have
\begin{align*}
    \langle \nu_1^*, \xi -\int_0^{(\cdot)} T d\tau \rangle  = \int_0^1 (\xi(s)-Ts) d\nu_1(s). 
\end{align*}
Hence  \eqref{OptimRedCond1} becomes
\begin{align}\label{OptimRedCond7}
&\lambda \psi_{0T}[T_*,1]T +\lambda\psi_{0\zeta}[T_*, 1]\zeta(1) \notag\\
&+\lambda\int_{Q_1} \Big(T L(x, \xi_*, \zeta_*, v_*) +  T_*[L_t(x, \xi_*, \zeta_*, v_*)\xi +  L_y(x, \xi_*, \zeta_*, v_*)\zeta + L_u(x, \xi_*, \zeta_*, v_*)v]\Big)  dxds   \notag\\
&+ \int_{Q_1}\phi_1 \Big(\dfrac{\partial \zeta}{\partial s} + T_* A\zeta + T_*\psi_t(x, \xi_*, \zeta_*)\xi  + T_*\psi_y(x, \xi_*, \zeta_*)\zeta - T_* v\Big) dxds        \nonumber \\   
&+ \int_{Q_1}\phi_1\Big(A\zeta_* +\psi(x, \xi_*, \zeta_*)-v_*\Big)T dxds    \notag\\
&+\langle \phi_2^*, \zeta (0)\rangle   +  \int_0^1 (\xi(s)-Ts) d\nu_1(s) 
+ \sum_{i=1}^m \mu_i [\psi_{i\xi}[T_*,1]T + \psi_{i\zeta}[T_*,1]\zeta(1) ]\notag\\
&+ \langle w^*, g_t(\cdot, \xi_*, \zeta_*, v_*)\xi + g_y(\cdot, \xi_*, \zeta_*, v_*)\zeta + g_u(\cdot, \xi_*, \zeta_*, v_*)v\rangle=0,  \quad    \forall (\xi, T, \zeta, v)\in Z.
\end{align}

\noindent {\bf Step 1}. Deriving optimality conditions for $v_*$.

By inserting $(\xi, T,  \zeta) = (0, 0, 0)$ into the above equality to get 
\begin{align}
\label{cm1}
    \lambda\int_{Q_1}   T_* L_u(x, \xi_*, \zeta_*, v_*)v  dxds   - \int_{Q_1} \phi_1 T_*v dxds + \langle w^*,  g_u(\cdot, \xi_*, \zeta_*, v_*)v\rangle = 0,  \quad \quad   \forall v \in U_1. 
\end{align}
It follows from (\ref{cm1}) that 
\begin{align*}
    g_u[\cdot, \cdot]w^* = T_*\phi_1  - \lambda T_* L_u[\cdot, \cdot]. 
\end{align*}
Hence
\begin{align}
    \label{cm3}
    w^* = \frac{T_*\phi_1  - \lambda T_* L_u[\cdot, \cdot]}{g_u[\cdot, \cdot]} \quad {\rm on} \ L^\infty(Q_1).
\end{align}
By $(H5)$, $\frac{1}{g_u[\cdot, \cdot]} \in L^p(Q_1)$. Since  $T_*\phi_1 - \lambda T_* L_u[\cdot, \cdot] \in L^q(Q_1)$,  we have  $\frac{T_*\phi_1 - \lambda T_* L_u[\cdot, \cdot]}{g_u[\cdot, \cdot]} \in L^1(Q_1)$. Consequently,   $w^*$ can be represented by a function  $e \in L^1(Q_1)$. From this and  (\ref{cm3}), we obtain  
$$
e(x,s)=\frac{T_*\phi_1(x,s)  - \lambda T_* L_u[x,s]}{g_u[x,s]} \quad {\rm a.a} \ (x, s)\in Q_1.
$$ This implies that
\begin{align} \label{cm4}
    \lambda T_* L_u[x, s]  -  T_*\phi_1(x, s) +  e(x, s) g_u[x, s]  =  0 \quad {\rm a.a.} \ (x, s) \in Q_1 
\end{align} which is assertion $(ii)$ of the proposition.  By \eqref{OptimRedCond4} and Corollary 4 in \cite{Pales}, we have 
$e(x, s)\in N([a, b], g[x, s])$ for a.a. $(x, s)\in Q_1$. 

\medskip

\noindent {\bf Step 2}. Deriving the first adjoint equation.

Inserting $(\xi, T, v) = (0, 0, 0)$ into equality (\ref{OptimRedCond7}), we get 
\begin{align}\label{cm5} 
&\int_{Q_1}\Big(\dfrac{\partial \zeta}{\partial s} + T_* A\zeta + T_*\psi_y[x, s]\zeta \Big) \phi_1 dxds     +\langle \phi_2^*, \zeta(0)\rangle   =\notag\\
&-\lambda\psi_{0\zeta}[T_*, 1]\zeta(1)- \sum_{i=1}^m \mu_i  \psi_{i\zeta}[T_*, 1]\zeta(1)  -  \int_{Q_1} \Big(\lambda T_* L_y[x, s] +  e g_y[x, s] \Big)\zeta  dxds, \quad \quad  \forall \zeta \in Y_1. 
\end{align}  Let us consider the  equation 
\begin{align}
    \label{cm7} 
    \begin{cases}
      - \dfrac{\partial \varphi}{\partial s} + T_* A^*\varphi + T_*\psi_y[x, s]\varphi = -  \lambda T_* L_y[\cdot, \cdot] -  e g_y[\cdot, \cdot]  \quad {\rm in} \ Q_1, \\
      \varphi = 0 \quad {\rm on} \ \Sigma_1, \\
      \varphi (1) =-\lambda\psi_{0T}[T_*,1] - \sum_{i=1}^m \mu_i  \psi_{i\zeta}[T_*, 1] \quad {\rm in} \ \Omega.
    \end{cases}
\end{align} Note that $ \lambda\psi_{0T}[T_*,1]+ \sum_{i=1}^m \mu_i\psi_{i\zeta}(\xi_*(1), \zeta_*(1))$ is a continuous linear mapping from $C(\bar\Omega)$ to $\mathbb{R}$. It can be identified  with a finite Radon  measure on $\bar\Omega$. Therefore, the above parabolic equation involves  measure data.  Since $0 \le T_*\psi_y[\cdot, \cdot] \in L^\infty(Q_1)$ and $- \lambda T_* L_y[\cdot, \cdot] - e g_y[\cdot, \cdot] \in L^1(Q_1)$, Theorem 4.2 in \cite{Arada-2002-1} implies that, equation (\ref{cm7}) has a unique weak solution $\varphi \in L^1(0, 1; W^{1, 1}_0(\Omega))$ and  the following Green formula is valid:
\begin{align}
    \label{cm8} 
    &\int_{Q_1}\Big(\dfrac{\partial \zeta}{\partial s} + T_* A\zeta + T_*\psi_y[x, s]\zeta \Big) \varphi dxds  +\int_\Omega \varphi(0)\zeta(0) dx=\notag\\
    &-\lambda\psi_{0T}[T_*,1]\zeta(1)- \sum_{i=1}^m \mu_i  \psi_{i\zeta}[T_*,1]\zeta(1)  -  \int_{Q_1} \Big(\lambda T_* L_y[x, s] +  e g_y[x, s] \Big)\zeta  dxds, \quad \quad  \forall \zeta \in Y_1.
\end{align} 
Subtracting (\ref{cm5}) from (\ref{cm8}), we get
\begin{align} \label{cm9} 
    \int_{Q_1}\Big(\dfrac{\partial \zeta}{\partial s} + T_* A\zeta + T_*\psi_y[x, s]\zeta \Big) \Big(\phi_1 -  \varphi \Big) dxds +\langle\phi_2^*, \zeta(0)\rangle-\int_\Omega \varphi(0)\zeta(0)dx  = 0 \quad  \forall \zeta \in Y_1. 
\end{align}
By Lemma \ref{Lemma-LinearizedEq}, we see that for each $\vartheta \in L^p(Q_1)$ and $\zeta_0\in H^1_0(\Omega)\cap W^{2-\frac{2}p, p}(\Omega)$ the  equation 
\begin{align*}
    \dfrac{\partial \zeta}{\partial s} + T_* A\zeta + T_*\psi_y[x, s]\zeta  =  \vartheta \quad {\rm in} \ Q_1, \quad \quad \zeta = 0 \quad {\rm on} \ \Sigma_1, \quad \quad \zeta(0)  =  \zeta_0 \quad {\rm in} \ \Omega
\end{align*}
has a unique solution $\zeta \in Y_1$. Inserting such a solution into (\ref{cm9}), we obtain 
\begin{align*}
    &\int_{Q_1} \big(\phi_1 -  \varphi \big) \vartheta dxds  = 0 \quad \quad  \forall \vartheta \in L^p(Q_1),\\
    &\langle\phi_2^*, \zeta_0\rangle-\int_\Omega\varphi(0)\zeta_0 dx=0\quad \forall \zeta_0\in H^1_0(\Omega)\cap W^{2-\frac{2}p, p}(\Omega). 
\end{align*} This implies that $\phi_2^*=\varphi(0)$ and 
$ \phi_1 =  \varphi$ on $Q_1$. Consequently, $\varphi \in L^q(Q_1)  \cap L^1(0, 1; W^{1, 1}_0(\Omega))$ and we obtain the first adjoint equation of assertion $(i)$ of the theorem. 
Besides, (\ref{cm4}) becomes
\begin{align}
    \label{cm12}
    \lambda T_* (s) L_u[x, s]  -  T_*\varphi(x, s) +  e(x, s) g_u[x, s]  =  0 \quad {\rm a.a.} \ (x, s) \in Q_1. 
\end{align} which is assertion $(ii)$ of the theorem.

\noindent {\bf Step 3}. Deriving the second adjoint equation (optimality condition for $\xi_*$).

Inserting $(T, \zeta, v) = (0, 0, 0)$ into  equality (\ref{OptimRedCond7}), we get
\begin{align}\label{cm13}
      - \int_0^1 \xi(s) d\nu_1(s)   =  \int_0^1 \int_\Omega\big(\lambda T_*L_t[x, s] +  \varphi T_*\psi_t[x, s] +   e g_t[x, s]\big)\xi dxds   \quad    \forall \xi \in C([0, 1], \mathbb{R}) 
\end{align} By Lemma 5.1 in \cite{Kien-2018}, the above equality is  valid for $\xi(s) =\xi_0 \chi_{(\tau, 1]}(s)$ with $0 \leq \tau < 1$, where $\chi_{(\tau,1]}$  is the indicator function of the set $(\tau, 1] \subset [0, 1]$. Inserting $\xi(s)$ into \eqref{cm13},  we get
\begin{align*}
      - \int_\tau^1\xi_0d\nu_1(s)   =  \int_\tau^1 \xi_0\int_\Omega\big(\lambda T_*L_t[x, s] +   \varphi T_*\psi_t[x, s] +  e g_t[x, s]\big) dxds \quad   \forall \xi_0\in\mathbb{R}, \     \forall \tau  \in [0, 1). 
\end{align*} 
From this and  definition of function $\phi$ in (\ref{defofp}), we obtain 
\begin{align*}
    \phi(\tau) =  \int_\tau^1 \int_\Omega\big(\lambda T_*L_t[x, s] +   \varphi T_*\psi_t[x, s] +    e g_t[x, s]\big) dxds, \quad \forall \tau  \in [0, 1). 
\end{align*} 
This implies that $\lim_{\tau\to 1}\phi(\tau)=0=\phi(1).$ We obtain equation \eqref{cm0.2}.

\noindent{\bf Step 4.} Establishing optimality conditions for $T_*$.

 Inserting $(\xi, \zeta, v) = (0, 0, 0)$ into (\ref{OptimRedCond7}) and notice that $- \int_0^1 Ts  d\nu_1(s)  =  \int_0^1 T\phi(s)ds$, we have
\begin{align*}
    \int_0^1 T\phi(s)ds+ \int_0^1 \int_\Omega \lambda L[x, s]T dxds &+  \int_0^1 \int_\Omega \varphi \Big(A\zeta_* +\psi[x, s]  -v_*\Big)T dxds  \\
    &+ \lambda \psi_{0\xi}[T_*,1]T+\sum_{i=1}^m \mu_i \psi_{i\xi}[T_*,1]T =   0,    \quad   \forall T \in  \mathbb{R}. 
\end{align*}
Hence we obtain 
\begin{align*}
    \int_0^1 \phi(s)ds+ \int_{Q_1}\Big(\lambda L[x, s]+  \varphi \big(A\zeta_* +\psi[x, s] -   v_*\big)\Big) dxds + \lambda \psi_{0\xi}[T_*,1]+\sum_{i=1}^m \mu_i \psi_{i\xi}[T_*,1]= 0  
\end{align*} which is equation \eqref{cm0.4}. 

\noindent{\bf Step 5.} Deriving the non-negative second-order condition.

Assertion $(v)$ of the theorem follows from condition (\ref{OptimRedCond5}) and formulas (\ref{DJ2})-(\ref{DG2}). 

Finally, if $(H6)'$ is satisfied,  then we have  from Lemma \ref{lemma-Robinson} that $(A3)'$ is valid. By the part $(b)$ of Proposition \ref{KeyProp}, we get $\lambda=1$.  The proof of the proposition is complete.  
\end{proof}

\medskip

In the rest of this section, we focus on second-order sufficient optimality conditions for  problem $(P_1)$.   For this  we need to enlarge the critical cone $\mathcal C_1[z_*]$ by the cone $\mathcal C'_1[z_*]$ which consists  of vectors $(\xi, T, \zeta, v) \in C([0, 1], \mathbb R)\times\mathbb{R} \times W^{1, 1}_2(0, 1; D, H) \times L^2(Q_1)$ satisfying the following conditions:

\begin{itemize}
    \item [$(b_1')$]     $\psi_{0T}[T_*, 1] T +\psi_{0\zeta}[T_*,1] \zeta(1) + \displaystyle \int_{Q_1}\big(T L[x, s] +  T_*[L_t[x, s]\xi  +  L_y[x, s]\zeta + L_u[x, s]v]\big) dxds \le 0$;

    \item [$(b_2')$]  $\dfrac{\partial \zeta}{\partial s} +T_*\big( A\zeta + \psi_t[\cdot, \cdot]\xi  +  \psi_y[\cdot, \cdot]\zeta -v\big)+ T(A\zeta_* +\psi[\cdot, \cdot]-v_*)=0, \quad    \zeta(0)=0$;

    \item [$(b_3')$] $\xi(s) =  Ts \quad \forall s \in [0, 1]$;

    \item [$(b_4')$]$\psi_{i\xi}[T_*, 1]T+\psi_{i\zeta}[T_*, 1]\zeta(1)\leq 0\ $  for $i\in \big\{1,2,.., m| \psi_i[T_*, 1]=0\big\}$;
      
    \item [$(b_5')$] $g_t[x, s]\xi(s) +  g_y[x, s]\zeta(x, s) + g_u[x, s]v(x, s)\in T([a, b],  g[x, s])$ for a.a. $(x, s)\in Q_1$.
\end{itemize}

The following proposition gives second-order sufficient conditions for locally optimal solutions to $(P_1)$.

\begin{proposition}\label{Pro-SOSC-P1}
    Suppose $z_*  = (\xi_*, T_*, \zeta_*, v_*)\in\Phi_1$ with $T_*>0$,  assumptions $(H1)$-$(H5)$ and $(H7)$,   and  there exist multipliers $(\lambda,  \mu, \varphi, \phi, e)$ with $\lambda=1$ satisfying  conditions $(i)-(iv)$ of Proposition \ref{Pro-KKT-P1}. Furthermore,  assume that the following conditions are fulfilled:

    \noindent $(v)$  (the strictly second-order condition)  
   \begin{align} \label{StrictSOC}  
   &\psi_{0TT}[T_*, 1]T^2 + 2\psi_{0 T\zeta}[T_*, 1]T\zeta(1)  +\psi_{0\zeta\zeta}[T_*, 1]\zeta^2(1) \notag\\
    & +\int_{Q_1} T_* \Big(L_{tt}[x, s]\xi^2   + 2L_{ty}[x, s]\xi \zeta+  2L_{tu}[x, s]\xi v +L_{yy}[x, s]\zeta^2 + L_{uu}[x, s]v^2  +  2L_{yu}[x, s]\zeta v\Big) dxds  \nonumber \\
     &+  \int_{Q_1} 2 T \Big(L_t[x, s]\xi  + L_y[x, s]\zeta  +  L_u[x, s]v\Big)  dxds     \nonumber \\
    &+  \int_{Q_1}\varphi \Big[T_*\Big(\psi_{tt}[x, s]\xi^2  +  \psi_{yy}[x, s]\zeta^2  + 2\psi_{ty}[x, s]\xi \zeta\Big)    + 2 T \Big(A \zeta +  \psi_t[x, s]\xi  + \psi_y[x, s]\zeta - v\Big)\Big]dxds \nonumber \\
     &+  \int_{Q_1} e(x, s) \Big(g_{tt}[x, s]\xi^2  +  g_{yy}[x, s]\zeta^2 + g_{uu}[x, s]v^2  + 2g_{ty}[x, s]\xi \zeta + 2g_{tu}[x, s]\xi v  +  2g_{yu}[x, s]\zeta v\Big)  dxds  \nonumber \\
     &+  \sum_{i = 1}^m \mu_i \Big(\psi_{iTT}[T_*, 1]T^2 + 2\psi_{i T\zeta}[T_*, 1]T\zeta(1)  +\psi_{i\zeta\zeta}[T_*, 1]\zeta^2(1)\Big) >0 
\end{align} for all $z=(\xi, T, \zeta, v)\in \mathcal{C}'_1[z_*]\setminus\{(0,0,0,0)\}$. 

    \noindent $(vi)$  (the Legendre-Clebsch condition) there is a number  $\Lambda > 0$ such that
    \begin{align} \label{Legendre-Clebsch-condition}
        T_*L_{uu}[x, s] + e(x, s)g_{uu}[x, s] \geq \Lambda   \quad {\rm a.a.} \ (x, s) \in Q_1. 
     \end{align}
    Then  there exist numbers $\delta > 0$ and $\kappa > 0$ such that 
    \begin{align}
        \widehat J (z) \ge \widehat J (z_*) + \kappa\Big( (T - T_*)^2    + \|v-v_*\|_{L^2(Q_1)}^2 \Big) 
    \end{align}
    for all $z = (\xi, T, \zeta, v) \in \Phi_1 \cap B_{*Z}(z_*, \delta)$. In particular, $z_*$ is a locally optimal solution to  $(P_1)$.
\end{proposition}
\begin{proof}  
Let us define  the so-called Lagrange function for $(P_1)$:
\begin{align}\label{LagrangeFunction2}
\mathcal L\big(\xi,T, \zeta, v, \varphi, \phi,  \mu, e\big)= \widehat J(\xi, T,  \zeta, v) +\langle(\varphi,\phi), F(\xi,T, \zeta, v)\rangle + \langle \mu, H(\xi, T, \zeta, v)\rangle +\langle e, G(\xi, T, \zeta, v)\rangle,
\end{align} where 
\begin{align}
    & \langle(\varphi,\phi), F(\xi, T, \zeta, v)\rangle =\int_{Q_1}\Big[\zeta\big(-\dfrac{\partial \varphi}{\partial s} + T A^*\varphi \big) +  \big(\psi(x, \xi, \zeta) - v\big)T\varphi \Big]dxds  +\langle \varphi(1), \zeta(1)\rangle   \notag\\
    &+ \int_0^1 \xi(s) d\phi(s)+   \int_0^1  T \phi(s )  ds + \phi(1)\xi(1),     \label{DefLF}\\
    &\langle \mu, H(\xi, T, \zeta, v)\rangle= \sum_{i=1}^m \mu_i  \psi_i(T, \zeta(1)),   \nonumber \\
    &\langle e, G(\xi, T, \zeta, v)\rangle = \int_{Q_1} eg(x, \xi, \zeta, v) dx ds.\nonumber  
\end{align} Then  from conditions $(i)-(iii)$ of Proposition \ref{Pro-KKT-P1} we can show that 
\begin{align}
\label{ss0}
    D_z\mathcal{L}(z_*, \varphi, \phi, \mu, e)=0.
\end{align} Besides, for each $d=(\xi, T,  \zeta, v)\in \mathcal C'_1[z_*]$,  $D^2_{zz}\mathcal{L}(z_*, \varphi, \phi,\mu, e)[d,d]$ equals to the left hand side of \eqref{StrictSOC}.  Hence
\begin{align}
\label{ss0.1}
   D^2_{zz}\mathcal{L}(z_*, \varphi, \phi,\mu, e)[d,d]>0  \quad  \forall d\in \mathcal C'_1[z_*]\setminus\{(0,0,0,0)\}. 
\end{align}

We now suppose to the contrary that the conclusion of the proposition were false. Then we could find sequences $\{z_n = (\xi_n, T_n,  \zeta_n, v_n) \}\subset \Phi_1$ and $\{\gamma_n\}\subset \mathbb{R}_+$ such that  $\xi_n \to \xi_*$ in $C([0, 1], \mathbb R)$, $\zeta_n \to \zeta_*$ in $Y_1$, $v_n\to v_*$ in $L^\infty(Q_1)$, $T_n \to T_*$ in $\mathbb R$,  $\gamma_n\to 0^+$ and
\begin{align}
\label{ss1}
\widehat J(z_n) < \widehat J(z_*) + \gamma_n \bigg(\|v_n - v_*\|^2_{L^2(Q_1)}    + |T_n-T_*|^2\bigg). 
\end{align}
If $v_n = v_*$ and $T_n = T_*$ then by the uniqueness we have $\zeta_n=\zeta_*$ and $\xi_n = \xi_*$. This leads to  $\widehat J(z_*) < \widehat J(z_*)$ which is absurd. Therefore, we can assume that $\|v_n - v_*\|^2_{L^2(Q_1)}    + |T_n-T_*|^2 \ne 0$ for all $n>0$.   Define 
\begin{align*}
    t_n=  \Big(\|v_n - v_*\|^2_{L^2(Q_1)}    + |T_n-T_*|^2\Big)^{1/2}.  
\end{align*} Then $t_n\to 0^+$ as $n \to + \infty$ and we get from  (\ref{ss1}) that 
\begin{align}
    \label{ss2}
    \widehat J(z_n) - \widehat J(z_*) \le o(t_n^2).
\end{align}
Put
 $\widehat \xi_n= \frac{\xi_n-\xi_*}{t_n}$, $\widehat \zeta_n= \frac{\zeta_n-\zeta_*}{t_n}$, $\widehat v_n= \frac{v_n-v_*}{t_n}$ and $\widehat T_n= \frac{T_n-T_*}{t_n}$.   Then we have 
 \begin{align*}
     &|\widehat T_n|^2  +  \|\widehat v_n\|^2_{L^2(Q_1)} = 1.
 \end{align*} Therefore, we can assume that $\widehat T_n\to \widehat T$ in $\mathbb R$ and  $\widehat v_n \rightharpoonup \widehat v$ in $L^2(Q_1)$. On the other hand  we have
\begin{align*}
&\|\widehat \xi_n\|_{C([0,1], \mathbb{R})} =|\widehat T_n|\leq 1,\\
&|\widehat \xi_n(s_1)- \widehat \xi_n(s_2)|\le |s_1-s_2||\widehat T_n| \le |s_1-s_2|  \quad \forall s_1, s_2 \in [0,1]. 
\end{align*}Hence  $\{\widehat \xi_n\}$ is equicontinuous. By the Arzel\'{a}–Ascoli theorem, we can assume that $\widehat \xi_n\to \widehat \xi$ in $C([0,1], \mathbb{R})$. 

Let us claim that $\widehat \zeta_n \rightharpoonup \widehat \zeta$ for some in $\widehat\zeta\in W^{1, 1}_2(0, 1; D, H)$ and $\widehat \zeta_n(1)\to\widehat\zeta(1)$ in $L^2(\Omega)$. In fact, since $(\xi_n, T_n,  \zeta_n, v_n)\in \Phi_1$ and $(\xi_*, T_*,  \zeta_*, v_*)\in\Phi_1$ we have 
\begin{align*}
    &\frac{\partial \zeta_n}{\partial s}  + T_n A\zeta_n + T_n \psi(x, \xi_n, \zeta_n) = T_n v_n,  \quad \zeta_n(0) =  y_0,   \\
    &\frac{\partial \zeta_*}{\partial s}  + T_* A \zeta_* + T_* \psi(x, \xi_*, \zeta_*) = T_*  v_*,  \quad \zeta_*(0) =  y_0.
\end{align*}
This implies that
\begin{align}
\label{ss3}
     &\frac{\partial (\zeta_n-\zeta_*)}{\partial s}+ T_nA(\zeta_n-\zeta_*)    + T_n[\psi(x, \xi_*, \zeta_n) - \psi(x, \xi_*, \zeta_*)]    \nonumber \\
     &= T_n\Big[(v_n - v_*)  -  (\psi(x, \xi_n, \zeta_n) - \psi(x, \xi_*, \zeta_n))\Big] +  (T_n - T_*)[v_*  - A \zeta_* - \psi(x, \xi_*, \zeta_*)],      \nonumber \\
     &(\zeta_n-\zeta_*)(0)=0.
\end{align}
By Taylor's expansions, there exist measurable functions $\theta_{1, n}, \theta_{2, n}$ such that 
\begin{align*}
    &\psi(x, \xi_n, \zeta_n) - \psi(x, \xi_*, \zeta_n)   =     \psi_t(x, \xi_* + \theta_{1, n}(\xi_n - \xi_*), \zeta_n)(\xi_n - \xi_*),   \quad  0\leq\theta_{1, n}(x,t)\leq 1,     \\
    &\psi(x, \xi_*, \zeta_n)-\psi(x, \xi_*, \zeta_*) = \psi_y(x, \xi_*, \zeta_* +\theta_{2, n}(\zeta_n-\zeta_*)) (\zeta_n-\zeta_*), \quad  0\leq\theta_{2, n}(x,t)\leq 1. 
\end{align*}
Note that  $T_n \to T_*$ in $\mathbb R$, $\xi_n \to \xi_*$ in $C([0, 1], \mathbb R)$ and $\zeta_n\to \zeta_*$ in $L^\infty(Q_1)$ so there exists $M>0$ such that $|T_n|  + \|\xi_n\|_{C([0, 1], \mathbb R)}  + \|\zeta_n\|_{L^\infty(Q_1)} + \|\xi_* + \theta_{1, n}(\xi_n - \xi_*)\|_{C([0, 1], \mathbb R)} + \|\zeta_* +\theta_n(\zeta_{2, n}-\zeta_*)\|_{L^\infty(Q_1)}\le M$. By $(H2)$, there exists $k_M>0$ such that 
\begin{align*}
    |\psi_t(x, \xi_* + \theta_{1, n}(\xi_n - \xi_*), \zeta_n)|  +    |\psi_y(x, \xi_*, \zeta_* +\theta_{2, n}(\zeta_n-\zeta_*))|\le k_M. 
\end{align*} 
Hence $T_n \psi_y(x, \xi_*, \zeta_* +\theta_{2, n}(\zeta_n-\zeta_*))$ is bounded in $L^\infty(Q_1)$. From (\ref{ss3}) we get 
\begin{align}
\label{ss4}
     &\frac{\partial \widehat \zeta_n}{\partial s}+ T_nA\widehat \zeta_n    +  T_n \psi_y(x, \xi_*, \zeta_* +\theta_{2, n}(\zeta_n-\zeta_*)) \widehat \zeta_n  \nonumber \\ &= T_n [\widehat v_n   -   \psi_t(x, \xi_* + \theta_{1, n}(\xi_n - \xi_*), \zeta_n) \widehat \xi_n] +  [v_*  - A \zeta_* - \psi(x, \xi_*, \zeta_*)]\widehat T_n,  \quad \quad 
     \widehat \zeta_n(0)=0.
\end{align}
By \cite[Theorem 5, p.360]{Evan}, there exists a constant $c>0$ such that 
\begin{align*}
\|\widehat \zeta_n\|_{W^{1,1}_2(0, 1; D, H)}  &\le c \Big\|T_n [\widehat v_n   -   \psi_t(x, \xi_* + \theta_{1, n}(\xi_n - \xi_*), \zeta_n) \widehat \xi_n] +  [v_*  - A \zeta_* - \psi(x, \xi_*, \zeta_*)]\widehat T_n \Big\|_{L^2(Q_1)}  \\
&\le cM + c M k_M |Q_1|^{\frac{1}{2}} + c\|v_*  - A \zeta_* - \psi(x, \xi_*, \zeta_*)\|_{L^2(Q_1)}. 
\end{align*}
Hence $\{\widehat \zeta_n\}$ is bounded in $W^{1, 1}_2(0, 1; D, H)$. Without loss of generality, we may assume that $\widehat \zeta_n \rightharpoonup \widehat \zeta$ in $W^{1, 1}_2(0, 1, D, H)$. By the Aubin theorem, the imbedding $W^{1, 1}_2(0, 1, D, H)\hookrightarrow L^2(0, 1; H)$ is compact. Hence $\widehat \zeta_n\to\widehat \zeta$ in norm of $L^2(0, 1; H)$. This implies that $\widehat \zeta_n(t)\to\widehat \zeta(t)$ a.a. $t\in [0, 1]$. On the other hand $W^{1, 1}_2(0, 1, D, H)\hookrightarrow C([0, 1],  H)$. Hence $\widehat \zeta_n(1)\to\widehat \zeta(1)$ in $H$.  The claim is justified. 

Let us divide the remainder of the proof into some steps.

\noindent {\bf Step 1.}     Showing that $(\widehat \xi, \widehat T,  \widehat \zeta, \widehat v) \in \mathcal C_1'[(\xi_*, T_*, \zeta_*, v_*)]$.

 We now  use the procedure in the proof of \cite[Theorem 3, p. 356]{Evan}. By  passing to the limit,   we obtain from (\ref{ss4}) that
\begin{align*}
     \frac{\partial \widehat \zeta}{\partial s}+ T_*A\widehat \zeta    +  T_* \psi_y[\cdot, \cdot] \widehat \zeta  = T_* (\widehat v - \psi_t[\cdot, \cdot]\widehat \xi) +  (v_*  - A \zeta_* - \psi[\cdot, \cdot])\widehat T,   \quad  \widehat \zeta(0)=0
\end{align*}
which implies that condition $(b'_2)$ is satisfied. Also, we have 
\begin{align}
    \label{ss6.1}
    |\widehat T_n s - \widehat \xi (s)|  =  |\widehat \xi_n (s) - \widehat \xi (s)| \le {\rm max}_{s \in [0, 1]}|\widehat \xi_n (s) - \widehat \xi (s)| = \|\widehat \xi_n - \widehat \xi\|_{C([0, 1], \mathbb R)} \quad \forall s \in [0, 1]. 
\end{align}
Letting $n\to\infty$ in (\ref{ss6.1}) and notice that $\widehat \xi_n\to \widehat \xi$ in $C([0,1], \mathbb{R})$, we obtain 
\begin{align*}
    \widehat \xi (s) = \widehat T s   \quad \forall s \in [0, 1].
\end{align*}
This implies that $(b'_3)$ is valid. 

By the mean value theorem, we have from (\ref{ss2}) that
\begin{align} \label{ss7}  
    &\psi_{0T}(T_*+\rho_1(T_n-T_*), \zeta_n(1))\widehat T_n +\psi_{0\zeta}(T_*, \zeta_*+\rho_2(\zeta_n(1)-\zeta_*(1)))\widehat\zeta_n(1)\notag\\
    &+\int_{Q_1}\widehat T_n L(x, \xi_n, \zeta_n, v_n) dx ds  +   \int_{Q_1} T_* L_t(x, \xi_* + \eta_1(\xi_n - \xi_*), \zeta_n, v_n)\widehat \xi_n dx ds  \nonumber \\
    &+ \int_{Q_1} T_* L_y(x, \xi_*, \zeta_* + \eta_2(\zeta_n - \zeta_*), v_n)\widehat \zeta_n dx ds +  \int_{Q_1} T_* L_u(x, \xi_*, \zeta_*, v_* +\eta_3(v_n - v_*))\widehat v_n dx ds  \le \frac{o(t^2_n)}{t_n},
\end{align}
where $0 \le \rho_i, \eta_j \le 1$. Since $\xi_n \to \xi_*$ in $C([0, 1], \mathbb R)$, $\zeta_n \to \zeta_*$ in $Y_1$, $v_n\to v_*$ in $L^\infty(Q_1)$ and $(H4)$, we have 
\begin{align*}
    &\|L(\cdot, \xi_n, \zeta_n, v_n)  -  L(\cdot, \xi_*, \zeta_*, v_*)\|_{L^\infty(Q_1)} \to 0 \quad {\rm as} \quad n \to \infty,   \\
    &\|L_t(\cdot, \xi_* + \eta_1(\xi_n - \xi_*), \zeta_n, v_n)  -  L_t(\cdot, \xi_*, \zeta_*, v_*)\|_{L^\infty(Q_1)} \to 0 \quad {\rm as} \quad n \to \infty,   \\
    &\|L_y(\cdot, \xi_*, \zeta_* + \eta_2(\zeta_n - \zeta_*), v_n)  -  L_y(\cdot, \xi_*, \zeta_*, v_*)\|_{L^\infty(Q_1)} \to 0 \quad {\rm as} \quad n \to \infty,   \\
    &\|L_u(\cdot, \xi_*, \zeta_*, v_* +\eta_3(v_n - v_*))  -  L_u(\cdot, \xi_*, \zeta_*, v_*)\|_{L^\infty(Q_1)} \to 0 \quad {\rm as} \quad n \to \infty.
\end{align*}
Combining this with $(H7)$ and the fact  $\widehat \zeta_n(1)\to \widehat \zeta(1)$ in $L^2(\Omega)$,    we deduce from (\ref{ss7}) when  $n\to\infty$ that
\begin{align*}
    \psi_{0T}[T_*, 1]\widehat T +\psi_{0\zeta}[T_*,1]\widehat\zeta(1)+\int_{Q_1}\Big(\widehat T L[x, s] +  T_*[L_t[x, s]\widehat \xi  +  L_y[x, s]\widehat \zeta + L_u[x, s]\widehat v]\Big) dxds \le 0.
\end{align*}
Hence condition $(b'_1)$  of  $\mathcal C'_1[z_*]$ is valid.  It remains to verify  conditions $(b'_4)$ and $(b'_5)$  of $\mathcal C'_1[z_*]$.  For each $i \in \{1, 2, ..., m\}$, we have 
\begin{align}
    \label{ss9}
    \psi_i(\xi_n(1), \zeta_n(1)) - \psi_i(\xi_*(1), \zeta_*(1)) 
    \in (- \infty, 0] - \psi_i(\xi_*(1), \zeta_*(1)).
\end{align}
Using the mean value theorem for the left hand side of (\ref{ss9}) and dividing both sides by $t_n$, we get
\begin{align} \label{ss10}
    &\psi_{iT}(T_* + \eta_4(T_n- T_*)), \zeta_n(1))\widehat T_n + \psi_{i\zeta}(T_*, \zeta_*(1) + \eta_5(\zeta_n(1) - \zeta_*(1)))\widehat \zeta_n(1) dx  \nonumber \\    
    &\in \frac{1}{t_n} \Big((- \infty, 0] - \psi_i(T_*, \zeta_*(1))\Big) \subseteq T\big((- \infty, 0]; \psi_i(T_*, \zeta_*(1))\big), 
\end{align}
where $0 \le \eta_4, \eta_5 \le 1$. Notice that $T_n \to \widehat T$ in $\mathbb R$,  $\widehat \zeta_n(1) \to \widehat \zeta(1)$ in $L^2(\Omega)$ and 
\begin{align*}
    &\|\psi_{it}(T_* + \eta_4(\xi_n(1) - \xi_*(1)), \zeta_n(1))  -   \psi_{i t}(T_*, \zeta_*(1))\|_{L^\infty(\Omega)}  \to 0 \quad {\rm as} \quad n \to \infty,   \\
    &\|\psi_{iy}(T_*, \zeta_*(1) + \eta_5(\zeta_n(1) - \zeta_*(1)))  -   \psi_{i y}(T_*, \zeta_*(1))\|_{L^\infty(\Omega)}  \to 0 \quad {\rm as} \quad n \to \infty.
\end{align*}
Letting $n \to \infty$ in (\ref{ss10}) and using $(H7)$, we get 
\begin{align*}
    \displaystyle \psi_{i T}(T_*, \zeta_*(1))\widehat T +  \displaystyle  \psi_{i \zeta}(T_*, \zeta_*(1))\widehat \zeta(1) \in T((-\infty, 0]; \psi_i(T_*, \zeta_*(1))), \quad i =  1,2,.., m.
\end{align*} Hence condition $(b'_4)$  is valid.  Also, we have
\begin{align*}
    G(\xi_n, T_n, \zeta_n, v_n)-G(\xi_*, T_*, \zeta_*, v_*)\in (K_1- g[\cdot, \cdot]).
\end{align*} 
 By a Taylor expansion and the definitions of $\widehat \xi_n, \widehat \zeta_n, \widehat v_n$, we have
\begin{align}\label{ss11}   
    g_t[\cdot, \cdot]\widehat \xi_n  +  g_y[\cdot, \cdot]\widehat \zeta_n + g_u[\cdot, \cdot]\widehat v_n  + \frac{o(t_n)}{t_n} \in \frac{1}{t_n} \big(K_1 - g[\cdot, \cdot] \big)      \subseteq {\rm cone}(K_1 - g[\cdot, \cdot])      \subseteq T_{L^2}(K_1; g[\cdot, \cdot])   
\end{align}
Since $T_{L^2}(K_1; g[\cdot, \cdot])$ is a closed convex subset of $L^2(Q_1)$, it is a weakly closed set of $L^2(Q_1)$ and so $T_{L^2}(K_1; g[\cdot, \cdot])$  is sequentially weakly closed.  By a standard argument, we can show that 
$$
g_t[\cdot, \cdot]\widehat \xi_n  +  g_y[\cdot, \cdot]\widehat \zeta_n + g_u[\cdot, \cdot]\widehat v_n \rightharpoonup g_t[\cdot, \cdot]\widehat \xi  +  g_y[\cdot, \cdot]\widehat \zeta + g_u[\cdot, \cdot]\widehat v\quad\text{in}\quad L^2(Q_1).
$$ 
 Letting $n\to\infty$ in (\ref{ss11}), we obtain 
\begin{align}\label{ss12}   
        g_t[\cdot, \cdot]\widehat \xi  +  g_y[\cdot, \cdot]\widehat \zeta + g_u[\cdot, \cdot]\widehat v  \in T_{L^2}(K_1; g[\cdot, \cdot]).
\end{align}
By Lemma 2.4 in \cite{Kien-2017}, we have
\begin{align*}
    T_{L^2}(K_1; g[\cdot, \cdot])  = \{ v\in L^2(Q_1): v(x, s)\in T([a, b]; g[x, s]) \quad {\rm a.a.} \ (x, s)\in Q_1\}.
\end{align*} Combining this with \eqref{ss12}, yields  
\begin{align*}
     g_t[x, s]\xi(s) +  g_y[x, s]\zeta(x, s) + g_u[x, s)v(x, s)\in T([a, b],  g[x, s]) \quad  {\rm a.a.} \  (x, s)\in Q_1.
\end{align*} Hence condition $(b'_5)$ is valid.   Consequently, $(\widehat \xi, \widehat T, \widehat \zeta, \widehat v) \in \mathcal C'_1[(\xi_*, T_*, \zeta_*, v_*)]$.

\medskip

\noindent \textbf{Step 2.}  Proving that $(\widehat \xi, \widehat T,  \widehat \zeta, \widehat v)= (0, 0, 0, 0)$. 
\medskip

Since $F(\xi_n, T_n,  \zeta_n, v_n) = F(\xi_*, T_*,  \zeta_*, v_*) = 0$ and integration by part,  we have from \eqref{DefLF} that  
\begin{align*}
    \langle (\varphi, \phi),   F(\xi_n, T_n,  \zeta_n, v_n) - F(\xi_*, T_*, \zeta_*, v_*)\rangle = 0.
\end{align*} Since $\langle \mu, H(\xi_*, T_*, \zeta_*, v_*)\rangle=0$ and $\mu\geq 0$, we have 
$$
\langle\mu,  H(\xi_n, T_n,  \zeta_n, v_n) - H(\xi_*, T_*, \zeta_*, v_*)\rangle \le 0.
$$ Since $e\in N(K_1, g[\cdot, \cdot])$, we get 
$$
\langle e,  G(\xi_n, T_n, \zeta_n, v_n) - G(\xi_*, T_*, \zeta_*, v_*)\rangle \le 0. 
$$ From these facts and  definition of $\mathcal{L}$ in \eqref{LagrangeFunction2}, we obtain  
\begin{align} \label{ss13}
    &\mathcal L\Big((\xi_n, T_n, \zeta_n, v_n), \varphi, \phi,  \mu, e\Big)  -  \mathcal L\Big((\xi_*, T_*, \zeta_*, v_*), \varphi, \phi,  \mu, e\Big)=\notag \\
    &\widehat J(\xi_n, T_n, \zeta_n, v_n) - \widehat J(\xi_*, T_*, \zeta_*, v_*) +\langle (\varphi, \phi), F(\xi_n, T_n, \zeta_n, v_n)-F(\xi_*, T_*, \zeta_*, v_*)\rangle + \notag\\
    &\langle \mu, H(\xi_n, T_n, \zeta_n, v_n)-H(\xi_*, T_*, \zeta_*, v_*)\rangle+\langle e, G(\xi_n, T_n, \zeta_n, v_n)-G(\xi_*, T_*, \zeta_*, v_*)\rangle\notag \\
    &\leq \widehat J(\xi_n, T_n, \zeta_n, v_n) - \widehat J(\xi_*, T_*, \zeta_*, v_*) \leq o(t_n^2).
\end{align} Here the last estimate in (\ref{ss13}) follows from (\ref{ss2}).  
From (\ref{ss13}), (\ref{ss0}) and by Taylor expansions for $\mathcal{L}$, we get 
\begin{align}
    \frac{o(t_n^2)}{t_n^2}& \geq \frac{2}{t_n^2}  \Big\{ \mathcal L\Big((\xi_n, \zeta_n, v_n, T_n), \varphi, \phi,  \mu, e\Big)  -  \mathcal L\Big((\xi_*, \zeta_*, v_*, T_*), \varphi, \phi,  \mu, e\Big)  \Big\}  \nonumber  \\
    &=\psi_{0TT}(T_* + \beta_{01}(T_n -T_*), \zeta_n(1)) \widehat T_n ^2 + 2\psi_{0T\zeta}(T_*, \zeta_*(1) + \beta_{02}(\zeta_n(1) - \zeta_*(1))) \widehat T_n \widehat \zeta_n(1) \nonumber \\
    &+ \psi_{0\zeta\zeta}(T_*, \zeta_*(1) + \beta_{03}(\zeta_n(1) - \zeta_*(1)))  \widehat \zeta_n^2(1) \nonumber \\
    &+\int_{Q_1} T_n \Big[L_{tt}(x, \xi_* + \alpha_1(\xi_n - \xi_*), \zeta_n, v_n)\widehat \xi_n^2 +   L_{yy}(x, \xi_*, \zeta_* + \alpha_2(\zeta_n - \zeta_*), v_n )\widehat \zeta_n^2 \Big]  dxds  \nonumber \\
    &+ \int_{Q_1} T_n \Big[L_{uu}(x, \xi_*, \zeta_*, v_* + \alpha_3(v_n - v_*))\widehat v_n^2  + 2L_{ty}(x, \xi_*, \zeta_* + \alpha_4(\zeta_n - \zeta_*), v_n)\widehat \xi_n \widehat \zeta_n \Big]  dxds  \nonumber \\
    &+ \int_{Q_1} T_n \Big[ 2L_{tu}(x, \xi_*, \zeta_*, v_* + \alpha_5(v_n - v_*))\widehat \xi_n \widehat v_n  +  2L_{yu}(x, \xi_*, \zeta_*, v_* + \alpha_6(v_n - v_*))\widehat \zeta_n \widehat v_n\Big]  dxds  \nonumber \\
    &+   \int_{Q_1} 2 \widehat T_n \Big[L_t[x, s]\widehat \xi_n  + L_y[x, s]\widehat \zeta_n  +  L_u[x, s]\widehat v_n\Big]  dxds \nonumber \\
    &+ \int_{Q_1}\varphi T_n  \Big[ \psi_{tt}(x, \xi_* + \alpha_7(\xi_n - \xi_*), \zeta_n)\widehat \xi_n^2 +   \psi_{yy}(x, \xi_*, \zeta_* + \alpha_8(\zeta_n - \zeta_*))\widehat \zeta_n^2      \Big]dx ds   \nonumber \\
    &+ \int_{Q_1}\Big[ 2 \varphi T_n   \psi_{ty}(x, \xi_*, \zeta_* + \alpha_9(\zeta_n - \zeta_*))\widehat \xi_n \widehat \zeta_n    +  2 \varphi \widehat T_n \Big(A\widehat \zeta_n +   \psi_t[x, s]\widehat \xi_n  + \psi_y[x, s]\widehat \zeta_n   - \widehat v_n \Big)\Big] dx ds  \nonumber \\
    &+ \sum_{i = 1}^m \mu_i \psi_{iTT}(T_* + \beta_{i1}(T_n -T_*), \zeta_n(1)) \widehat T_n ^2 + \sum_{i = 1}^m \mu_i 2\psi_{iT\zeta}(T_*, \zeta_*(1) + \beta_{i2}(\zeta_n(1) - \zeta_*(1))) \widehat T_n \widehat \zeta_n(1) \nonumber \\
    &+ \sum_{i = 1}^m \mu_i \psi_{i\zeta\zeta}(T_*, \zeta_*(1) + \beta_{i3}(\zeta_n(1) - \zeta_*(1)))  \widehat \zeta_n^2(1) \nonumber \\
    &+ \int_{Q_1}e(x, s) \Big(g_{tt}(x, \xi_* + \alpha_{10}(\xi_n - \xi_*), \zeta_n, v_n)\widehat \xi_n^2 +   g_{yy}(x, \xi_*, \zeta_* + \alpha_{11}(\zeta_n - \zeta_*), v_n )\widehat \zeta_n^2 \Big)  dxds  \nonumber \\
    &+ \int_{Q_1}e(x, s) \Big(g_{uu}(x, \xi_*, \zeta_*, v_* + \alpha_{12}(v_n - v_*))\widehat v_n^2  + 2g_{ty}(x, \xi_*, \zeta_* + \alpha_{13}(\zeta_n - \zeta_*), v_n)\widehat \xi_n \widehat \zeta_n \Big)  dxds  \nonumber \\
    &+ \int_{Q_1}e(x, s) \Big( 2g_{tu}(x, \xi_*, \zeta_*, v_* + \alpha_{14}(v_n - v_*))\widehat \xi_n \widehat v_n  +  2g_{yu}(x, \xi_*, \zeta_*, v_* + \alpha_{15}(v_n - v_*))\widehat \zeta_n \widehat v_n\Big)  dxds \notag\\
    &=:  \Sigma_n +   \int_{Q_1} \big[T_nL_{uu}(x, \xi_*, \zeta_*, v_* + \alpha_3(v_n - v_*))\widehat v_n^2 + e(x, s) g_{uu}(x, \xi_*, \zeta_*, v_* + \alpha_{12}(v_n - v_*))\widehat v_n^2\big]dxds,   \label{ss13.1}
\end{align}
where $0 \le \alpha_i, \beta_{kj} \le 1$,  $i = 1, 2,..., 15$, $k = 0,  1, 2, ..., m$, $j = 1, 2, 3$, and 
\begin{align}
\Sigma_n&=\psi_{0TT}(T_* + \beta_{01}(T_n -T_*), \zeta_n(1)) \widehat T_n ^2 + 2\psi_{0T\zeta}(T_*, \zeta_*(1) + \beta_{02}(\zeta_n(1) - \zeta_*(1))) \widehat T_n \widehat \zeta_n(1) \nonumber \\
    &+ \psi_{0\zeta\zeta}(T_*, \zeta_*(1) + \beta_{03}(\zeta_n(1) - \zeta_*(1)))  \widehat \zeta_n^2(1) \nonumber \\
    &+\int_{Q_1} T_n \Big[L_{tt}(x, \xi_* + \alpha_1(\xi_n - \xi_*), \zeta_n, v_n)\widehat \xi_n^2 +   L_{yy}(x, \xi_*, \zeta_* + \alpha_2(\zeta_n - \zeta_*), v_n )\widehat \zeta_n^2 \Big]  dxds  \nonumber \\
    &+ \int_{Q_1} T_n \Big[2L_{ty}(x, \xi_*, \zeta_* + \alpha_4(\zeta_n - \zeta_*), v_n)\widehat \xi_n \widehat \zeta_n \Big]  dxds  \nonumber \\
    &+ \int_{Q_1} T_n \Big[ 2L_{tu}(x, \xi_*, \zeta_*, v_* + \alpha_5(v_n - v_*))\widehat \xi_n \widehat v_n  +  2L_{yu}(x, \xi_*, \zeta_*, v_* + \alpha_6(v_n - v_*))\widehat \zeta_n \widehat v_n\Big]  dxds  \nonumber \\
    &+ \int_{Q_1} 2 \widehat T_n \Big[L_t[x, s]\widehat \xi_n  + L_y[x, s]\widehat \zeta_n  +  L_u[x, s]\widehat v_n\Big]  dxds \nonumber \\
     &+ \int_{Q_1}\varphi T_n  \Big[ \psi_{tt}(x, \xi_* + \alpha_7(\xi_n - \xi_*), \zeta_n)\widehat \xi_n^2 +  \psi_{yy}(x, \xi_*, \zeta_* + \alpha_8(\zeta_n - \zeta_*))\widehat \zeta_n^2      \Big]dx ds   \nonumber \\
    &+ \int_{Q_1}\Big[ 2 \varphi T_n   \psi_{ty}(x, \xi_*, \zeta_* + \alpha_9(\zeta_n - \zeta_*))\widehat \xi_n \widehat \zeta_n    +  2 \varphi \widehat T_n \Big(A\widehat \zeta_n +   \psi_t[x, s]\widehat \xi_n  + \psi_y[x, s]\widehat \zeta_n   - \widehat v_n \Big)\Big] dx ds  \nonumber \\
    &+ \sum_{i = 1}^m \mu_i \psi_{iTT}(T_* + \beta_{i1}(T_n -T_*), \zeta_n(1)) \widehat T_n ^2 + \sum_{i = 1}^m \mu_i 2\psi_{iT\zeta}(T_*, \zeta_*(1) + \beta_{i2}(\zeta_n(1) - \zeta_*(1))) \widehat T_n \widehat \zeta_n(1) \nonumber \\
    &+ \sum_{i = 1}^m \mu_i \psi_{i\zeta\zeta}(T_*, \zeta_*(1) + \beta_{i3}(\zeta_n(1) - \zeta_*(1)))  \widehat \zeta_n^2(1) \nonumber \\
    &+ \int_{Q_1}e(x, s) \Big(g_{tt}(x, \xi_* + \alpha_{10}(\xi_n - \xi_*), \zeta_n, v_n)\widehat \xi_n^2 +   g_{yy}(x, \xi_*, \zeta_* + \alpha_{11}(\zeta_n - \zeta_*), v_n )\widehat \zeta_n^2 \Big)  dxds  \nonumber \\
    &+ \int_{Q_1}e(x, s) \Big(2g_{ty}(x, \xi_*, \zeta_* + \alpha_{13}(\zeta_n - \zeta_*), v_n)\widehat \xi_n \widehat \zeta_n \Big)  dxds  \nonumber \\
    &+ \int_{Q_1}e(x, s) \Big( 2g_{tu}(x, \xi_*, \zeta_*, v_* + \alpha_{14}(v_n - v_*))\widehat \xi_n \widehat v_n  +  2g_{yu}(x, \xi_*, \zeta_*, v_* + \alpha_{15}(v_n - v_*))\widehat \zeta_n \widehat v_n\Big)  dxds.  \nonumber
\end{align}
It follows from (\ref{ss13.1}) that 
\begin{align}\label{Convergence0}
0\geq \lim_{n\to\infty}\Big[\Sigma_n +  \int_{Q_1}\big[ T_n L_{uu}(x, \xi_*, \zeta_*, v_* + \alpha_3(v_n - v_*))\widehat v_n^2 + e(x, s)g_{uu}(x, \xi_*, \zeta_*, v_* + \alpha_{12}(v_n - v_*))\widehat v_n^2  \big ]dxds\Big].
\end{align}        
Recall that $\|\xi_n- \xi_*\|_{C([0, 1], \mathbb R)}\to 0$, $\|\zeta_n- \zeta_*\|_{L^\infty(Q_1)}\to 0$, $\|v_n- v_*\|_{L^\infty(Q_1)}\to 0$, $|T_n- T_*|\to 0$, $\|\widehat\zeta_n-\widehat \zeta\|_{L^2(Q_1)}\to 0$, $|\widehat T_n -\widehat T|\to 0$, $\widehat v_n \rightharpoonup \widehat v$ in $L^2(Q_1)$ and $\|\widehat\zeta_n(1)-\widehat\zeta(1)\|_{L^2(\Omega)}\to 0$.  
Combining these facts with  $(H4)$ and $(H7)$,  we get 
\begin{align}\label{limitSigma}
 \lim_{n\to\infty}\Sigma_n&=\psi_{0TT}[T_*,1] \widehat T ^2 + 2\psi_{0T\zeta}[T_*,1] \widehat T \widehat \zeta(1)+ \psi_{0\zeta\zeta}[T_*,1]\widehat \zeta^2(1) \nonumber \\
    &+\int_{Q_1} T_* \Big[L_{tt}[x, s]\widehat \xi^2 +   L_{yy}[x,s]\widehat \zeta^2 \Big]  dxds  + \int_{Q_1} T_* \Big[2L_{ty}[x, s]\widehat \xi \widehat \zeta \Big]  dxds  \nonumber \\
    &+ \int_{Q_1} T_* \Big[ 2L_{tu}[x,s]\widehat \xi \widehat v  +  2L_{yu}[x, s]\widehat \zeta \widehat v\Big]  dxds +   \int_{Q_1} 2 \widehat T \Big[L_t[x, s]\widehat \xi  + L_y[x, s]\widehat \zeta  +  L_u[x, s]\widehat v\Big]  dxds \nonumber \\
    &+  \int_{Q_1}\varphi \Big[T_*\Big(\psi_{tt}[x, s]\xi^2  +  \psi_{yy}[x, s]\zeta^2  + 2\psi_{ty}[x, s]\xi \zeta\Big)    + 2 T \Big(A \zeta +  \psi_t[x, s]\xi  + \psi_y[x, s]\zeta - v\Big)\Big]dxds \nonumber \\
    &+ \sum_{i = 1}^m \big(\mu_i \psi_{iTT}[T_*,1] \widehat T^2 + \mu_i 2\psi_{iT\zeta}[T_*, 1] \widehat T \widehat \zeta(1)+  \mu_i \psi_{i\zeta\zeta}[T_*,1] \widehat \zeta^2(1) \big)\nonumber \\
    &+ \int_{Q_1}e(x, s)\big(g_{tt}[x,s]\widehat \xi^2 +   g_{yy}[x,s]\widehat \zeta^2 +2g_{ty}[x,s]\widehat \xi \widehat \zeta +2g_{tu}[x,s]\widehat \xi \widehat v  +  2g_{yu}[x,s]\widehat \zeta \widehat v\big)  dxds\notag\\
    &=: \Sigma.   
\end{align} We now deal with the second term in \eqref{Convergence0}. By using  assumption $(vi)$, we see that the function 
$$
L^2(Q_1)\ni v\mapsto \int_{Q_1}\big[T_* L_{uu}[x,s] + e(x,s)g_{uu}[x,s] \big] v^2(x,s)dxds
$$ is convex and so it is sequentially lower semicontinuous. Hence 
$$
\liminf_{n\to\infty} \int_{Q_1}\big[T_* L_{uu}[x,s] + e(x,s)g_{uu}[x,s] \big] \widehat v_n^2(x,s)dxds\geq \int_{Q_1}\big[T_* L_{uu}[x,s] + e(x,s)g_{uu}[x,s] \big] \widehat v^2(x,s)dxds.
$$ This implies that 
\begin{align*}
    &\liminf_{n \to \infty} \int_{Q_1}\Big(T_n L_{uu}(x, \xi_*, \zeta_*, v_* + \alpha_3(v_n - v_*))  +  eg_{uu}(x, \xi_*, \zeta_*, v_* + \alpha_{12}(v_n - v_*))\Big)\widehat v^2_n  dxds  \\
    &\geq  \liminf_{n \to \infty} \int_{Q_1}\Big(T_n L_{uu}(x, \xi_*, \zeta_*, v_* + \alpha_3(v_n - v_*))  +  eg_{uu}(x, \xi_*, \zeta_*, v_* + \alpha_{12}(v_n - v_*)) \\  
    &\quad\quad\quad - T_* L_{uu}[x, s] - eg_{uu}[x, s]\Big)\widehat v^2_n  dxds + \liminf_{n \to \infty} \int_{Q_1}\Big(T_* L_{uu}[x, s] + eg_{uu}[x, s]\Big)\widehat v^2_n  dxds\\  
    &\ge  \int_{Q_1}\Big(T_* L_{uu}[x, s] + eg_{uu}[x, s]\Big)\widehat v^2  dxds.
\end{align*} Combining this with \eqref{limitSigma} and \eqref{Convergence0} we have  
\begin{align*}
&0\geq  \liminf_{n\to\infty}\Sigma_n\\
&+ \liminf_{n\to\infty}\Big[\int_{Q_1}\big[ T_n L_{uu}(x, \xi_*, \zeta_*, v_* + \alpha_3(v_n - v_*))\widehat v_n^2 + e(x, s)g_{uu}(x, \xi_*, \zeta_*, v_* + \alpha_{12}(v_n - v_*))\widehat v_n^2  \big ]dxds\Big]\\
&\geq \Sigma + \int_{Q_1}\Big(T_* L_{uu}[x, s] + eg_{uu}[x, s]\Big)\widehat v^2  dxds\\
&= D^2_{zz}\mathcal{L}(z_*, \varphi, \mu, \phi,  e)\Big[(\widehat \xi,  \widehat T, \widehat \zeta, \widehat v),(\widehat \xi, , \widehat T, \widehat \zeta, \widehat v)\Big]. 
\end{align*} From this and strictly second-order condition (\ref{ss0.1}), we conclude that $(\widehat \xi, \widehat T, \widehat \zeta, \widehat v)= (0, 0, 0, 0)$.

\medskip

\noindent \textbf{Step 3.}   Showing a contradiction.  
\medskip

From Step 2, we have $\widehat T_n\to 0$,  $|\widehat T_n|^2\to 0$, $\widehat\xi_n\to 0$ in $C([0,1], \mathbb{R})$ and  $\widehat\zeta_n\to 0$ in $L^2(Q_1)$. Hence $\lim_{n\to\infty}\Sigma_n=\Sigma=0$. Note that  $|\widehat T_n|^2 +   \|\widehat v_n\|^2_{L^2(Q_1)} = 1$. Using these facts and \eqref{Legendre-Clebsch-condition}, we have from \eqref{Convergence0} that
\begin{align*}
&0\geq \liminf_{n\to\infty}\Sigma_n\\
&+\liminf_{n\to\infty}\int_{Q_1}\big[ T_n L_{uu}(x, \xi_*, \zeta_*, v_* + \alpha_3(v_n - v_*)) + e(x, s)g_{uu}(x, \xi_*, \zeta_*, v_* + \alpha_{12}(v_n - v_*))\big]\widehat v_n^2 dxds\\
&\geq 0+\liminf_{n\to\infty} \int_{Q_1}\big[ T_n L_{uu}(x, \xi_*, \zeta_*, v_* + \alpha_3(v_n - v_*)) + e(x, s)g_{uu}(x, \xi_*, \zeta_*, v_* + \alpha_{12}(v_n - v_*))\big]\widehat v_n^2dxds\\
&\geq \liminf_{n\to\infty}  \int_{Q_1}\Big[ T_n L_{uu}(x, \xi_*, \zeta_*, v_* + \alpha_3(v_n - v_*)) + e(x, s)g_{uu}(x, \xi_*, \zeta_*, v_* + \alpha_{12}(v_n - v_*)) \\
&\quad\quad\quad\quad- T_* L_{uu}[x,s] -e(x, s)g_{uu}[x,s] \Big]\widehat v_n^2 dxds\\
&+\liminf_{n\to\infty}\int_{Q_1}\big[T_* L_{uu}[x,s] -e(x, s)g_{uu}[x,s] \big]\widehat v_n^2 dxds\\
&=0+ \liminf_{n\to\infty}\int_{Q_1}\big[T_* L_{uu}[x,s] -e(x, s)g_{uu}[x,s] \big]\widehat v_n^2 dxds\notag\\
&\geq \liminf_{n\to\infty}\int_{Q_1}\Lambda v_n^2 dxds\\
&=\liminf_{n\to\infty}\Big(\int_{Q_1}\Lambda v_n^2 dxds +\Lambda |\widehat T_n|^2\Big)=\liminf_{n\to\infty}\Lambda\big(\|\widehat v_n\|^2_{L^2(Q_1)}+|\widehat T_n|^2\big)\geq\Lambda
\end{align*} which is absurd. The proof of proposition is complete.  
\end{proof}

\section{Proof of  main results}

\subsection{Proof of Theorem 3.1}

We first prove the following lemma.  

\begin{lemma}\label{lemma5.1} If $(T,  y, u) \in \mathcal{C}[(T_*, y_*, u_*)]$, then   $(\xi, T, \zeta, v) \in \mathcal{C}_1[(\xi_*, T_*, \zeta_*, v_*)]$, where $\xi(s)=Ts$, $\zeta(x, s) = y(x, T_*s)$, $v(x, s) = u(x, T_*s)$. In particular, if  $(T,  y, u) \in \mathcal{C}_0[(T_*, y_*, u_*)]$, then   $(\xi, T, \zeta, v) \in \mathcal{C}_{1,0}[(\xi_*, T_*, \zeta_*, v_*)]$.  
\end{lemma}
\begin{proof} Assume that $(T,  y, u) \in \mathcal{C}_0[(T_*, y_*, u_*)]$.   By definition of $\mathcal C[(T_*, y_*, u_*)]$, there exists a sequence $\{(T_k, y_k, u_k)\} \subset \mathcal{C}_0[(T_*, y_*, u_*)]$ converging to $(T,  y, u)$. Since $\{(T_k, y_k, u_k)\} \subset \mathcal{C}_0[(T_*, y_*, u_*)]$, it satisfies the following conditions:
\begin{itemize}
      \item [$(c_1)$] $\psi_{0T}[T_*, T_*]T_k +\psi_{0\zeta}[T_*,T_*]y_k(T_*)  +    \displaystyle \int_{Q_{T_*}}\Big( \frac{T_k}{T_*} L[x, t] +L_t[x, t]\frac{T_kt}{T_*}+ L_y[x, t]y_k + L_u[x, t]u_k  \Big) dxdt \leq 0$;

    \item [$(c_2)$]   $\dfrac{\partial y_k}{\partial t} +  Ay_k +  \psi_t[x, t] \dfrac{T_kt}{T_*}+  \psi_y[x, t]y_k - u_k  + \dfrac{T_k}{T_*}(Ay_* +\psi[x, t]  -u_*) =  0, \quad    y(0)=0$;

    \item [$(c_3)$] $\psi_{iT}[T_*, T_*]T_k +\psi_{i\zeta}[T_*, T_*]y_k(T_*) \le 0$  \ \  for $i\in \big\{1,2,.., m| \psi_i[T_*, T_*]=0\big\}$;

    \item [$(c_4)$] $g_t[\cdot, \cdot] \frac{T_kt}{T_*}+ g_y[\cdot, \cdot]y_k + g_u[\cdot, \cdot]u_k  \in {\rm cone}\big(K_{T_*} -g[\cdot, \cdot]\big)$.
\end{itemize}
    We define $\xi_k=s T_k$, $\zeta_k(x, s)= y_k(x, \xi_*(s))$ and $v_k(x, s) = u_k(x, \xi_*(s))$. By changing  variable $t=\xi_*(s)=T_*s$,  we obtain
\begin{itemize}
       \item [$(b_1)$] $\psi_{0T}[T_*,1]T_k  + \psi_{0\zeta}[T_*, 1]\zeta_k(1)+ \displaystyle \int_{Q_1}\Big(T_k L[x, s] +  T_*\big(L_t[x, s]\xi_k  +  L_y[x, s]\zeta_k + L_u[x, s]v_k\big)\Big) dxds \leq 0$;

    \item [$(b_2)$] $\dfrac{\partial \zeta_k}{\partial s} +T_*\big( A\zeta_k + \psi_t[\cdot, \cdot]\xi_k  +  \psi_y[\cdot, \cdot]\zeta_k -v_k\big)+ T_k(A\zeta_* +\psi[\cdot, \cdot]-v_*)=0, \quad    \zeta(0)=0$;

    \item [$(b_3)$] $\xi_k(s)=T_ks$ for all $s\in [0,1]$;

    \item [$(b_4)$] $\psi_{i\xi}(\xi_*(1), \zeta_*(1))\xi_k(1)+\psi_{i\zeta}(\xi_*(1), \zeta_*(1))\zeta_k(1)\leq 0$  \ \  for $i\in \{1,2,.., m| \psi_i(\xi_*(1), \zeta_*(1))=0\}$;

    \item [$(b_5)$] $ g_t[\cdot, \cdot]\xi_k  +  g_y[\cdot, \cdot]\zeta_k + g_u[\cdot, \cdot]v_k \in{\rm cone}\big(K_1-g[\cdot, \cdot]\big)$.   
\end{itemize}
    This implies that $\{(\xi_k, T_k, \zeta_k, v_k)\}\subset \mathcal C_{1,0}[(\xi_*, T_*,  \zeta_*, v_*)]$. It is clear that $(\xi_k, T_k, \zeta_k, v_k)\to (\xi, T,  \zeta, v)$. Hence  $(\xi, T,  \zeta, v)\in  \mathcal C_1[(\xi_*, T_*, \zeta_*, v_*)]$. The lemma is proved. 
\end{proof}

\noindent \textbf{Proof of Theorem \ref{main-theorem}}.  
Let $ (T_*, y_*, u_*)\in\Phi$  be  a locally optimal solution to $(P)$. As before, we define  
\begin{align*}
 \xi_*(s)=T_*s, \quad    \zeta_* (x, s)= y_*(x, \xi_*(s)), \quad  v_*(x, s) = u_*(x, \xi_*(s)).  
\end{align*}
By Proposition \ref{relationPandP1}, vector $(\xi_*, T_*,  \zeta_*, v_*)$ is a locally optimal solution to $(P_1)$ and $\widehat J (\xi_*, T_*, \zeta_*, v_*) = J (T_*, y_*, u_*)$.  Let $d=( T,  y, u) \in \mathcal{C}_0[(T_*, y_*, u_*)]$.  By Lemma \ref{lemma5.1}, vector $z:= (\xi, T, \zeta, v) \in \mathcal{C}_{1,0} [(\xi_*, T_*, \zeta_*, v_*)]$, where $\xi(s)=Ts$,  $\zeta(x, s) = y(x, \xi_*(s))$, $v(x, s) = u(x, \xi_*(s))$.  According to Proposition \ref{Pro-KKT-P1}, there exist Lagrange multipliers $\lambda \in \mathbb R_+$, $ \mu = (\mu_1, \mu_2, ...,  \mu_m) \in \mathbb R^m$, $\varphi \in L^q(Q_1)  \cap L^1(0, 1; W^{1, 1}_0(\Omega))$, $ e \in L^1(Q_1)$ and an absolutely continuous function $ \phi : [0, 1] \to \mathbb{R}$ satisfying the conditions (\ref{cm0.1})-(\ref{cm0.6}). Define 
\begin{align*}
    \widetilde\varphi(x, t) = \varphi (x, \frac{t}{T_*}), \quad   \widetilde\phi(t)=   \phi (\frac{t}{T_*}), \quad  \widetilde e(x, t)= \frac{1}{T_*}e(x, \frac{t}{T_*}). 
\end{align*} By replacing  $s=\xi_*^{-1}(t)=\frac{t}{T_*}$ into \eqref{cm0.1}-\eqref{cm0.5}, we  can show  that $(\lambda, \mu, \widetilde\varphi, \widetilde\phi, \widetilde e)$ satisfies conditions $(i)-(iv)$  of Theorem 3.1. Also, by changing variable $s=\frac{t}{T_*}$ in \eqref{cm0.6}, we obtain \eqref{NSOC-P}. 
 
The last assertion of Theorem \ref{main-theorem} follows from the last assertion of Proposition \ref{Pro-KKT-P1} and Lemma \ref{lemma-Robinson}. The proof of Theorem \ref{main-theorem} is complete.  $\hfill\square$

\subsection{Proof of Theorem 3.2}

\begin{lemma}\label{Lemma-RelCrCone} Suppose $T_*>0$, $(\xi_*, T_*, \zeta_*, v_*)\in\Phi_1$ and $(\xi, T,  \zeta, v) \in \mathcal{C}'_1[(\xi_*, T_*, \zeta_*, v_*)]$. Then  $(T, y, u) \in \mathcal{C}'[(T_*, y_*, u_*)]$, where $y(x, t):=\zeta(x, \frac{t}{T_*}), u(x, t):=v(x, \frac{t}{T_*}), y_*(x, t):=\zeta_*(x, \frac{t}{T_*}) $ and $u_* := v_*(x, \frac{t}{T_*})$. 
\end{lemma}
\begin{proof}  By definition of $\Phi_1$, $(\xi_*, T_*, \zeta_*, v_*)\in C([0,1], \mathbb{R})\times\mathbb{R}\times Y_1\times U_1$ and satisfies the following constraints:
\begin{align*}
    &\frac{\partial \zeta_*}{\partial s} + T_*A\zeta_* + T_*\psi(x, \xi_*, \zeta_*) = T_*v_* \quad \text{in} \ Q_1, \quad  \zeta_*(x, s)=0 \quad  \text{on} \ \Sigma_1 = \Gamma\times [0, 1], \\
    &\zeta_*(0)=y_0 \quad \text{in} \ \Omega, \\
    &\xi_*(s) =T_*s\  \text{for all}\ s\in [0,1], \\
    &\psi_i(\xi_*(1), \zeta_*(1)) \le 0,   \quad  i=1,2,..., m,\\
    &a\leq g(x, \xi_*(s), \zeta_*(x, s), v_*(x, s)) \le b \quad {\rm a.a.} \ x \in \Omega, \ \forall s  \in [0, 1]. 
\end{align*} By definition of $y_*$, we have 
$$
\frac{d}{dt}y_*(x, t)=\frac{d}{ds}\zeta_*(x, \frac{t}{T_*})\frac{1}{T_*}.
$$ Changing variable $s=\frac{t}{T_*}$ in the above constraints, we obtain 
\begin{align*}
    &\frac{\partial y_*}{\partial t} + Ay_* + \psi(x, t, y_*) = u_* \quad \text{in} \ Q_{T_*}, \quad  y_*(x, t)=0 \quad  \text{on}\  \Sigma_{T_*} = \Gamma\times [0, T_*], \\
    &y_*(0)=y_0 \quad \text{in}\ \Omega, \\
    &t\in [0, T_*], \\
    &\psi_i(T_*, y_*(T_*)) \le 0,   \quad  i=1,2,..., m,\\
    &a\leq g(x, t, y_*(x, t), u_*(x, t)) \le b \quad {\rm a.a.} \ x \in \Omega, \  t  \in [0, T_*]. 
\end{align*} This implies that $(T_*, y_*, u_*)\in\Phi$. We now assume that  $(\xi, T,  \zeta, v) \in \mathcal{C}'_1[(\xi_*, T_*, \zeta_*, v_*)]$. By definition,  $(\xi, T, \zeta, v) \in C([0, 1], \mathbb R)\times\mathbb{R} \times W^{1, 1}_2(0, 1; D, H) \times L^2(Q_1)$ and satisfies conditions $(b_1')-(b_5')$. Namely, we have 
\begin{itemize}
    \item [$(b_1)'$] $\psi_{0T}[T_*, 1] T +\psi_{0\zeta}[T_*,1] \zeta(1) + \displaystyle \int_{Q_1}\Big(T L[x, s] +  T_*[L_t[x, s]\xi(s)  +  L_y[x, s]\zeta + L_u[x, s]v]\Big) dxds \leq 0$;
    
    \item [$(b_2)'$] $\dfrac{\partial \zeta}{\partial s} +T_*\big( A\zeta + \psi_t[\cdot, \cdot]\xi  +  \psi_y[\cdot, \cdot]\zeta -v\big)+ T(A\zeta_* +\psi[\cdot, \cdot]-v_*)=0, \quad    \zeta(0)=0$;

    \item [$(b_3)'$] $\xi(s) =   Ts \quad \forall s \in [0, 1]$;

    \item [$(b_4)'$]$\psi_{i\xi}[T_*, 1] T+\psi_{i\zeta}[T_*, 1]\zeta(1)\leq 0$\  for\ $i\in \big\{1,2,.., m| \psi_i[T_*, 1]=0\big \}$;
      
    \item [$(b_5)'$] $g_t[x, s]\xi(s) +  g_y[x, s]\zeta(x, s) + g_u[x, s)v(x, s)\in T([a, b],  g[x, s])$ for a.a. $(x, s)\in Q_1$.
\end{itemize} Changing variable  $s=\frac{t}{T_*}$ with $t\in[0, T_*]$ and putting $y(x, t)=\zeta(x, \frac{t}{T_*})$ and $u=v(x, \frac{t}{T_*})$, we obtain 
\begin{itemize}
\item[$(c_1)'$] $\psi_{0T}[T_*, T_*]T +\psi_{0y}[T_*,T_*]y(T_*)  +   \displaystyle \int_{Q_{T_*}}\Big(\frac{T}{T_*} L[x, t] +  [L_t[x, t]\frac{Tt}{T_*}  +  L_y[x, t]y + L_u[x, t]u]\Big) dxdt \leq 0$;
 \item [$(c_2')$] $\dfrac{\partial y}{\partial t} +  Ay +  \psi_t[x, t] \dfrac{Tt}{T_*}+  \psi_y[x, t]y - u  + \dfrac{T}{T_*}(Ay_* +\psi[x, t]  -u_*) =  0, \quad    y(0)=0$;
 \item [$(c_3)'$]$\psi_{i\xi}[T_*, T_*]T+\psi_{i\zeta}[T_*,T_*]y(T_*)\leq 0$\  for $i\in \big\{1,2,.., m| \psi_i[T_*, T_*]=0\big\}$;

\item [$(c_4)'$] $g_t[x, t]\frac{Tt}{T_*} +  g_y[x, t] y(x, t) + g_u[x, t)u(x, t)\in T([a, b],  g[x, t])$ for a.a. $(x, t)\in Q_{T_*}$. 
\end{itemize}  Hence $(T, y, u)\in\mathcal{C}'[(T_*, y_*, u_*)]$. The lemma is proved. 
\end{proof}

\noindent{\bf Proof of Theorem \ref{main-theorem+}}.   Let $(T_*, y_*, u_*) \in \Phi$.  We define 
\begin{align}
\label{22.1}
\xi_*(s)=T_*s, \quad \zeta_*(x,s)=y_*(x, T_*s), \quad   v_*(x,s)=u_*(x, T_*s) 
\end{align} 
for all $s\in [0,1]$.  By a simple argument, we can show that $(\xi_*, T_*, \zeta_*, v_*)\in\Phi_1$. Besides, we also have $y_*(x, t)=\zeta_*(x, \frac{t}{T_*})$ and $u_*(x, t)=v_*(x, \frac{t}{T_*})$ for all $t\in [0, T_*]$. Let $(1,\mu, \widetilde{\varphi}, \widetilde{\phi}, \widetilde{e})$ be  multipliers which satisfy conditions $(i)-(iv)$ of Theorem \ref{main-theorem} and define 
\begin{align}
\label{22.2}
\varphi(x, s)=\widetilde{\varphi}(x, T_*s), \quad  \phi(s)=   \widetilde{\phi}(T_*s), \quad  e(x, s)=T_*\widetilde{e}(x, T_*s) 
\end{align}
for all $s\in [0,1]$. Then by changing variable $t = \xi_*(s) = T_*s$ in $(i)-(iv)$ of Theorem \ref{main-theorem}, it's easy to check that $(1, \mu, \varphi, \phi, e)$ satisfies conclusions $(i)-(iv)$ of Proposition \ref{Pro-KKT-P1}. 
Now we take any $(\xi, T, \zeta, v)\in \mathcal{C}'_1[(\xi_*, T_*, \zeta_*, v_*)]\setminus\{(0,0,0,0)\}$. Define
\begin{align}
\label{22.3}
y(x, t)=\zeta(x, \frac{t}{T_*}),\ u(x, t)=v(x, \frac{t}{T_*})
\end{align}
for $t \in [0, T_*]$.  By Lemma \ref{Lemma-RelCrCone}, we have $(T, y, u)\in \mathcal{C}'[(T_*, y_*, u_*)]\setminus\{0,0, 0\}$. It follows from this and \eqref{SSOC-P} that
\begin{align}  
&\psi_{0TT}(T_*, y_*(T_*))T^2 + 2\psi_{0 T y}(T_*, y_*(T_*))T y(T_*) +\psi_{0 yy}(T_*, y_*(T_*))y(T_*)^2 \notag\\
&+  \int_{Q_{T_*}}\Big( L_{tt}[x, t](\frac{Tt}{T_*})^2   + 2L_{ty}[x, s]\frac{Tt}{T_*} y(x,t)+  2L_{tu}[x, t]\frac{Tt}{T_*}u(x,t) \Big) dxdt \notag\\
&+ \int_{Q_{T_*}}\Big( L_{yy}[x, t]y(x,t)^2 + L_{uu}[x, t]u(x,t)^2  +  2L_{yu}[x, t]y(x,t)u(x,t)\Big) dxdt  \nonumber \\
 &+ 2  \int_{Q_{T_*}}\frac{T}{T_*} \Big(L_t[x, t]\frac{Tt}{T_*}  + L_y[x, t]y(x,t)  +  L_u[x, t]u(x,t)\Big)  dxdt\notag\\
 &+  \int_{Q_{T_*}}\widetilde\varphi(x,t)\Big(\psi_{tt}[x, t](\frac{Tt}{T_*})^2   + 2\psi_{ty}[x, s]\frac{Tt}{T_*} y(x,t)+ \psi_{yy}[x, t]y(x,t)^2 \Big)  dxdt   \nonumber \\
 &+  \int_{Q_{T_*}}  2 \frac{T}{T_*}\widetilde\varphi(x,t) \Big(A y + \psi_t[x, t]\frac{Tt}{T_*}  + \psi_y[x, t]y - u\Big)dxdt \nonumber \\
 &+  \int_{Q_{T_*}} \widetilde e(x, s) \big(g_{tt}[x, t](\frac{Tt}{T_*})^2  +  g_{yy}[x, t]y^2 + g_{uu}[x, t]u^2  + 2g_{ty}[x, t]\frac{Tt}{T_*} y + 2g_{tu}[x, t]\frac{Tt}{T_*} u  +  2g_{yu}[x, t]yu\big)  dxdt  \nonumber \\
 &+  \sum_{i = 1}^m \mu_i \Big[\psi_{iTT}(T_*, y_*(T_*))T^2 + 2\psi_{i T \zeta}(T_*, y_*(T_*))T y(T_*) +\psi_{i\zeta \zeta}(T_*, y_*(T_*))y(T_*)^2\Big]   \geq 0. \nonumber 
     \end{align}
 By changing variable $t=T_*s$ and noting that $\xi(s)=Ts$, (\ref{22.1}), (\ref{22.2}) and (\ref{22.3}), we obtain that the condition $(v)$ of Proposition \ref{Pro-SOSC-P1} is satisfied. 
On the other hand, changing variable  $t=T_*s$ in \eqref{L-C-condition-P} and noting that $e(x, s)=T_*\widetilde{e}(x, T_*s)$,  we obtain  
\begin{align*}
L_{uu}[x, s] + \frac{1}{T_*} e(x, s)g_{uu}[x, s]\geq \Lambda_0 \quad {\rm a.a.} \ (x, s) \in Q_1. 
\end{align*}
This implies that   the condition $(vi)$ of Proposition \ref{Pro-SOSC-P1} is satisfied with $\Lambda = T_* \Lambda_0$.
 Thus $(\xi_*, T_*, \zeta_*, v_*)\in\Phi_1$ and $(1, \mu, \varphi, \phi, e)\in\Lambda[(\xi_*, T_*, \zeta_*, v_*)]$ which satisfy all sufficient conditions of Proposition \ref{Pro-SOSC-P1}. Accordingly, $(\xi_*, T_*, \zeta_*, v_*)$ is a locally optimal solution to $(P_1)$. Since $y_*(x, t)=\zeta_*(x, \frac{t}{T_*})$ and $u_*(x,t)=v_*(x, \frac{t}{T_*})$ for $t \in [0, T_*]$, Proposition \ref{relationPandP1+} implies that $(T_*, y_*, u_*)$ is a locally optimal solution to $(P)$. Moreover, according to the conclusion of Proposition \ref{Pro-SOSC-P1},  there exist numbers $\delta > 0$ and $\kappa > 0$ such that 
    \begin{align}\label{GrowthCond}
        \widehat J (z) \ge \widehat J (z_*) + \kappa\Big( (T - T_*)^2  + \|v - v_*\|^2_{L^2(Q_1)}  \Big) 
    \end{align} for all $z = (\xi, T, \zeta, v) \in \Phi_1 \cap B_Z(z_*, \delta)$. According to the proof of Proposition \ref{relationPandP1+}, there exists $\epsilon_0  \in (0, \frac{\delta}6)$ such that, if $(T, y, u)\in\Phi$ and 
    $$
    {\rm dist}[(T, y, u), (T_*, y_*, u_*)] \leq \epsilon_0, 
    $$ then $(\xi, T, \zeta, v)\in \Phi_1\cap B_Z(z_*, \delta)$ with $\xi(s)=Ts$, $\zeta(x, s)=y(x, Ts)$ and $v(x, s)=u(x, Ts)$ for $s\in [0,1]$. Therefore, for all $(T, y, u)\in\Phi\cap N\big((T_*, y_*, u_*), \epsilon_0\big)$,  we have from \eqref{GrowthCond} that
     \begin{align*}
       J(T, y, u)= \widehat J (z) &\geq \widehat J (z_*) + \kappa\Big((T - T_*)^2   +    \|v - v_*\|^2_{L^2(Q_1)}   \Big)\\
       &=J (T_*, y_*, u_*) + \kappa\Big((T - T_*)^2   +    \int_{Q_1}|u(x, Ts) - u_*(x, T_*s)|^2 dxds   \Big)\\
       &=J (T_*, y_*, u_*) + \kappa\Big((T - T_*)^2  +    \frac{1}{T_*}\int_{Q_{T_*}}|u(x, \frac{Tt}{T_*}) - u_*(x, t)|^2 dxdt  \Big)
    \end{align*} which is the conclusion of Theorem \ref{main-theorem+}.  The proof of Theorem \ref{main-theorem+} is complete. $\hfill\square$

    \medskip

\noindent {\bf Acknowledgments.} This research  was  supported by Vietnam Academy of Science and Technology under grant number CTTH00.01/25-26. The second author would like to thank the Alexander Von Humboldt Foundation for their support during  his visit at the Faculty of Mathematics of the  University of Duisburg-Essen.

\end{document}